\documentclass[a4paper,12pt]{amsart}

\usepackage{fullpage}
\usepackage{amsmath,amssymb,graphicx,psfrag}
\usepackage[T1]{fontenc}
\usepackage{amsmath,amssymb,ifthen}
\usepackage{color}
\usepackage{moreverb}
\usepackage{enumitem}
\usepackage{pgfplots}
\usepackage{adjustbox}

\renewcommand{\H}{\widetilde{H}}

\newcommand{\N}{\mathbb{N}}
\newcommand{\R}{\mathbb{R}}

\newcommand{\Approx}{\mathbb{A}}
\newcommand{\HH}{\mathcal{H}}

\newcommand{\MM}{\mathcal{M}}
\newcommand{\OO}{\mathcal{O}}
\newcommand{\QQ}{\mathcal{Q}}
\newcommand{\RR}{\mathcal{R}}
\newcommand{\PP}{\mathcal{P}}
\renewcommand{\SS}{\mathcal{S}}
\newcommand{\TT}{\mathcal{T}}
\newcommand{\NN}{\mathcal{N}}
\newcommand{\UU}{\mathcal{U}}
\newcommand{\XX}{\mathcal{X}}

\newcommand{\dual}[3][]{#1\langle#2\,,\,#3#1\rangle}
\newcommand{\bigdual}[3][]{#1\Big\langle#2\,,\,#3#1\Big\rangle}

\newcommand{\enorm}[2][]{#1|\!#1|\!#1|\,#2\,#1|\!#1|\!#1|}
\newcommand{\norm}[3][]{#1\|#2#1\|_{#3}}
\newcommand{\bignorm}[3][]{#1 \Big\|#2#1 \Big\|_{#3}}
\newcommand{\diam}{{\rm diam}}
\newcommand{\set}[3][\big]{#1\{#2\,:\,#3#1\}}
\newcommand{\refine}{{\rm refine}}

\newcommand{\bilinb}[3][]{#1 b(#2\,,\,#3#1)}
\newcommand{\bilina}[3][]{#1 a(#2\,,\,#3#1)}

\def\Cmesh{C_{\rm mesh}}
\def\Cred{C_{\rm red}}
\def\Crel{C_{\rm rel}}
\def\Cson{C_{\rm son}}

\def\Copt{C_{\rm opt}}
\def\Cinv{C_{\rm inv}}
\def\Cinvtilde{\widetilde{C}_{\rm inv}}

\def\Cmark{C_{\rm mark}}
\def\Cstab{C_{\rm stb}}
\def\Clin{C_{\rm lin}}
\def\Cmon{C_{\rm mon}}
\def\Cshape{C_{\rm shape}}

\def\qmesh{q_{\rm mesh}}
\def\qred{q_{\rm red}}
\def\qlin{q_{\rm lin}}

\def\thetaopt{\theta_{\rm opt}}

\def\llin{\ell_{\rm lin}}
\def\lopt{\ell_{\rm opt}}
\def\lcea{\ell_{\text{C\'ea}}}

\def\nv{\mathbf{n}}
\def\g{g}

\def\patch{\omega}

\newcommand{\T}{\mathbb{T}}

\newcommand{\po}{\partial \Omega}

\def\exto{E_{0,\Gamma}}

\def\tildphi{\phi}
\def\tildpsi{\psi}
\def\tildv{v}

\def\slpk{\widetilde{V}_k}

\def\slpnull{\widetilde{V}_0}
\def\slok{V_k}
\def\slonull{V_0}

\def\dlpk{\widetilde{K}_k}

\def\dlpnull{\widetilde{K}_0}
\def\dlok{K_k}
\def\dlonull{K_0}

\def\adlok{K'_k}
\def\adlonull{K'_0}

\def\hsok{W_k}
\def\hsonull{W_0}

\def\Sv{\widetilde{S}_{V,k}}
\def\Av{\widetilde{A}_{V,k}}
\def\Sk{\widetilde{S}_{K,k}}
\def\Ak{\widetilde{A}_{K,k}}

\def\Sopv{S_{V,k}}
\def\Aopv{A_{V,k}}
\def\Sopk{S_{K,k}}
\def\Aopk{A_{K,k}}

\def\Kop{\mathcal{C}}
\def\Kopk{\mathcal{C}_k}

\newcommand{\traceint}{\gamma^{\rm int}_0}
\newcommand{\traceext}{\gamma_0^{\rm ext}}
\newcommand{\trace}{\gamma_0}
\newcommand{\condint}{\gamma_1^{\rm int}}
\newcommand{\condext}{\gamma_1^{\rm ext}}

\def\Omegaext{\Omega^{\rm ext}}
\def\Tref{T_{\rm ref}}

\def\coarse{\bullet}
\def\fine{\circ}

\newtheorem{lemma}{Lemma}

\newtheorem{theorem}[lemma]{Theorem}
\newtheorem{algorithm}[lemma]{Algorithm}
\newtheorem{remark}[lemma]{Remark}
\newtheorem{proposition}[lemma]{Proposition}

\renewcommand{\subsection}[1]{\refstepcounter{subsection}\medskip{\bf\thesubsection.~#1.}}

\newenvironment{explain}{\begin{list}{$\bullet$}{%
\setlength{\labelsep}{2.3mm}%
\setlength{\labelwidth}{3mm}%
\setlength{\leftmargin}{5mm}%
}}{\end{list}}

\usepackage{fancyhdr}
\advance\footskip0.4cm
\advance\textheight-0.4cm
\calclayout
\title{Adaptive BEM with optimal convergence rates \\ for the Helmholtz equation}

\author{Alex Bespalov}
\address{School of Mathematics, University of Birmingham, Edgbaston, Birmingham B15 2TT, UK}
\email{A.Bespalov@bham.ac.uk}

\author{Timo Betcke}
\address{Centre for Inverse Problems, University College London,  Gower Street, London WC1E 6BT, UK}
\email{T.Betcke@ucl.ac.uk}

\author{Alexander Haberl}
\author{Dirk Praetorius}
\address{TU Wien, Institute for Analysis and Scientific Computing, Wiedner Hauptstra\ss{}e 8--10, 1040 Wien, Austria}
\email{Alexander.Haberl@asc.tuwien.ac.at \qquad\rm(corresponding author)}
\email{Dirk.Praetorius@asc.tuwien.ac.at}


\begin{document}

\begin{abstract}
We analyze an adaptive boundary element method for the weakly-singular and hypersingular
integral equations for the 2D and 3D Helmholtz problem. The proposed adaptive algorithm 
is steered by a residual error estimator and does not rely on any {\sl a priori} information 
that the underlying meshes are sufficiently fine.
We prove convergence of the error estimator with optimal algebraic rates, independently 
of the (coarse) initial mesh. 
As a technical contribution, we prove certain local inverse-type estimates for the boundary integral operators associated with
the Helmholtz equation.
\end{abstract}

\maketitle
\thispagestyle{fancy}

\vspace*{-10mm}
\section{Introduction}
\label{section:introduction}

Adaptive boundary element methods (ABEMs) with (dis)continuous piecewise polynomials for second order elliptic problems are well 
understood if the boundary integral operator is strongly elliptic. In particular for the Laplace equation and lowest order boundary elements, 
optimal algebraic rates of convergence have been proved in~\cite{ffkmp:part1,ffkmp:part2,fkmp}
for polyhedral boundaries and in~\cite{gant} for smooth boundaries. An abstract framework is also found in~\cite{axioms}.
With the recent work~\cite{invest}, these results can also be extended to piecewise smooth boundaries. 

In recent years, isogeometric analysis has lead to a variety of works proving optimal rates for ABEM using spline basis functions; 
see, e.g.,~\cite{igabem,fghp16,fghp17} for the Laplace problem in two dimension as well as~\cite{gantnerphd} for a generalization 
to second-order linear elliptic PDEs in three dimensions. 

On the other hand, boundary element methods (BEMs) for the Helmholtz equation are very popular and used in many applications;
see, e.g.,~\cite{MR2916382,MR700400} for an overview of techniques in  acoustic scattering. 
To our knowledge, there are no results concerning optimal convergence of ABEM for indefinite problems, even for
sufficiently fine initial meshes. With this paper, we fill this gap in the theory.

In this work, we generalize existing results concerning optimal convergence of ABEM for the Laplace equation to the Helmholtz equation. 
To this end, let $\Omega \subset \R^d$ with $d = 2,3$ be a bounded Lipschitz domain with boundary  $\partial \Omega \supseteq \Gamma$.
 We consider ABEM for the Dirichlet or Neumann boundary value problem for the Helmholtz equation, i.e., 
 \begin{align}
 \begin{split}
 \label{eq:helmholtz}
	 - \Delta u - k^2 u &= 0 \text{ in } \Omega 
	 \quad \text{subject to either} \quad
	 u = g \text{ on }\Gamma
	 \quad \text{or} \quad \partial_{\nv} \, u =\phi \text{ on }\Gamma.
\end{split}
\end{align}
where $k \in \R$ denotes the wavenumber.
Independently, whether a direct or an indirect approach is used, the Dirichlet boundary value problems leads to the following weakly-singular integral equation. 
Suppose that $k^2$ is not an eigenvalue of the interior Dirichlet Problem (IDP).
Given a right-hand side $f \in H^{1/2}(\Gamma)$, find $\phi \in  \H^{-1/2}(\Gamma)$ such that
\begin{align}\label{eq:modelproblem}
	\slok \phi = f \quad \text{on} \,\, \Gamma,
\end{align}
where $\slok$ denotes the single-layer operator associated with the Helmholtz equation.
For $k=0$,  $\slok$ coincides with the single-layer operator of the Laplace equation. In this case, we refer to~\cite{fkmp,ffkmp:part1,gantumur}, where optimal algebraic convergence rates for the weakly-singular integral equation for the Laplace operator are shown.

In this paper, we focus on the case $k\neq0$. We build on the abstract framework developed in~\cite{helmholtz}
and propose an adaptive algorithm (Algorithm~\ref{algorithm}) for the numerical solution of problem~\eqref{eq:modelproblem}, which
does not require any {\sl a priori} information on whether the
initial mesh (or any locally refined mesh generated by the algorithm) is sufficiently fine.
In the algorithm, 
the local mesh-refinement is guided by the weighted-residual {\sl a~posteriori} error estimator,
and the classical adaptive loop is complemented by
an additional step that performs a uniform mesh-refinement
if the Galerkin formulation of~\eqref{eq:modelproblem} does not admit a unique solution.
Even though the latter is unlikely to happen in practice, this case cannot be avoided in theory.

The main result of this work is Theorem~\ref{theorem:optimal}.
It states that Algorithm~\ref{algorithm} generates a convergent sequence of discrete solutions and, moreover,
the generated sequence of {\sl a~posteriori} error estimators converges linearly
with an optimal algebraic rate.
Theorem~\ref{theorem:optimal} is the first result that proves optimal convergence
rates for ABEM for the Helmholtz equation.
In addition to that, we emphasize that our adaptive algorithm
effects the optimal rate for any given, possibly coarse, initial mesh.
Although the presentation focuses on the weakly-singular equation in~\eqref{eq:modelproblem},
the adaptive algorithm and the main result of Theorem~\ref{theorem:optimal}
extend immediately to the hypersingular integral equation
corresponding to the Neumann boundary value problem in~\eqref{eq:helmholtz}.

The proof of Theorem~\ref{theorem:optimal} relies on the abstract framework developed in~\cite{helmholtz}
for compactly perturbed elliptic problems and
requires verification of the so-called {\sl axioms of adaptivity}~\cite{axioms}
for the weighted-residual error estimator.
In this work, we verify these by employing novel inverse-type estimates
for the underlying boundary integral operators.
These estimates exploit potential decompositions from~\cite{Mel12} and
generalize existing results for the Laplacian ($k = 0$)
to the case of an arbitrary wavenumber $k \ge 0$.
%
%
%
%
%
%
%

\smallskip
{\bf Outline.}\quad
This work and its main results are structured as follows: 
Section~\ref{section:definitions} recaps the functional analytic framework and introduces the involved integral operators 
as well as the Galerkin discretization by piecewise polynomials. 
In Section~\ref{section:inverse_estimate}, we prove inverse-type estimates in the style of~\cite{invest} for the Helmholtz operators. 
The exact adaptive algorithm and the {\sl a~posteriori} weighted-residual error estimator are given in Section~\ref{section:algorithm}.  
The main result (Theorem~\ref{theorem:optimal}) of this work is given in Section~\ref{section:optimalconvergence}. 
Further, Section~\ref{section:hypsing} comments on the extension of the analysis to the hypersingular equation. 
In the last Section~\ref{section:numerics}, we underpin our theoretical findings with some numerical experiments. 
A rigorous proof of the essential estimator properties and Theorem~\ref{theorem:optimal} is
	given in the Appendix.

Throughout all statements, the dependencies of all constants are given. In proofs, we may abbreviate the notation by use of the symbol $\lesssim$ which indicates
$\leq$ up to some multiplicative constant which is clear from the context. 
Analogously, $\gtrsim$ indicates $\geq$ up to a multiplicative constant. 
The symbol $\simeq$ states that both estimates $\lesssim$ and $\gtrsim$ hold.


\section{Preliminaries}
\label{section:definitions}

Let $\Omega \subset \R^d$ with $d = 2,3$ be a bounded Lipschitz domain with
piecewise $C^\infty$-boundary $\partial \Omega$ and exterior normal vector $\nv(y)$ for every $y \in \po$; see~\cite[Definition 2.2.10]{sauterschwab}.
Let $\Omegaext := \R^d \setminus \overline{\Omega}$ denote the corresponding exterior domain.
We suppose that $\Gamma = \partial \Omega$ or $\emptyset \neq \Gamma \subset \partial \Omega$ is a 
relative open set which stems from a Lipschitz dissection $\partial \Omega = \Gamma \cup \partial \Gamma \cup (\partial \Omega \setminus \Gamma)$;
see~\cite[p. 99]{mclean}.

\subsection{Sobolev spaces}
\label{subsection:sobolev}
For $s \in \{-1/2,0,1/2\}$, the Sobolev spaces $H^{1/2+s}(\po)$ are defined 
as in \cite[p. 100]{mclean} via Bessel-potentials and the Lip\-schitz parametrization of $\po$.
Let $\dual{\cdot}{\cdot}$ denote the duality pairing which extends the
$L^2(\po)$-scalar product. For $s \in \{-1/2,0,1/2\}$, the negative-order Sobolev spaces are defined by duality $H^{-(1/2+s)}(\po):=H^{1/2+s}(\po)'$.

If $\Gamma \subsetneqq \po$, let $\exto$ denote the extension operator which extends a function on $\Gamma$ to $\po$ by zero.
Then, the spaces $H^{1/2+s}(\Gamma)$ and $\H^{1/2+s}(\Gamma)$ are defined as in~\cite{mclean} by
\begin{eqnarray*}
	H^{1/2+s} (\Gamma) &:= \{ v|_\Gamma: v \in H^{1/2+s}(\po) \}, 
	  &\norm{v}{H^{1/2+s} (\Gamma)} := \inf \{ \norm{w}{H^{1/2+s}(\po)}: w|_\Gamma = v \}, \\
	\H^{1/2+s} (\Gamma) &:= \{ v: \exto v \in H^{1/2+s}(\po) \},  &\norm{v}{ \H^{1/2+s} (\Gamma)} := \norm{\exto v}{H^{1/2+s}(\po)}.
\end{eqnarray*}
For $s=1/2$, we have the following equivalences
\begin{align*}
\norm{u}{H^1(\po)}^2 \simeq \norm{u}{L^2(\po)}^2 + \norm{\nabla_\Gamma u}{L^2(\po)}^{2} \quad \text{as well as} \quad 
\norm{u}{\H^1(\Gamma)}^2 \simeq \norm{u}{L^2(\Gamma)}^2 + \norm{\nabla_\Gamma u}{L^2(\Gamma)}^{2};
\end{align*}
see, e.g.~\cite[Facts 2.1]{invest}. For $s \in \{-1/2,0,1/2\}$, the corresponding negative-order spaces are obtained by duality
\begin{align*}
	\H^{-(1/2+s)}(\Gamma) := H^{1/2+s}(\Gamma)' \quad \text{and} \quad H^{-(1/2+s)}(\Gamma) := \H^{1/2+s}(\Gamma)'.
\end{align*}
We emphasize that, for all $\psi \in L^2(\Gamma)$, it holds that $\exto \psi \in H^{-1/2}(\Omega)$
as well as $\norm{\psi}{\H^{-1/2}(\Gamma)}= \norm{\exto \psi}{H^{-1/2}(\po)}$.
We note the continuous inclusions
\begin{align*}
	\H^{\pm (1/2+s)}(\Gamma) \subseteq H^{\pm (1/2+s)}(\Gamma) \quad \text{and} \quad \H^{\pm (1/2+s)}(\po) = H^{\pm (1/2+s)}(\po).
\end{align*}
We make the following convention: 
If $\Gamma \subsetneqq \po$, and it is clear from the context, 
we identify any $v \in \H^{1/2+s}(\Gamma)$ with its extension $\exto v \in H^{1/2+s}(\po)$. 
Further, the operators $\slpk,\slok, \adlok$ are often applied to functions in $L^2(\Gamma)$, resp.\ 
$\dlpk,\dlok,\hsok$ are applied to functions in $\H^{1/2}(\Gamma)$. To ease notation, 
for $\psi \in L^2(\Gamma)$ and $v \in \H^{1/2}(\Gamma)$, we implicitly extend by zero, e.g., 
we write $\slok \psi$ instead of $\slok(\exto \psi)$ and $\dlok v$ instead of $\dlok(\exto v)$.

\subsection{Trace operators} 
We denote by $\traceint: H^1(\Omega) \to H^{1/2}(\po)$ the interior trace operator.
For $u \in H^1_{\Delta} := \{u \in H^1(\Omega): -\Delta u \in L^2(\Omega) \}$, we define the interior conormal derivative operator
via Green's first identity as 
\begin{align*}
	\condint\!:\! H^1_{\Delta}(\Omega) \to H^{-1/2}(\po), \, \dual{\condint u}{\traceint v}_{\po} := \dual{\nabla u\!}{\!\nabla v}_{\Omega} - \dual{- \Delta u\!}{\!v}_{\Omega} \quad \forall v \in H^1(\Omega).
\end{align*}
To define the exterior counterparts $\traceext$ and $\condext$, let $U \subset \R^d$ be a bounded Lipschitz domain such that 
$\overline{\Omega} \subset U \subset \R^d$. Then, the exterior trace operator $\traceext: H^1(U \setminus\overline{\Omega}) \to H^{1/2}(\po)$ 
is defined analogously as restriction to $\po$.
The exterior conormal derivative operator $\condext: H^1_{\Delta}(U \setminus \overline{\Omega}) \to H^{-1/2}(\po)$ 
is defined by $\dual{\condext u}{\traceext v}_{\po} := \dual{\nabla u}{\nabla v}_{U \setminus \Omega} - \dual{- \Delta u}{v}_{U \setminus \Omega}$ 
for all $v \in H^1(U \setminus \overline{\Omega})$ with $\traceext v =0$ on $\partial U$.
If a function $u$ admits interior and exterior trace, resp., interior and exterior conormal derivative, we define the jump
\begin{align*}
	[\gamma_1 u]:= \condext u - \condint u \quad \text{resp.} \quad [u] = \traceext u - \traceint u.
\end{align*}
We denote the surface gradient by $\nabla_{\Gamma}(\cdot)$.

\subsection{Layer potentials and boundary integral operators} 
\label{subsection:operators}
Let $k$ denote the wavenumber of the Helmholtz equation. For $k>0$, the Helmholtz kernel is given by
\begin{align}\label{eq:helmholtz_kernel}
 G_k(x,y) = \frac{i}{4} H^{(1)}_0 (k|x-y|)
 \quad\text{for $d=2$\quad resp.\quad}
 G_k(x,y) = \frac{e^{ik|x-y|}}{4 \pi |x-y|}
 \quad\text{for $d=3$},
\end{align}%
where $H^{(1)}_0$ is the first-kind Hankel function of order zero. 
For $k<0$, we define $G_{k} := \overline{G_{-k}}$ and if $k=0$, we employ the fundamental solution of
the Laplace operator
\begin{align} \label{eq:laplace_kernel}
 G_0(x,y) = -\frac{1}{2 \pi} \log|x-y|
 \quad\text{for $d=2$\quad resp.\quad}
 G_0(x,y) = \frac{1}{4 \pi |x-y|}
 \quad\text{for $d=3$}.
\end{align}%
For all $k \in \R$, the single-layer and double-layer potential operators are defined as
\begin{align*}
	(\slpk \phi) (x) := \int_{\po} G_k(x,y) \phi(y) \, dy 
	\quad \text{and} \quad 
	(\dlpk \phi ) (x) := \int_{\po} \partial_{\nv(y)} G_k(x,y) \phi(y) \, dy 
\end{align*}
which give rise to corresponding bounded linear operators $\slpk \in L\big( H^{-1/2}(\po); H^1(U)\big)$ 
and $\dlpk \in  L\big( H^{1/2}(\po); H^1(U \setminus \po)\big)$.

The single-layer potential induces the single-layer operator 
$$\slok := \traceint \slpk : H^{-1/2 +s}(\po) \to H^{1/2+s} (\po)$$
for $-1/2 < s \leq 1/2$; see, e.g., Theorem~\ref{theorem:invest} in the case of $s=1/2$. 
For $k=0$, $\slonull$ is even a well-defined isomorphism for $-1/2 \leq s \leq 1/2$, and elliptic as well as symmetric for $s=0$.
For $k \neq 0$, 
the single-layer operator $\slok$ is invertible, if and only if $k^2$ is not an eigenvalue of the interior Dirichlet problem (IDP) for the Laplace operator, i.e., it holds that
\begin{align}
	\tag{IDP}
	\forall u \in H^1(\Omega) \quad 
	\Big(- \Delta u =k^2 u \quad \text{with}  \quad \traceint u = 0 
	\quad \Longrightarrow \quad u = 0 \quad \text{in} \,\ \Omega \Big);
\end{align}
see, e.g.,~\cite[Theorem 3.9.1]{sauterschwab}.
Throughout, we assume that $k^2$ satisfies (IDP). 

The double-layer potential induces the double-layer operators 
$$\dlok^\sigma := \trace^\sigma \dlpk: H^{1/2+s}(\po) \to H^{1/2+s}(\po)$$
with $\sigma \in \{ \textrm{int}, \textrm{ext} \}$ and $-1/2 < s \leq  1/2$; see Theorem~\ref{theorem:invest} for $s=1/2$. 
Combining the two operators, we define $\dlok:= \frac{1}{2}(\dlok^{\textrm{int}} + \dlok^{\textrm{ext}}) : H^{1/2+s}(\po) \to H^{1/2+s}(\po)$.

We define the adjoint double-layer operator $\adlok: H^{-1/2+s} (\po) \to H^{-1/2+s}(\po)$ by $\adlok:= -\frac{1}{2} \textrm{Id} + \condint \slpk$. 
Further, the hypersingular operator is given by $\hsok := -\condint \dlpk:  H^{1/2+s} (\po) \to H^{-1/2+s}(\po)$.
For $k=0$, the operators $\dlonull$, $\adlonull$, as well as $\hsonull$ are even well defined 
 for $s = \pm 1/2$; see~\cite[Remark 3.1.18]{sauterschwab}.


\subsection{Admissible triangulations}
Let  $\Tref$ denote the reference element defined by 
\begin{align*}
\Tref=(0,1)
\quad \text{for $d=2$\quad resp.\quad}
\Tref ={\rm conv} \{(0,0),(1,0),(0,1) \}
\quad \text{for $d=3$},
\end{align*}
i.e., $\Tref$
is the open unit interval for $d=2$ and 
the Kuhn simplex for $d=3$. 
%
A set $\TT_\coarse$ is a regular triangulation of $\Gamma$, if the following conditions {\rm (a)}--{\rm (d)} hold:
\begin{enumerate}[label= {\rm (\alph*)}]
\item Each $T \in \TT_\coarse$ is a relative open subset of $\Gamma$, and there exists a bijective element map $\g_T \in C^\infty(T_{\rm ref},T)$ such that $\g_T(\overline{\Tref})=\overline{T}$.

\item The union of all elements cover $\Gamma$, i.e., $\overline{\Gamma} = \bigcup_{T \in \TT_\coarse} \overline{T}$.

\item For all $T,T' \in \TT_\coarse$, the intersection $ \overline{T} \cap \overline{T'}$ is either empty, or a joint node ($d \geq 2$), or  a joint facet ($d = 3$).

\item In the case of $d=3$, there holds the following: If  $\overline{T} \cap \overline{T'}$ is a facet, there exist facets $f,f' \subseteq \partial \Tref$ of $\Tref$ such that $ \overline{T} \cap \overline{T'} = \g_T(f)  = \g_{T'}(f')$, and the composition
$\g_T^{-1} \circ \g_{T'} : f' \to f$ is even affine. 
 \end{enumerate}

The element patch of $T \in \TT_\coarse$,  is given by
\begin{align*}
	\patch_{\coarse}(T):= {\rm interior} \Big( \bigcup_{{T' \in \TT_\coarse} \atop {\overline{T} \cap \overline{T'} \neq \emptyset}} \overline{T'} \Big). 
\end{align*}
For a set of elements $\UU \subseteq \TT_\coarse$, let $\patch_{\coarse}(\UU) := \{T' \in \TT_\coarse: \,\, \exists T \in \UU: \,\, T' \subseteq \patch_{\coarse}(T)	\}$.
Define the local mesh-size function
$h_\coarse \in L^\infty(\Gamma)$ by $h_\coarse|_T := h_\coarse(T) := |T|^{1/(d-1)}$ for all $T \in \TT_\coarse$.

To introduce shape regularity, let $G_T(x) := D \, \g_T(x)^\intercal D \, \g_T(x) \in \R^{(d-1) \times (d-1)}$ be the symmetric Gramian matrix of $g_T$ and 
$\lambda_{\textrm{min}} (G_T(x))$ and $\lambda_{\textrm{max}} (G_T(x))$ its extremal eigenvalues. A regular triangulation $\TT_\bullet$ is
$\gamma$-shape regular triangulation, if the following holds:
\begin{itemize} 
	\item 	For all $T \in \TT_\bullet$, the corresponding element maps $g_T(\cdot)$ satisfy that
	\begin{align}
		\sigma(T) := \displaystyle\sup_{x \in \Tref} \Big( \frac{h_\bullet(T)^2}{\lambda_{\textrm{min}} (G_T(x))} + \frac{\lambda_{\textrm{max}} (G_T(x))}{h_\bullet(T)^2} \Big) \leq \gamma.
	\end{align}
	
	\item If $d=2$, it additionally holds that
	\begin{align}
		\widetilde{\sigma}(\TT_\bullet) := \max_{{T,T' \in \TT_\bullet}\atop{T'\subseteq\omega_\bullet(T)}} \frac{|T|}{|T'|} \leq \gamma.
	\end{align}
\end{itemize}
Note that the Gramian matrix $G_T(x)$ is symmetric and positive definite. This implies that $0 \leq \lambda_{\textrm{min}} (G_T) \leq \lambda_{\textrm{max}} (G_T)$ and hence, $\sigma(T) \geq 0$. 
The additional assumption for $d=2$ ensures that the mesh-size of neighboring elements remains comparable. 

Throughout, we assume that $\TT_\coarse$ is a $\gamma$-shape regular triangulation.
The next lemma recaps some important properties of $\gamma$-shape regular meshes; see~\cite[Lemma 2.6]{invest}. 

\begin{lemma}\label{lemma:mesh}
There exists a constant $C>0$ that depends only on 
	$\gamma$ and the Lipschitz character of $\po$, such that the following assertions {\rm (a)}--{\rm (d)} hold:
	\begin{enumerate}[label= {\rm (\alph*)}]
		\item For all $T,T' \in \TT_\coarse$ such that $\overline{T} \cap \overline{T'} \neq \emptyset$, it holds that $h_\coarse(T) \leq C h_\coarse(T')$.
		
		\item The number of elements in an element patch is bounded by $C$.
		
		\item For all $T \in \TT_\coarse$ and all elements $T',T'' \subseteq \omega_{\coarse}(T)$, 
		there exists a sequence $T' = T_1, \ldots,T_n=T''$ with $T_i \subseteq \omega_{\coarse}(T)$ for all $1 \leq i \leq n$ 
		such that $\overline{T_i} \cap \overline{T_{i+1}}$ is a joint facet of $T_i$ and $T_{i+1}$ (for $d=3$), 
		resp., a joint node (for $d=2$).
		
		\item There exists a constant $\Cshape>0$ which depends only on $\gamma$, such that
		\begin{align*}
			\max_{T \in \TT_\coarse} \frac{ \diam(T)}{h_\coarse(T)} \leq \Cshape \quad \text{with} \quad \diam(T) := \sup_{x,y \in T} |x-y|.  
		\end{align*}
	\end{enumerate}
\end{lemma}

\subsection{Discrete spaces}
\label{section:Galerkindiscretization}
Let $\TT_\coarse$ be a regular triangulation of $\Gamma$. For a fixed polynomial
 degree $p \geq 0$, we define the space of (discontinuous) $\TT_\coarse$-piecewise polynomials by
\begin{align*}
 	\PP^p(\TT_\coarse) := \big\{\Phi_\coarse \in L^\infty(\Gamma): \,\, \forall T \in \TT_\coarse, \quad \Phi_\coarse \circ \g_T \,\, \text{is a polynomial of degree} \,\leq p  \big\}.
 \end{align*}
 Further, let  $\SS^p(\TT_\coarse) := \PP^p(\TT_\coarse) \cap H^1(\Gamma)$ resp. $\widetilde{\SS}^p(\TT_\coarse) := \PP^p(\TT_\coarse) \cap \H^1(\Gamma)$
 be the space of continuous piecewise polynomials.
Note 
the following (compact) inclusions
 \begin{align}
 	\PP^p(\TT_\coarse) \subset L^2(\Gamma) \subset \H^{-1/2}(\Gamma) \quad \text{and} \quad \widetilde{\SS}^p(\TT_\coarse) \subset \H^1(\Gamma)  \subset \H^{1/2}(\Gamma).
 \end{align}
In the case of $\Gamma = \po$, there holds $\widetilde{\SS}^p(\TT_\coarse) = \SS^p(\TT_\coarse)$ and $\SS^p(\TT_\coarse) \subset H^1(\Gamma)$.
Throughout this paper, we use the following convention: All quantities which are associated with a triangulation $\TT_\coarse$, have the 
same index, e.g., $h_\coarse \in \PP^0(\TT_\coarse)$ is the local mesh size function or $\Phi_\coarse $ will denote the discrete solution in $\PP^p(\TT_\coarse)$.


\section{Inverse estimate}
\label{section:inverse_estimate}
The main result of this section is the following inverse-type estimate which 
generalizes~\cite[Theorem 3.1]{fkmp} and~\cite[Theorem 3.1]{invest}
from $k=0$ to general $k \geq 0$.
\begin{theorem}\label{theorem:invest}
	The single-layer and the double-layer operator satisfy
	\begin{align}\label{eq:invest:wellposedness}
		\slok \in L \big( L^2(\Gamma),H^1(\Gamma) \big) \quad \text{resp.} \quad \dlok \in L \big( \H^1(\Gamma), H^1(\Gamma) \big).
	\end{align} 
	Additionally, let $\TT_\coarse$ be a $\gamma$-shape regular triangulation of $\Gamma$. 
	Then, there exists a constant $\Cinv>0$ which depends only on $\Gamma$, $\Omega$, and  
	$\gamma$, such that for all $k \geq 0$, it holds that 
	\begin{align}
		\Cinv^{-1} \, \norm{h^{1/2}_{\coarse} \, \nabla_\Gamma \, \slok \psi}{L^2(\Gamma)} 
		&\leq   (1+k^{3}) \norm{\psi}{\H^{-1/2}(\Gamma)} + \norm{h^{1/2}_{\coarse} \, \psi}{L^2(\Gamma)} ,
		 \label{eq:invest:slo}\\		
		\Cinv^{-1} \, \norm{h^{1/2}_{\coarse} \, \nabla_\Gamma \, \dlok v}{L^2(\Gamma)}
		&\leq   (1+k^3) \norm{v}{\H^{1/2}(\Gamma)} + \norm{h^{1/2}_{\coarse} \, \nabla_{\Gamma} v}{L^2(\Gamma)} ,
		\label{eq:invest:dlo} \\
		 \Cinv^{-1} \, \norm{h^{1/2}_{\coarse} \adlok \psi}{L^2(\Gamma)}
		&\leq   (1+k^{3}) \norm{\psi}{\H^{-1/2}(\Gamma)} + \norm{h^{1/2}_{\coarse} \, \psi}{L^2(\Gamma)}  ,
		 \label{eq:invest:adlo}	 \\
		 \Cinv^{-1} \, \norm{h^{1/2}_{\coarse} \, \hsok v}{L^2(\Gamma)} 
		&\leq   (1+k^3) \norm{v}{\H^{1/2}(\Gamma)} + \norm{h^{1/2}_{\coarse} \, \nabla_{\Gamma} v}{L^2(\Gamma)}  ,
		\label{eq:invest:hso} 
	\end{align}
	for all $\psi \in L^2(\Gamma)$ and $ v \in \H^1(\Gamma)$.
	Furthermore, there exists $\Cinvtilde>0$ which depends only on $\Omega$, $\Gamma$, $\gamma$, and $p$, such that
	\begin{align}
		\norm{h_\coarse^{1/2} \, \nabla_\Gamma \, \slok \Psi_\coarse}{L^2(\Gamma)} + \norm{h_\coarse^{1/2} \, \adlok \Psi_\coarse}{L^2(\Gamma)} 
		&\leq \Cinvtilde (1+k^{3}) \norm{\Psi_\coarse}{\H^{-1/2}(\Gamma)}, 
		 \label{eq:discrete:invest:slo} \\				
		\norm{h_\coarse^{1/2} \, \nabla_\Gamma \, \dlok V_\coarse}{L^2(\Gamma)} + \norm{h_\coarse^{1/2} \, \hsok V_\coarse}{L^2(\Gamma)} 
		&\leq \Cinvtilde (1+k^{3})  \norm{V_\coarse}{\H^{1/2}(\Gamma)} ,
		\label{eq:discrete:invest:dlo}
	\end{align} 
	for all $ \Psi_\coarse \in \PP^p(\TT_\coarse)$ and $ V_\coarse \in \widetilde{\SS}^{p+1}(\TT_\coarse)$.
	In particular, the constants $\Cinv,\Cinvtilde$ are independent of the wavenumber $k \geq 0$.
\end{theorem}

The proof of Theorem~\ref{theorem:invest} is based on the decomposition of the layer potentials into a singular part, which consists of
the layer potentials $\slpnull$, resp., $\dlpnull$, of the Laplacian and two smoothing operators $\widetilde{S}$ and $\widetilde{A}$. 
For the decomposition, we employ the following notation
\begin{align*} 
	|\nabla^n \psi(x)|^2 := \sum_{\alpha \in \N_0^d \atop |\alpha| = n}	\frac{n!}{\alpha!} |D^\alpha \psi(x)|^2 \,\, \text{with} \,\, \alpha! := \alpha_1! \cdot \alpha_2! \ldots \cdot \alpha_d!
	\quad \text{and} \quad 	|\nabla^0 \psi(x)|^2 := |\psi(x)|^2.
\end{align*}
Lemma~\ref{lemma:decomposition:slp} provides such a decomposition for the single-layer potential,
while Lemma~\ref{lemma:decomposition:dlp} states a similar result for the double-layer potential.	
\begin{lemma}[{\cite[Theorem 5.1.1]{Mel12}}]\label{lemma:decomposition:slp}
	Let $R>0$ with $\overline{\Omega} \subsetneqq B_R := \{x \in \R^d: |x| < R \}$. Let $0 < \rho < 1$.
	Then, it holds that
	\begin{align}\label{eq:decomp:slp}
		\slpk = \slpnull + \Sv + \Av,
	\end{align}
	with linear potential operators $\Sv: H^{-1/2+s}(\po) \to H^{3+s} (B_R)$ and
	$\Av:  H^{-1/2+s}(\po) \to H^{3+s} (B_R) \cap C^\infty(B_R)$ for all $-1/2 < s < 1/2$. 
	Moreover, there exist positive constants $C_1^V,C_2^V,C_3^V>0$ such that
	\begin{align}
		\norm{\Sv \psi}{H^{s'}(B_R)} &\leq C_1^V \rho^2(\rho k^{-1})^{1+s-s'} \norm{\psi}{H^{-1/2+s}(\po)} \quad \text{for all } 0 \leq s' \leq 3+s, \label{eq:decomp:slp:S} \\ 
		\norm{\nabla^n \Av \psi}{L^2(B_R)} &\leq C_2^V k^{n+1} \norm{\slpnull \psi}{L^2(B_R)} \leq C_3^V k^{n+1} \norm{\psi}{H^{-1}(\po)} \label{eq:decomp:slp:A}
		\quad \text{for all } n \in \N_0.	
	\end{align} 
	The constants $C_1^V,C_2^V$, and $C_3^V$ depend only on $R$, $\Omega$, but not on the wavenumber $k$.\qed
\end{lemma}
The decomposition of the double-layer potential additionally involves certain Besov spaces, 
which are defined by the K-method of interpolation; see e.g.~\cite{tartar,triebel83,triebel92}.
For bounded Lipschitz domains $\widetilde{\Omega} \subset \R^d$ and $s \in \N_0$ as well as $s' \in (0,1)$, we require the Besov space
$B^{s+s'}_{2,\infty}(\widetilde{\Omega}) := \big( H^s(\widetilde{\Omega}),H^{s+1}(\widetilde{\Omega})  \big)_{s',\infty}$.

\begin{lemma}[{\cite[Theorem 5.2]{Mel12}}]\label{lemma:decomposition:dlp}
	Let $R>0$ with $\overline{\Omega} \subsetneqq B_R := \{x \in \R^d: |x| < R \}$. 
	It holds that
	\begin{align}\label{eq:decomp:dlp}
		\dlpk = \dlpnull + \Sk + \Ak,
	\end{align}
	with linear potential operators $\Sk: L^2(\po) \to B^{5/2}_{2,\infty} (B_R)$ as well as
	$\Ak: L^2(\po) \to B^{5/2}_{2,\infty} (B_R) \cap C^\infty(B_R)$. 
	Moreover, there exist constants $C_1^K,C_2^K,C_3^K>0$, such that
	\begin{align}
		\norm{\Sk v}{B^{5/2}_{2,\infty}(B_R)} &\leq C_1^K k \norm{v}{L^2(\po)}, \label{eq:decomp:dlp:S} \\ 
		\norm{\nabla^n \Ak v}{L^2(B_R)} &\leq C_2^K k^{n+1} \norm{\dlpnull v}{L^2(B_R)} \leq C_3^K k^{n+1} \norm{v}{L^2(\po)} \quad \text{for all } n \in \N_0.	\label{eq:decomp:dlp:A}
	\end{align} 
	The constants $C_1^K,C_2^K$, and $C_3^K$ depend only on $R$, $\Omega$, but not on the wavenumber $k$.\qed
\end{lemma}


\begin{proof}[Proof of Theorem~\ref{theorem:invest}]
Let $k>0$ and $R>0$ with $B_R \supsetneqq \overline{\Omega}$. 
For convenience of the reader, we split the proof into several steps.

\medskip
{\bf Step~1 (Proof of~(\ref{eq:invest:wellposedness}) for $\boldsymbol{\slok}$): }\quad
Let $\psi \in L^2(\Gamma)$ and recall that $\norm{ \psi}{\H^{-1/2}(\Gamma)} = \norm{\tildpsi }{H^{-1/2}(\po)}$, where we identify $\psi$ with its extension $\exto \, \psi$. 
With Lemma~\ref{lemma:decomposition:slp} and the definition of $\slok := \traceint \, \slpk$, we decompose $\slok = \slonull + \Sopv + \Aopv$, where 
\begin{align*}
	\Sopv := \traceint \, \Sv \quad \text{and} \quad \Aopv := \traceint \, \Av.
\end{align*}
For all  $1/2 < s' \leq 3+s \le 3+1/2$, equation~\eqref{eq:decomp:slp:S} implies that 
\begin{align}\label{eq:proof:inves:temp:one}
\norm{\Sv \, \tildpsi}{H^{s'}(B_R)} \stackrel{\eqref{eq:decomp:slp:S}}{\lesssim} \rho^2 \, (\rho \, k^{-1})^{1+s-s'} \norm{ \tildpsi}{H^{-1/2+s}(\po)}.
\end{align} 
For $s'=2$ and $s=0$, this reveals $\Sv \, \tildpsi \in H^2(B_R)$.
Further, stability of $\traceint$ yields that
 \begin{align}\label{eq:mapping:sv}
 	\norm{\Sopv \tildpsi}{H^{1}(\Gamma)} &\leq \norm{\Sopv \tildpsi}{H^{1}(\po)} 
 	\lesssim \norm{\Sv \tildpsi}{H^{3/2}(B_R)}
 	\lesssim\norm{\Sv \tildpsi}{H^{2}(B_R)} 
   \stackrel{\eqref{eq:proof:inves:temp:one}}{\lesssim} \rho  k \, \norm{\psi}{\H^{-1/2}(\Gamma)}.
\hspace*{-4mm}
 \end{align}
  Next, note that equation~\eqref{eq:decomp:slp:A} proves that $\Aopv \tildpsi \in H^2(B_R)$. 
  With the (compact) embedding $H^{-1/2}(\po) \subset H^{-1}(\po)$ with
 $\norm{\cdot}{H^{-1}(\po)} \lesssim \norm{\cdot}{H^{-1/2}(\po)}$, this yields that
  \begin{align}\label{eq:mapping:av}
   \norm{\Av \, \tildpsi }{H^2(B_R)} 
  \stackrel{\eqref{eq:decomp:slp:A}}{\lesssim}  (k + k^2 +k^3) \, \norm{\tildpsi }{H^{-1}(\po)} 
  \lesssim (1+k^3) \, \norm{\psi}{\H^{-1/2}(\Gamma)}.
  \end{align}
Similarly to~\eqref{eq:mapping:sv}, continuity of the trace operator proves that 
 \begin{align}\label{eq:mapping:av:two}
 \begin{split}
 	\norm{\Aopv \, \tildpsi }{H^1(\Gamma)}  &\leq  \norm{\Aopv \, \tildpsi}{H^1(\po)} 
	\lesssim 
	\norm{\Av \, \tildpsi}{H^2(B_R)} 
	\stackrel{\eqref{eq:mapping:av}}{\lesssim}  (1+k^3) \, \norm{\psi}{\H^{-1/2}(\Gamma)}.
\end{split}
\end{align} 
Combining the estimates~\eqref{eq:mapping:sv} and~\eqref{eq:mapping:av:two} with the (compact) embedding
$L^2(\Gamma) \subset \H^{-1/2}(\Gamma)$, we see that $\Aopv,\Sopv \in L \big(L^2(\Gamma), H^1(\Gamma) \big)$. With $\slonull \in L \big(L^2(\Gamma), H^1(\Gamma) \big)$, we conclude that $ \slok = \slonull + \Sopv + \Aopv \in L \big(L^2(\Gamma), H^1(\Gamma) \big)$.

\medskip

{\bf Step~2 (Proof of equation~(\ref{eq:invest:slo})):} \quad 
Recall that $\slok = \slonull + \Sopv + \Aopv$. This decomposition directly yields that
\begin{align} \label{eq:proof:invest:slo}
\begin{split}
	\norm{h_\coarse^{1/2} \nabla_\Gamma  \slok \tildpsi}{L^2(\Gamma)} &\leq \norm{h_\coarse^{1/2} \nabla_\Gamma  \slonull \, \tildpsi}{L^2(\Gamma)} 
	+ \norm{h_\coarse^{1/2}  \nabla_\Gamma  \Sopv \tildpsi}{L^2(\Gamma)} + \norm{ h_\coarse^{1/2}  \nabla_\Gamma  \Aopv \tildpsi}{L^2(\Gamma)}.
	\hspace*{-4mm}
\end{split}	
\end{align}
We treat each term on the right-hand side separately. \cite[Theorem~3.1]{invest} yields that
\begin{align*}
	\norm{h_\coarse^{1/2} \, \nabla_\Gamma \, \slonull \,  \tildpsi}{L^2(\Gamma)} \lesssim  \norm{\psi}{\H^{-1/2}(\Gamma)} 
		+ \norm{h_\coarse^{1/2} \, \psi}{L^2(\Gamma)}.
\end{align*}
Next, $\norm{h_\coarse}{L^\infty(\Gamma)} \lesssim  \diam(\Omega) \lesssim 1$ and equation~\eqref{eq:mapping:sv} imply that 
\begin{align*}
	 \norm{h_\coarse^{1/2} \, \nabla_\Gamma \, \Sopv \, \tildpsi}{L^2(\Gamma)} 
	 \lesssim 
	 \norm{\Sopv \, \tildpsi}{H^{1}(\Gamma)} %
	 \stackrel{\eqref{eq:mapping:sv} }{\lesssim} 
	 k \, \norm{\psi}{\H^{-1/2}(\Gamma)}.
\end{align*}
Finally, we use equation~\eqref{eq:mapping:av:two} to estimate the last term on the right hand side of~\eqref{eq:proof:invest:slo} by
\begin{align*} 
\norm{h_\coarse^{1/2} \, \nabla_\Gamma \, \Aopv \tildpsi}{L^2(\Gamma)} 
\lesssim 
\norm{\Aopv \tildpsi}{H^1(\Gamma)} 
\stackrel{\eqref{eq:mapping:av:two} }{\lesssim}  (1 + k^3) \, \norm{\psi}{\H^{-1/2}(\Gamma)}. 
\end{align*}
Combining the latter four estimates, we prove that
\begin{align*}
	\norm{h_\coarse^{1/2} \, \nabla_\Gamma \, \slok \psi}{L^2(\Gamma)} \lesssim  (1+k^3) \, \norm{\psi}{\H^{-1/2}(\Gamma)}
		+ \norm{h_\coarse^{1/2} \, \psi}{L^2(\Gamma)}.
\end{align*}
This concludes the proof of~\eqref{eq:invest:slo}.

\medskip
{\bf Step~3 (Proof of equation~(\ref{eq:invest:adlo})):} \quad 
Recall the definition of the adjoint double-layer operator. With
$\adlok = - \frac{1}{2}{\rm Id} + \condint \, \slpk = \adlonull + \condint \, \Sv + \condint \, \Av$, this implies that
\begin{align*}
\norm{h_\coarse^{1/2} \, \adlok \, \tildpsi}{L^2(\Gamma)} \leq  \norm{h_\coarse^{1/2} \, \adlonull \,  \tildpsi}{L^2(\Gamma)} 
 + \norm{h_\coarse^{1/2} \, \condint \, \Sv \,  \tildpsi}{L^2(\Gamma)} +  \norm{h_\coarse^{1/2} \, \condint \,  \Av \,  \tildpsi}{L^2(\Gamma)}.
\end{align*}
Again, we treat each term on the right-hand side separately.
First, \cite[Theorem~3.1]{invest} yields that
\begin{align*}
\norm{h_\coarse^{1/2} \, \adlonull \, \tildpsi}{L^2(\Gamma)} \lesssim \norm{\psi}{\H^{-1/2}(\Gamma)} 
		+ \norm{h_\coarse^{1/2} \, \psi}{L^2(\Gamma)}.
\end{align*}
Recall from Step~1 that
$\Sv \, \tildpsi, \,  \Av \, \tildpsi \in H^{2}(B_R)$. Therefore, \cite[Remark 2.7.5]{sauterschwab} implies that
$\condint \Sv \tildpsi, \condint \Av \tildpsi \in H^{1/2}(\po)$. 
With $\norm{h_\coarse}{L^\infty(\Gamma)} \lesssim \diam(\Omega) \lesssim 1$, the
(compact) embedding $H^{1/2}(\po) \subset L^2(\po)$ and stability (\cite[Remark 2.7.5]{sauterschwab}) of the conormal derivative yield that
\begin{align*} 
	\norm{h_\coarse^{1/2} \, \condint \, \Sv\,  \tildpsi}{L^2(\Gamma)} 
	\lesssim  \norm{\condint \, \Sv \, \psi}{H^{1/2}(\po)} 
	\lesssim \norm{\Sv \, \tildpsi}{H^2(B_R)} \stackrel{\eqref{eq:mapping:sv}}{\lesssim} 
	 k \, \norm{ \psi}{\H^{-1/2}(\Gamma)}.
\end{align*}
Third, we argue as before and prove that
\begin{align*}
	\norm{h_\coarse^{1/2} \, \condint \, \Av \tildpsi}{L^2(\Gamma)} 
	\lesssim \norm{\condint \, \Av \, \tildpsi}{H^{1/2}(\po)} 
	\lesssim \norm{\Av \, \tildpsi}{H^{2}(B_R)} \stackrel{\eqref{eq:mapping:av}}{\lesssim} 
	(1+k^3) \, \norm{\psi}{\H^{-1/2}(\Gamma)} .
\end{align*}
Combining the right-hand sides of all estimates, we obtain that
\begin{align*}
	\norm{h_\coarse^{1/2} \, \adlok \tildpsi}{L^2(\Gamma)} \lesssim (1+k^{3}) \norm{\psi}{\H^{-1/2}(\Gamma)} + \norm{h_\coarse^{1/2} \, \psi}{L^2(\Gamma)},
\end{align*}
and conclude the proof of~\eqref{eq:invest:adlo}.

\medskip

{\bf Step~4 (Proof of~(\ref{eq:invest:wellposedness}) for $ \boldsymbol{\dlok}$):} \quad Let $v \in \H^1(\Gamma)$. 
Analogously to Step~1, Lemma~\ref{lemma:decomposition:dlp} yields that $\dlok = \dlonull + \Sopk + \Aopk$, where
 $\Sopk := \traceint \, \Sk$ and $\Aopk := \traceint \, \Ak$. 
 
For $-\infty < \sigma < s < \infty$, $0<q<\infty$, and $0<r,t\leq \infty$,
 there holds the continuous embedding $B^s_{q,r}(B_R) \subset B^\sigma_{q,t}(B_R)$; see, e.g.,~\cite[Section 2.32]{triebel92}.
This implies that $B^{5/2}_{2,\infty} (B_R) \subset B^{2}_{2,2}(B_R) = H^2(B_R)$ with
 $\norm{\cdot}{H^2(B_R)} \lesssim \norm{\cdot}{B^{5/2}_{2,\infty}(B_R)}$.
Analogously to~\eqref{eq:mapping:sv}, continuity of the interior trace operator $\traceint$ and inequality~\eqref{eq:decomp:dlp:S} reveal that
 \begin{align}\label{eq:mapping:sk}
 \begin{split}
 \norm{\Sopk \tildv}{H^1(\Gamma)} 
 &\lesssim \norm{\Sk \tildv}{H^{2}(B_R)} 
 \lesssim \norm{\Sk \tildv}{B^{5/2}_{2,\infty}(B_r)}  \stackrel{\eqref{eq:decomp:dlp:S}}{\lesssim} k \, \norm{\tildv}{L^2(\po)} = k \, \norm{v}{L^2(\Gamma)} .
 \end{split}
 \end{align}
The operator $\Aopk$ is treated analogously to Step~1 and hence satisfies that
\begin{align} \label{eq:mapping:ak}
	\norm{\Aopk \, \tildv}{H^1(\Gamma)} \lesssim \norm{\Ak \, \tildv}{H^{3/2}(B_R)} \leq \norm{\Ak \, \tildv}{H^2(B_R)} \lesssim ( 1 + k^3) \, \norm{v}{L^2(\Gamma)}.
\end{align}
Then, the estimates~\eqref{eq:mapping:sk} and~\eqref{eq:mapping:ak} prove that $\Sopk, \, \Aopk \in L\big(L^2(\Gamma),H^1(\Gamma)\big)$.
With $\dlonull \in L \big( \H^1(\Gamma), H^1(\Gamma) \big)$,
we conclude that $\dlok = \dlonull + \Sopk + \Aopk \in L\big(\H^1(\Gamma), H^1(\Gamma)\big)$.

\medskip

{\bf Step~5 (Proof of equation~(\ref{eq:invest:dlo})):} \quad  Let $v \in \H^1(\Gamma)$. 
Analogously to Step~2, the decomposition $\dlok = \dlonull + \Sopk + \Aopk$ implies that
 \begin{align*}
 \begin{split}
 	\norm{h_\coarse^{1/2} \, \nabla_\Gamma \, \dlok \tildv}{L^2(\Gamma)} 
 	&\leq \norm{h_\coarse^{1/2} \, \nabla_\Gamma \, \dlonull \tildv}{L^2(\Gamma)} 
	+ \norm{h_\coarse^{1/2} \, \nabla_\Gamma \, \Sopk \tildv}{L^2(\Gamma)} 
 		+ \norm{h_\coarse^{1/2} \, \nabla_\Gamma \, \Aopk \tildv}{L^2(\Gamma)}.
 \end{split}
 \end{align*}
We proceed as before.
First, \cite[Theorem~3.1]{invest} yields that
 \begin{align*}
 	\norm{h_\coarse^{1/2} \, \nabla_\Gamma \, \dlonull \, \tildv}{L^2(\Gamma)} 
 	\lesssim  \norm{v}{\H^{1/2}(\Gamma)} + \norm{h_\coarse^{1/2} \,\nabla_\Gamma \, v}{L^2(\Gamma)} .
 \end{align*}
Second, $\norm{h_\coarse}{L^\infty(\Gamma)} \lesssim \diam(\Omega) \lesssim 1$  and equation~\eqref{eq:mapping:sk} imply that
 \begin{align*}
 	\norm{h_\coarse^{1/2} \, \nabla_\Gamma \, \Sopk \, \tildv}{L^2(\Gamma)} 
 	\lesssim \norm{\Sopk \, \tildv}{H^1(\Gamma)} \stackrel{\eqref{eq:mapping:sk}}{\lesssim} k \, \norm{v}{L^2(\Gamma)}.
 \end{align*}
Third, we use equation~\eqref{eq:mapping:ak} to see that
\begin{align*}
\norm{h_\coarse^{1/2} \, \nabla_\Gamma \Aopk v}{L^2(\Gamma)} 
\lesssim \norm{\Aopk \, v}{H^1(\Gamma)}
\stackrel{\eqref{eq:mapping:ak}}\lesssim (1 + k^3) \, \norm{v}{L^2(\Gamma)}.
\end{align*}
 Combining the latter estimates, we obtain that 
 \begin{align*}
 		\norm{h_\coarse^{1/2} \, \nabla_\Gamma \, \dlok \tildv}{L^2(\Gamma)}  &\lesssim \norm{v}{\H^{1/2}(\Gamma)} + (1 + k^3) \, \norm{v}{L^2(\Gamma)}  + \norm{h_\coarse^{1/2} \,\nabla_\Gamma \, v}{L^2(\Gamma)} \\
 		&\lesssim (1+k^3) \, \norm{v}{\H^{1/2}(\Gamma)} + \norm{h_\coarse^{1/2} \,\nabla_\Gamma \, v}{L^2(\Gamma)}. 
 \end{align*}
 This concludes the proof of~\eqref{eq:invest:dlo}.

\medskip

{\bf Step~6 (Proof of equation~(\ref{eq:invest:hso})):} \quad
Recall the definition of $\hsok$. With $\dlpk = \dlpnull + \Sk + \Ak$
there holds  $\hsok = -\condint \, \dlpk = \hsonull - \condint \, \Sk - \condint \, \Ak$
and hence
\begin{align*}
	\norm{h_\coarse^{1/2} \, \hsok \tildv}{L^2(\Gamma)} &\leq \norm{h_\coarse^{1/2} \,  \hsonull \, \tildv}{L^2(\Gamma)}
	+ \norm{h_\coarse^{1/2} \,  \condint \Sk \, \tildv}{L^2(\Gamma)} + \norm{ h_\coarse^{1/2} \, \condint \, \Ak \, \tildv}{L^2(\Gamma)},
\end{align*}
We proceed as before.
First, \cite[Theorem~3.1]{invest} yields that
\begin{align*}
\norm{h_\coarse^{1/2} \, \hsonull \, \tildv}{L^2(\Gamma)} \lesssim \norm{v}{\H^{1/2}(\Gamma)} + \norm{h_\coarse^{1/2} \,\nabla_\Gamma \, v}{L^2(\Gamma)}.
\end{align*}
Recall from Step~4 that $\Sk \tildv, \Ak \tildv \in H^2(B_R)$ and hence $\condint \Sk \, \tildv, \, \condint \, \Ak \tildv \in H^{1/2}(\po)$.
As in Step~3, stability of $\condint$ gives
  \begin{align*}
  	\norm{h_\coarse^{1/2} \, \condint \, \Sk \, \tildv}{L^2(\Gamma)} &\lesssim \norm{\Sk \, \tildv}{H^2(B_R)} \stackrel{\eqref{eq:mapping:sk}}{\lesssim}
  	k \, \norm{v}{L^2(\Gamma)}, \\
  	\norm{h_\coarse^{1/2} \, \condint \, \Ak \, \tildv}{L^2(\Gamma)} &\lesssim \norm{\Ak \, \tildv}{H^{2}(B_R)} 
  		\stackrel{\eqref{eq:mapping:ak}}{\lesssim} (1+k^3) \, \norm{v}{L^2(\Gamma)}.
  \end{align*}
Combining the latter four estimates, we conclude the proof of~\eqref{eq:invest:hso}.

\medskip

{\bf Step~7 (Proof of equations~(\ref{eq:discrete:invest:slo})--(\ref{eq:discrete:invest:dlo})):}  \quad
 According to~\cite{ghs,Geo08} or \cite[Lemma~A.1]{invest}, there hold the following inverse estimates
	\begin{align}
		\norm{h_\coarse^{1/2} \, (p+1)^{-1} \, \Psi_\coarse}{L^2(\Gamma)} 
			&\lesssim \norm{\Psi_\coarse}{H^{-1/2}(\Gamma)} \quad \text{for all } \Psi_\coarse \in \PP^p(\TT_\coarse), \label{eq:smallinvest:one} \\
		\norm{h_\coarse^{1/2} \, (p+1)^{-1} \, \nabla_\Gamma V_\coarse}{L^2(\Gamma)} 
			&\lesssim \norm{V_\coarse}{\H^{1/2}(\Gamma)} \quad \text{for all } V_\coarse \in \widetilde{\SS}^p(\TT_\coarse),\label{eq:smallinvest:two}
	\end{align}
	where $p$ is the fixed polynomial degree. The hidden constant depends only on $\po$, $\Gamma$, and the shape regularity of $\TT_\coarse$.
	 Applying~\eqref{eq:smallinvest:one}--\eqref{eq:smallinvest:two} to the right-hand sides of equations~\eqref{eq:invest:slo} and~\eqref{eq:invest:hso}, we conclude~\eqref{eq:discrete:invest:slo}. Using~\eqref{eq:smallinvest:one}--\eqref{eq:smallinvest:two} to estimate the right-hand sides of~\eqref{eq:invest:dlo}
	 and~\eqref{eq:invest:hso}, we reveal~\eqref{eq:discrete:invest:dlo}. This concludes the proof. 
\end{proof}

\section{Adaptive algorithm}
\label{section:algorithm}

In this section, we introduce the adaptive algorithm as well as a suitable \textsl{a posteriori}
error estimator. We show that our ABEM
setting fits in the abstract framework 
of~\cite[Section 2]{helmholtz}, where an adaptive algorithm for compactly perturbed elliptic problems is analyzed and 
optimal algebraic convergence rates are proved.

\subsection{Framework}
\label{subsection:abstract}
We consider the model problem~\eqref{eq:modelproblem} in the following functional analytic framework. 
For each admissible triangulation $\TT_\coarse$, we consider $\TT_\coarse$-piecewise polynomial ansatz and test spaces $\PP^p(\TT_\coarse)$.
On Lipschitz boundaries $\po$, the operator $\Kopk := \slok-\slonull: H^{-1/2}(\po) \to H^{1/2}(\po)$ is compact;
see, e.g.,~\cite[Lemma 3.9.8]{sauterschwab} or~\cite[Section 6.9]{steinbach}.
This implies compactness of $\Kopk : \H^{-1/2}(\Gamma) \to H^{1/2}(\Gamma)$. 
Therefore, the model problem~\eqref{eq:modelproblem} can equivalently be reformulated as follows: 
Given $f \in H^{1/2}(\Gamma) $, find $\phi \in \H^{-1/2}(\Gamma)$ such that 
\begin{align}\label{eq:modelproblem:kompakt}
	\big( \slonull + \Kopk \big) \tildphi = f \quad \text{on } \Gamma.  
\end{align}
The weak formulation of~\eqref{eq:modelproblem:kompakt} 
thus seeks 
$\phi \in \H^{-1/2}(\Gamma) $ such that
\begin{align}\label{eq:cont_weakform:kompact}
 \dual{\slonull \, \tildphi}{\psi} + \dual{\Kopk \, \tildphi}{\psi} = \dual{f}{\psi} \quad \text{ for all } \psi \in \H^{-1/2}(\Gamma).
\end{align}
Recall that $\slonull$ is an elliptic and symmetric isomorphism. Hence,
\begin{align*}
	\bilina{\chi}{\psi} := \dual{\slonull \, \chi}{\psi} \quad \text{for all } \chi,\psi \in \H^{-1/2}(\Gamma),
\end{align*}
defines a scalar product that induces an equivalent energy norm 
$\enorm{\psi} := \bilina{\psi}{\psi}^{1/2}\simeq \norm{\psi}{\H^{-1/2}(\Gamma)}$ on $\H^{-1/2}(\Gamma)$. 
The Galerkin discretization 
seeks 
$\Phi_\coarse \in \PP^p(\TT_\coarse)$ such that
\begin{align}\label{eq:discreteform:kompakt}
	 \bilina{\Phi_\coarse }{\Psi_\coarse} + \dual{ \Kopk \, \Phi_\coarse }{\Psi_\coarse} = \dual{f}{\Psi_\coarse} \quad \text{ for all } \Psi_\coarse \in \PP^p(\TT_\coarse).
\end{align}
Existence of solutions to the~\eqref{eq:discreteform:kompakt} can be guaranteed by the following proposition which is applied for $\HH = \H^{-1/2}(\Gamma)$ and $\XX_\coarse = \PP^p(\TT_\coarse)$; 
 see~\cite[Theorem 4.2.9]{sauterschwab} or~\cite[Proposition 1]{helmholtz}. 

Similar results are also found in~\cite[Theorem 6.1]{epsbuch}, which is based on~\cite[Lemma 1.1]{cs87}, or in in~\cite[Theorem 5.7.8]{brennerscott}, where the proof relies on additional regularity of the dual problem.

\begin{proposition}\label{prop:ss}

Let $\HH$ be a separable Hilbert space and let $(\XX_\ell)_{\ell \in \N_0}$ be a dense sequence of discrete subspaces $\XX_\ell \subset \HH$, i.e., 
 $\min_{\Psi_\ell \in \XX_\ell} \norm{\psi - \Psi_\ell}{\HH} \to 0$ as $\ell \to \infty$ for all $\psi \in \HH$.
Let $\bilina{\cdot}{\cdot}$ be an hermitian continuous, and elliptic sesquilinear form on $\HH$.
Moreover, let $\Kop: \HH \to \HH^*$ be a compact operator and $f \in \HH^*$. Consider the following variational formulation: Find $\phi \in \HH$ such that
\begin{align}\label{eq:ss:model}
\bilinb{\phi}{\psi} := \bilina{\phi}{\psi} + \dual{\Kop \phi}{\psi} = \dual{f}{\psi} \quad \text{for all } \psi \in \HH.
\end{align}
 Suppose well-posedness of~\eqref{eq:ss:model}, i.e., for all $\phi \in \HH$, it holds that
 \begin{align}
 \phi = 0	\quad \Longleftrightarrow \quad \forall \psi \in \HH \quad \bilinb{\phi}{\psi} = 0.
 \end{align} 
 Then, there exists some index $\ell_\coarse \in \N$ such that for all discrete subspaces $\XX_{\coarse} \subsetneqq \HH$ with 
 $\XX_\coarse \supseteq \XX_{\ell_\coarse}$, the following holds:
 There exists $\beta>0$, which depends only on $\XX_{\ell_\coarse}$, such that the discrete inf--sup constant on $\XX_\coarse$ is uniformly bounded from below, i.e.,
 \begin{align}\label{eq:ssprop:infsup}
 	0 < \beta \leq \beta_\coarse := \inf_{\Phi_\coarse \in \XX_\coarse \setminus \{0\} } \sup_{\Psi_\coarse \in \XX_\coarse \setminus \{0\}}
 	\frac{|\bilinb{\Phi_\coarse}{\Psi_\coarse}|}{\norm{\Phi_\coarse}{\HH}\norm{\Psi_\coarse}{\HH}}.
 \end{align}
 In particular, the discrete formulation~\eqref{eq:discreteform:kompakt} admits a unique solution $\Phi_\coarse \in \XX_\coarse$
 and there
 exists $C>0$, which depends only on $\bilinb{\cdot}{\cdot}$ and $\beta$, but not on $\XX_\coarse$, such that 
 \begin{align}\label{eq:ssprop:cea}
 	\norm{\phi - \Phi_\coarse}{\HH} \leq C \min_{\Psi_\coarse \in \XX_\coarse} \norm{\phi - \Psi_\coarse}{\HH},
 \end{align}
 i.e., uniform 
 validity of the C\'ea lemma.
 If the spaces $\XX_\ell$ are nested, i.e., $\XX_{\ell} \subseteq \XX_{\ell +1}$ for all $\ell \in \N_0$, the latter guarantees convergence 
 $\norm{\phi - \Phi_\ell}{\HH} \to 0$ as $\ell \to \infty$. \qed
\end{proposition}

\subsection{Mesh-refinement}
\label{section:mesh}
From now on, suppose that $\TT_0$ is a given $\gamma$-shape regular triangulation of $\Gamma$.
For mesh-refinement, we consider 2D newest vertex bisection (NVB) for $d=3$ (see e.g.,~\cite{stevenson08}), 
or extended 1D bisection (EB) from~\cite{affkp} for $d=2$.
Given a $\gamma$-shape regular triangulation $\TT_\coarse$ and a set of marked elements $\MM_\coarse\subseteq\TT_\coarse$,
the call $\TT_\fine={\rm refine}(\TT_\coarse,\MM_\coarse)$ returns for both refinement strategies the coarsest refinement $\TT_\fine$ of $\TT_\coarse$ such that all $T\in\MM_\coarse$ have been refined, i.e.,
\begin{itemize}
\item $\MM_\coarse\subseteq\TT_\coarse\backslash\TT_\fine$,
\item the number of elements $\#\TT_\fine$ is minimal amongst all other refinements $\TT'$ of $\TT_{\coarse}$.
\end{itemize}
Furthermore, we write $\TT_\fine\in \refine(\TT_\coarse)$ if $\TT_\fine$ is obtained by a finite number of refinement steps, i.e., there exists $n\in\N_0$ as well as a finite sequence $\TT^{(0)},\dots,\TT^{(n)}$ of triangulations and corresponding sets $\MM^{(j)}\subseteq\TT^{(j)}$ such that
\begin{itemize}
\item $\TT_\coarse = \TT^{(0)}$,
\item $\TT^{(j+1)} = {\rm refine}(\TT^{(j)},\MM^{(j)})$ for all $j=0,\dots,n-1$,
\item $\TT_\fine = \TT^{(n)}$.
\end{itemize}
In particular, $\TT_\coarse\in{\rm refine}(\TT_\coarse)$. To abbreviate notation, we let $\T:={\rm refine}(\TT_0)$ be the set of all possible triangulations which can be obtained from the initial triangulation $\TT_0$. 

Both refinement strategies guarantee uniform $\gamma$-shape regularity of all $\TT_\coarse \in \T$, where $\gamma$ depends only on $\TT_0$. Hence, Lemma~\ref{lemma:mesh}
applies for any triangulation $\TT_\coarse \in \T$.
Moreover, for all $T \in \TT_\coarse$, it holds that $T = \bigcup\big\{T'\in\TT_\fine\,:\,T'\subseteq T\big\}$.
In the following, we recall further properties of these mesh-refinement strategies, which are exploited below. 

First, refining an element results in at least $2$ and at most $\Cson$ sons, where $\Cson =2$ for EB and $\Cson =4$ for NVB; 
see e.g.,~\cite{kpp} for NVB and~\cite[Section 3]{affkp} for EB.
In particular, it holds that
\begin{align}\label{mesh:sons}
\# (\TT_\coarse \setminus \TT_\fine) + \# \TT_\coarse \leq \# \TT_\fine \quad \text{for all} \,  \TT_\coarse \in \T \text{ and all } \TT_\fine \in \refine(\TT_\coarse).
\end{align}

Second, refinement of an element yields a contraction of the local mesh-size function. 
Even though the proof is found, e.g., in~\cite{gantnerphd}, we include it
for the sake of completeness.  

\begin{lemma}\label{mesh:lemma:meshsize}
There exist $0 < \qmesh <1$, such that for all $\TT_\coarse,\TT_\fine \in \T$ with $\TT_\fine \in \refine(\TT_\coarse)$,
it holds that $h_\fine|_T \leq \qmesh \, h_\coarse|_T$ on all $T \in \TT_\coarse \setminus \TT_\fine$. 
\end{lemma}

\begin{proof}
	We argue by contradiction. To this end, let $(\TT_\coarse^{n})_{n \in \N},(\TT_\fine^{n})_{n \in \N} \subset \T$ be sequences of
	refinements with $\TT_\fine^{n} \in \refine(\TT_\coarse^{n})$ and elements 
	$T_\coarse^n  \in \TT_\coarse^{n} \setminus \TT_\fine^{n}$ as well as $T_\fine^n  \in \TT_\fine^{n} \setminus  \TT_\coarse^{n}$ such that
	\begin{align*}
	 T_\fine^n \subsetneqq T_\coarse^n \quad \text{as well as} \quad \frac{|T_\coarse^n|}{|T_\fine^n|} \to 1 \quad \text{as } n \to \infty.
	\end{align*}  
	This implies that $|T_\coarse^n \setminus T_\fine^n| / |T_\coarse^n| \to 0$ as $n \to \infty$. 
	Further, for all $n \in \N$ there exists $T \in \TT_0$ such $T_\fine^{n} \subsetneqq T_\coarse^{n} \subseteq T$.
	We obtain a corresponding sequence $\widetilde{T}_\fine^n \subsetneqq \widetilde{T}_\coarse^n \subseteq \Tref$ 
	with $g_{T}(\widetilde{T}_\coarse^n) = T_\coarse^n$ as well as $g_{T}(\widetilde{T}_\fine^n) = T_\fine^n$.
	Since bisection is done at first on the reference element, it holds that $|\widetilde{T}_\fine^n|\leq |\widetilde{T}_\coarse^n| /2 $ for all $n \in \N_0$.
	Then, $\gamma$-shape regularity implies that $ | \det G_{T}(x) | \simeq ( h_\coarse(T))^{2(d-1)} = |T|^2$ for all $x \in \Tref$. This reveals the contradiction
		\begin{align*}
			\frac{1}{2} \leq \frac{|\widetilde{T}_\coarse^n \setminus \widetilde{T}_\fine^n|}{|\widetilde{T}_\coarse^n|} 
			\simeq 	\frac{ \int_{\widetilde{T}_\coarse^n \setminus \widetilde{T}_\fine^n } |\det G_{T}(t)|^{1/2}\, dt }{\int_{\widetilde{T}_\coarse^n } |\det G_{T}(t)|^{1/2} \, dt}
			= \frac{|T_\coarse^n \setminus T_\fine^n|}{|T_\coarse^n|} \stackrel{n \to \infty}{\longrightarrow} 0
		\end{align*}
	and this concludes the proof. 
\end{proof}

Third, for a sequence $( \TT_\ell)_{\ell \in \N_0}$ with $\TT_\ell = \refine(\TT_{\ell-1}, \MM_{\ell-1})$ 
for arbitrary $\MM_{\ell-1} \subseteq \TT_{\ell-1}$, EB and NVB satisfy the mesh-closure estimate 
\begin{align}\label{mesh:mesh-closure}
\# \TT_\ell - \# \TT_0 \leq \Cmesh \sum_{j =0}^{\ell -1} \# \MM_j \quad \text{for all} \, \ell \in \N,
\end{align}
where the constant $\Cmesh\ge1$ depends only on the initial mesh $\TT_0$. 
In particular,~\eqref{mesh:mesh-closure} guarantees that the number of additional refinements of elements, 
in order to avoid hanging nodes and preserve conformity (NVB) or to preserve $\gamma$-shape regularity (EB), 
does not dominate the number of marked elements. 
For newest vertex bisection, the mesh-closure estimate has first been proved for $d=2$ in~\cite{bdd} and later for $d \geq 2$ in \cite{stevenson08}. 
While both works require an additional admissibility assumption on $\TT_0$, \cite{kpp} proved that this condition is redundant for $d=2$. 
For EB,~\eqref{mesh:mesh-closure} is proved in~\cite[Theorem 2.3]{affkp}.

Finally, we recall the 
overlay-estimate; For all $\TT \in \T$ as well as $\TT_\coarse,\TT_\fine \in \refine(\TT)$
 there exists a common refinement $\TT_\coarse \oplus \TT_\fine \in \refine(\TT_\coarse) \cap \refine(\TT_\fine) \subseteq \refine(\TT)$, such that 
\begin{align}\label{mesh:overlay}
 \# (\TT_\fine \oplus \TT_\coarse) \leq \# \TT_\fine + \# \TT_\coarse - \# \TT.
\end{align}
For NVB, the proof is found in \cite{ckns,stevenson07}.  
For EB, the proof is trivial; see~\cite{affkp}.

\subsection{Residual \textsl{a posteriori} error estimator}
\label{subsection:errorestimator}
Let $\TT_\coarse \in \refine(\TT_0)$.
Suppose that $f \in H^1(\Gamma)$ and that the solution $\Phi_\coarse \in \PP^p(\TT_\coarse)$ of~\eqref{eq:discreteform:kompakt} exists. Recall that $ \slok: L^2(\Gamma) \to H^1(\Gamma)$. Therefore, we can compute for all $T \in \TT_\coarse$ the local refinement indicators 
$\eta_\coarse(T) := \norm{h_\coarse^{1/2} \, \nabla_{\Gamma} \, (\slok \Phi_\coarse - f)}{L^2(T)} \geq0$ as well as the corresponding \textsl{a posteriori}
error estimator 
\begin{align}\label{eq:estimator}
\eta_\coarse :=  \eta_\coarse(\TT_\coarse) \quad \text{with} \quad  
\eta_\coarse(\UU_\coarse) := \Big( \sum_{T \in \UU_\coarse} \eta_\coarse(T)^2 \Big)^{1/2}  \quad \text{for all } \, \UU_\coarse \subseteq \TT_\coarse.
\end{align}
For $\UU_\coarse \subseteq \TT_\coarse$, define
\begin{align*}
\bigcup \UU_\coarse := \{x \in \Gamma : \exists T \in \TT_\coarse, x \in T\}.
\end{align*}%
It holds that
$\eta_\coarse(\UU_\coarse) = \norm{h^{1/2}\, \nabla_{\Gamma} \, (\slok \Phi_\coarse - f)}{L^2(\bigcup \UU_\coarse )}$.
The error estimator~\eqref{eq:estimator} has first been proposed for \textsl{a posteriori} BEM error control for the 
weakly-singular integral equation in 2D in~\cite{cs95,cc96} and later in 3D in~\cite{cms}.

\subsection{Adaptive algorithm}
\label{subsection:algorithm}
Based on the error estimator $\eta_\coarse$
we consider the following algorithm, 
where the expanded making strategy in Step(iv)--(v) goes back to~\cite{helmholtz}.

\begin{algorithm}\label{algorithm}
\textsc{Input:} Parameters $0<\theta\le1$ and $\Cmark\ge1$ as well as initial triangulation $\TT_0$ with $\Phi_{-1}:=0\in \PP^p(\TT_0)$  and $\eta_{-1}:=1$.\\
\textsc{Adaptive loop:} For all $\ell=0,1,2,\dots$, iterate the following Steps~\rm{(i)--(vi)}:
\begin{itemize}
\item[\rm{(i)}] \framebox{{\tt If:}} \quad \eqref{eq:discreteform:kompakt} does not admit a unique solution in $\PP^p(\TT_\ell)$: 
	\begin{itemize}
 		\item Define $\Phi_\ell:=\Phi_{\ell-1}\in \PP^p(\TT_0)$ and $\eta_\ell:=\eta_{\ell-1}$.
 		\item Let $\TT_{\ell+1}:=\refine(\TT_{\ell},\TT_{\ell})$ be the uniform refinement of $\TT_\ell$,
 		\item Increase $\ell \to \ell+1$ and continue with Step~\rm{(i)}.
 	\end{itemize}
\medskip
\item[\rm{(ii)}] \framebox{{\tt Else:}} \quad compute the unique solution $\Phi_\ell\in \PP^p(\TT_\ell)$ to~\eqref{eq:discreteform:kompakt}.
\smallskip
\item[\rm{(iii)}] Compute the corresponding indicators $\eta_\ell(T)$ for all $T\in\TT_\ell$.
\item[\rm{(iv)}] Determine a set $\MM_\ell' \subseteq\TT_\ell$ of up to the multiplicative factor $\Cmark$ minimal cardinality such that $\theta\eta_\ell^2\le \eta_\ell(\MM_\ell')^2$. 
\item[\rm{(v)}] Find $\MM_\ell'' \subseteq\TT_\ell$ such that $\# \MM_\ell'' = \# \MM_\ell'$ as well as $h_\ell(T) \geq h_\ell(T')$ for all $T \in \MM_\ell'' $ and $T' \in \TT_\ell \setminus \MM_\ell''$. Define $\MM_\ell := \MM_\ell' \cup \MM_\ell''$.
\item[\rm{(vi)}] Generate $\TT_{\ell+1}:=\refine(\TT_\ell,\MM_\ell)$,  increase $\ell \to \ell+1$, and continue with Step~\rm{(i)}.
\end{itemize}
\textsc{Output:} Sequences of successively refined triangulations $\TT_\ell$, discrete solutions $\Phi_\ell$, and corresponding estimators $\eta_\ell$.
\end{algorithm}

\begin{remark}
\begin{explain}
\item Apart from Step~{\rm(i)} and Step~{\rm(v)}, Algorithm~\ref{algorithm} is the usual adaptive loop based on the D\"orfler marking strategy~\cite{doerfler} in Step~{\rm(iv)} as used, e.g., in~\cite{ckns,ffp,axioms} as well as~\cite{ffkmp:part1,ffkmp:part2,fkmp}.
\item While $\Cmark=1$ requires to sort the indicators and hence leads to log-linear effort, Stevenson~\cite{stevenson07} showed that $\Cmark=2$ allows to determine $\MM_\ell'$ in linear complexity.

\item Step~{\rm(v)} of Algorithm~\ref{algorithm} is called expanded D\"orfler marking and ensures $\norm{h_\ell}{L^\infty(\Gamma)} \to \infty$ as $\ell \to \infty$, see~\cite[Proposition 16]{helmholtz}.
In particular, Step~{\rm(v)} implies $\overline{\bigcup_{\ell \in \N_0} \PP^p(\TT_\ell)} = \H^{-1/2}(\Gamma)$. This guarantees definiteness and hence well-posedness of~\eqref{eq:cont_weakform:kompact} on the discrete limit space, i.e., \cite[Axiom (A5)]{helmholtz} is satisfied. 
\end{explain}

\end{remark}

\subsection{Properties of the error estimator}
\label{subsetcion:estimatorproperties}
The proof of convergence with optimal algebraic rates for the adaptive scheme relies on the following essential properties of the 
\textsl{a posteriori} error estimator. These, so-called \emph{ axioms of adaptivity} are found in~\cite[Section 2.3]{helmholtz} and 
slightly generalize those of~\cite{axioms}. This is due to the fact that we always have to  guarantee unique solvability of the discrete 
problem~\eqref{eq:discreteform:kompakt} in order to compute the corresponding error estimator.

\begin{lemma}
\label{lemma:axioms}
There exist $\Cstab, \Cred, \Crel, \Crel >0$ and $0<q_{\rm red}<1$
such that for all $\TT_\coarse\in \T$ and all $\TT_\fine \in {\rm refine}(\TT_\coarse)$ the following implication holds:
Provided that the discrete solutions $\Phi_\coarse\in \PP^p(\TT_\coarse)$ and $\Phi_\fine\in \PP^p(\TT_\fine)$ exist, 
there holds the following~{\rm (i)--(iv)}:
\begin{enumerate}
	
\item[{\bf (i)}] \textbf{Stability on non-refined element domains}
\begin{align}\label{eq:stability}
	\big|\eta_\fine(\TT_\fine \cap\TT_\coarse)-\eta_\coarse(\TT_\fine \cap \TT_\coarse)\big| \le C_{\rm stb}\,\|{\Phi_\fine -\Phi_\coarse}\|_{\H^{-1/2}(\Gamma)}.
\end{align}

\item[{\bf (ii)}] \textbf{Reduction on refined element domains}
\begin{align}\label{eq:reduction}
\eta_\fine(\TT_\fine \backslash \TT_\coarse)^2 \le q_{\rm red}\,\eta_\coarse(\TT_\coarse \backslash \TT_\fine)^2 + C_{\rm red} \,\|{\Phi_\fine-\Phi_\coarse}\|_{\H^{-1/2}(\Gamma)}^2.
\end{align}

\item[{\bf (iii)}] \textbf{Discrete reliability}
\begin{align}\label{eq:discrete:reliability}
\|{\Phi_\fine-\Phi_\coarse}\|_{\H^{-1/2}(\Gamma)} \le \Crel \,\beta_\fine^{-1}\,\eta_\coarse(\RR_{\coarse,\fine}),
\end{align}
where $\beta_\fine$ is the discrete inf-sup constant from~\eqref{eq:ssprop:infsup} on $\PP^p(\TT_\fine)$ 
and $\RR_{\coarse,\fine} := \omega_{\coarse} (\TT_\coarse \setminus \TT_\fine) \subseteq \TT_\coarse$.
In particular, it holds that $\TT_\coarse \backslash \TT_\fine \subseteq \RR_{\coarse,\fine}$ as well as $\#\mathcal{R}_{\coarse,\fine} \le \Crel \,\#(\TT_\coarse\backslash\TT_\fine)$.

\item[{\bf (iv)}] \textbf{Reliability}
\begin{align*}
\|{\phi-\Phi_\coarse}\|_{\H^{-1/2}(\Gamma)} \le \Crel'\,\eta_\coarse.
\end{align*}

\end{enumerate}
The involved constants $\Cstab, \Cred, \Crel, \Crel ,\qred>0$ depend only on the given data, the polynomial degree $p$, the initial mesh $\TT_0$, and $\gamma$-shape regularity.
\end{lemma}

For the finite element method (FEM), the validity of these \emph{axioms of adaptivity} is well-known; see~\cite{axioms} for problems in the frame of the Lax--Milgram lemma or~\cite{helmholtz} for compactly perturbed elliptic problems. For BEM, the verification of the \emph{stability} and \emph{reduction} axioms is more involved than for the FEM and requires novel inverse estimates in the spirit of Theorem~\ref{theorem:invest}: For the Laplace equation, we refer to~\cite{ffkmp:part1,fkmp,gantumur}
as well as the overview in~\cite{axioms}. 
For BEM for the Helmholtz equation, most of the proofs are similar to the Laplace-case. 
However, for the sake of completeness and to underline that Theorem~\ref{theorem:invest} is crucial, the proof of Lemma~\ref{lemma:axioms} is given in Appendix~\ref{appendix:estimator}.

\section{Optimal Convergence}
\label{section:optimalconvergence}

In this section, we prove linear convergence as well as optimal algebraic convergence rates for the sequence of \textsl{a~posteriori} error estimators, generated by Algorithm~\ref{algorithm}. 
According to Section~\ref{section:algorithm}, the error estimator as well 
as the mesh-refinement strategy satisfy all assumptions needed in order to apply the abstract framework from~\cite[Section~2]{helmholtz}. 

\subsection{Approximation class}
\label{subsection:approximationclass}
For $N \in \N$, we define the set of all refinements which have at most $N$ elements more than a given mesh $\TT$, i.e.,

\begin{align*}
	\T_N(\TT) :=& \big\{ \TT_\coarse \in \refine(\TT): \# \TT_\coarse - \# \TT\leq N \\
	& \qquad  \text{ and the unique solution } \Phi_\coarse \in \PP^p(\TT_\coarse) \text{ to~\eqref{eq:discreteform:kompakt} exists} \big\}.
\end{align*}
If $\T_N(\TT) = \emptyset$ for some $N \in \N$, then we set $\min_{\TT_\coarse \in \T_N(\TT)} \eta_\coarse =0$.
 For any $s>0$, we define the abstract approximation class
\begin{align}\label{eq:approxclass}
	\norm{\phi}{\Approx_s(\TT)} := \sup_{N \in \N_0} \big( (N+1)^s \min_{\TT_\coarse \in \T_N(\TT)}  \eta_\coarse \big),
\end{align}
where $\eta_\coarse$ denotes the estimator corresponding to the optimal mesh $\TT_\coarse \in \T_N$.
To abbreviate notation, we define $\T_N := \T_N(\TT_0)$ and $\norm{\phi}{\Approx_s}:= \norm{\phi}{\Approx_s(\TT_0)}$.
 That means that, if $\norm{\phi}{\Approx_s} < \infty$ for some $s>0$, 
then there exists a sequence of optimal meshes such that the corresponding 
estimator sequence decays at least like $\eta_\coarse = \OO\big( (\# \TT_\coarse)^{-s} \big)$.
 Note that, in general, the sequence of corresponding optimal spaces $\PP^p(\TT_\coarse)$ is not necessarily nested. 


\subsection{Optimal convergence rates}
\label{subsection:optiomalrates}
The next theorem is the main result of this work. It proves that Algorithm~\ref{algorithm} does not only lead to convergence as well as linear convergence of the sequence of solutions, 
but also guarantees optimal algebraic convergence rates for the sequence of \textsl{a posteriori} error estimators.

\begin{theorem}\label{theorem:optimal}
	Let $(\TT_\ell)_{\ell \in \N_0}$ and $(\eta_\ell)_{\ell \in \N_0}$ be the sequences of meshes and corresponding estimators produced by Algorithm~\ref{algorithm}. 
		Let $0 < \theta \leq 1$. Then, there exist constants $0 < \qlin < 1$ and $\Clin >0$  as well as $\llin \in \N_0$ 
		such that Algorithm~\ref{algorithm} guarantees that
		\begin{align}\label{eq:linearconvergence}
			\eta_{\ell + n} \leq \Clin \, \qlin^n \, \eta_\ell \quad \text{for all } \ell,n \in \N \text{ with } \ell \geq \llin .
		\end{align}
	In particular, this implies convergence 
	\begin{align}\label{eq:convergence}
		\lim_{\ell \to \infty} \norm{\phi - \Phi_\ell}{\H^{-1/2}(\Gamma)}= \lim_{\ell \to \infty} \eta_\ell = 0.
	\end{align} 
	Moreover, there exist $\lcea \in \N_{0}$ and $C_\ell \geq 1$ with $\lim_{\ell \to \infty} C_\ell = 1$ such that the 
	sequence of discrete solutions $\Phi_\ell \in \PP^p(\TT_\ell)$ satisfies that
	\begin{align}\label{eq:cea}
		\enorm{\phi - \Phi_\ell} \leq C_\ell \min_{\Psi_\ell \in \PP^p(\TT_\ell)} \enorm{\phi - \Psi_\ell}
		\quad \text{for all } \ell \geq \lcea. 
	\end{align}
	Moreover, there exists $\widehat{\beta}>0$, $\lopt>0$, as well as $\thetaopt := (1+\Cstab^2 \Crel^2 / \widehat{\beta})^{-1}$, such that for all $0< \theta < \thetaopt$ and all $s>0$, it holds that
	\begin{align}\label{eq:optimal}
		\norm{\phi}{\Approx_s} < \infty \quad \Longleftrightarrow \quad \exists \Copt>0 \,\, \forall \ell \geq \lopt \quad 
		\eta_\ell \leq \Copt\, (\# \TT_\ell - \# \TT_0 +1 )^{-s}.
	\end{align}
	The constant $\Copt$ depends only on $\lopt$, $\Cson$, $\TT_0$, $\theta$, $s$, and on the constants in Lemma~\ref{lemma:axioms}.
	
\end{theorem}

The proof follows ideas of~\cite[Section 4.3]{helmholtz}, where we exploit the estimator properties of Lemma~\ref{lemma:axioms}. 
For the sake of completeness 
and since Theorem~\ref{theorem:optimal} is the main result of the present work, a rigorous proof 
is given in Appendix~\ref{appendix:optimal_proof} and improves~\cite{helmholtz}.


\begin{remark}\label{remark:direct}
For the presentation, we focus on the model problem~\eqref{eq:modelproblem} for some indirect BEM. 
In the case of a direct boundary element approach, the model problem reads
\begin{align}\label{eq:remark:slok}
\slok \phi = (\dlok + \frac{1}{2} {\rm Id})\, g \quad \text{on } \Gamma,
\end{align}
where $g \in H^{1/2}(\Gamma)$ is the given Dirichlet data and $\phi = \partial_n u \in H^{-1/2}(\po)$ is the sought normal derivative of the solution $u \in H^1(\Omega)$
of the (equivalent) boundary value problem 
\begin{align*}
	- \Delta u - k^2 u = 0 \quad \text{in } \Omega \qquad \text{subject to} \qquad  u=g \quad \text{on } \Gamma.
\end{align*}
The implementation of the right-hand side requires to approximate $g \approx G_\coarse \in \SS^{p+1}(\TT_\coarse)$.
Suitable approximations $G_\coarse = I_\coarse g$ together with some local data oscillations which control the 
approximation error $\norm{g-G_\coarse}{H^{1/2}(\po)}$,
are discussed and analyzed for the Laplace problem in~\cite{ffkmp:part1}. Provided that $g \in H^1(\po)$, it is shown that the adaptive algorithm then still leads to optimal 
convergence behavior. Together with the present analysis, the results of~\cite{ffkmp:part1} transfer immediately to the direct boundary element approach~\eqref{eq:remark:slok}.
\end{remark}

\section{Hypersingular integral equation}
\label{section:hypsing}
In this section, we briefly comment on the extension of our analysis to
the hypersingular integral equation. 
In case of the Laplace equation ($k=0$), a proof of optimal algebraic convergence rates is found in~\cite{ffkmp:part2,gantumur}. 
Throughout this section, we additionally suppose that $\partial\Omega$ is connected.
The model problem reads as follows: 
Given $f \in H^{-1/2}_\star(\po) := \set{\phi \in H^{-1/2}(\po)}{\dual{\phi}{1}=0}$ and the hypersingualar operator $\hsok:= - \condint \dlpk$,
find $u \in H^{1/2}_\star(\po) := \set{v \in H^{1/2}(\po)}{\dual{1}{v}=0}$ such that
\begin{align}\label{eq:modelproblem:hypsing}
	\hsok u = f \quad \text{on } \po.
\end{align}
The proof of convergence as well as optimal convergence rates for the related ABEM follows 
similar to the one for the weakly singular integral equation. Therefore, we focus only on the differences and highlight the necessary modifications.
For the Laplace case $k=0$, we also refer to~\cite{ffkmp:part2}. 

\subsection{Framework}
The operator $\hsonull$ is symmetric and positive semi-definite on $H^{1/2}(\po)$: 
\begin{align*}
	\dual{\hsonull \, v}{w} = \dual{\hsonull \, w}{v}  \quad \text{and } \quad \dual{\hsonull \, v}{v}\geq 0 \quad \text{for all } v,w \in H^{1/2}(\po).
\end{align*}
Since $\partial\Omega$ is connected, the
kernel of $\hsonull$ are the constant functions, and the bilinear form
$\dual{\hsonull (\cdot)}{\cdot}$ provides a scalar product on $H_\star^{1/2}(\po)$. 
Hence, this can be expanded to
\begin{align*}
	\bilina{u}{v} := \dual{\hsonull \, v}{w} + \dual{1}{v} \dual{1}{w} \quad \text{for all } v,w \in H^{1/2}(\po),
\end{align*}
which is a scalar product on $H^{1/2}(\po)$. According to the Rellich compactness theorem, there holds the norm equivalence
$\enorm{v}^2:=  \bilina{v}{v} \simeq \norm{v}{H^{1/2}(\po)}^2$ for all $v \in H^{1/2}(\po)$.

For $k\neq0$, it is well-known that the hypersingular integral operator $W_k$ is invertible, if and only if $k^2$ is not an eigenvalue of the interior Neumann problem (see \cite[Proposition 2.5]{st_hypsing}), i.e., it holds that
\begin{align}\tag{\rm INP}\label{eq:inp}
    \forall u \in H^1(\Omega)\quad\Big(
	\Delta u = k^2u \text{ with } \condint u = 0 \text{ and} \int_\Gamma u\,dx=0 \quad \Longrightarrow \quad  u=0 \text{ in } \Omega
	\Big);
\end{align}
To ensure solvability, we assume throughout that $k^2$ satisfies~\eqref{eq:inp}.
On Lipschitz boundaries $\po$, the operator $\widetilde{\Kop}_{W_k} := \hsok - \hsonull: H^{1/2}(\po) \to H^{-1/2}(\po)$ is compact; see~\cite[Lemma 3.9.8]{sauterschwab}. Define $\dual{ \Kop_{W_k} \, v }{w}  := \dual{  \widetilde{\Kop}_{W_k} \, v }{w} -  \dual{1}{v} \dual{1}{w}$ for all $v,w \in H^{1/2}(\po)$. 
Note that, $\Kop_{W_k}: H^{1/2}(\po) \to H^{-1/2}(\po)$ is uniquely defined and compact. 
Reformulation of~\eqref{eq:modelproblem:hypsing} yields the following equivalent formulation: Given $f \in H^{-1/2}_\star(\po)$, find $u \in H^{1/2}(\po)$ such that 
\begin{align}\label{eq:weakformulation:kompakt:hypsing}
 \bilina{u}{v} + \dual{ \Kop_{W_k} \, u }{v}  = \dual{f}{v} \quad \text{ for all } v \in H^{1/2}(\po).
\end{align}
The corresponding discrete formulation of~\eqref{eq:weakformulation:kompakt:hypsing} reads as: Find $U_\coarse \in \SS^p(\TT_\coarse)$ such that
\begin{align}\label{eq:discreteformulation:kompakt:hypsing}
 \bilina{U_\coarse}{V_\coarse} + \dual{ \Kop_{W_k} \, U_\coarse }{V_\coarse} 
 =  \dual{f}{V_\coarse} \quad \text{ for all } V_\coarse \in \SS^{p}(\TT_\coarse).
\end{align}
Then, the weak formulation~\eqref{eq:weakformulation:kompakt:hypsing} together with its Galerkin formulation~\eqref{eq:discreteformulation:kompakt:hypsing} fits in the abstract framework of Proposition~\ref{prop:ss}. 
Hence, existence and uniqueness of solutions of~\eqref{eq:discreteformulation:kompakt:hypsing} are guaranteed in the sense of Proposition~\ref{prop:ss}. 

\subsection{Error estimator}
Analogously to $H^{\pm 1/2}_\star(\po)$, we define $L^2_\star(\po) := \set{\phi \in L^2(\po)}{\int_\Gamma \phi \, ds=0}$.
Let $\TT_\coarse \in \T := \refine(\TT_0)$ be a triangulation such that the corresponding discrete solution $U_\coarse \in \SS^{p}(\TT_\coarse)$ of~\eqref{eq:discreteformulation:kompakt:hypsing}
exists. Suppose that $f \in L^2_\star(\po)$. Then, the local contributions of the weighted-residual error estimator for the hypersingular integral equation are defined by
\begin{align}\label{eq:est:hypsing}
	\eta_\coarse(T) := \norm{h_\coarse^{1/2} (f -\hsok \, U_\coarse)}{L^2(T)} \quad \text{for all } T \in \TT_\coarse.
\end{align}
The proofs of the estimator properties in Lemma~\ref{lemma:axioms} (i.e., stability on non-refined domains, reduction on refined element domains,
 discrete reliability as well as reliability)
 are similar to the ones for the weakly-singular case and can be found in~\cite[Proposition 3.5]{ffkmp:part2}.
The main difference is the use of the inverse inequality~\eqref{eq:invest:hso} instead of~\eqref{eq:invest:slo}.

\subsection{Optimal convergence rates}
We apply Algorithm~\ref{algorithm} to the model problem~\eqref{eq:weakformulation:kompakt:hypsing}. 
Recall that the error estimator~\eqref{eq:est:hypsing} satisfies Lemma~\ref{lemma:axioms} and model problem~\eqref{eq:weakformulation:kompakt:hypsing}
fits in the abstract setting of~\cite{helmholtz}.  
Verbatim argumentation as for the weakly-singular case proves Theorem~\ref{theorem:optimal} for the hypersingular integral equation. For details, see~\cite{haberlphd}.
Theorem~\ref{theorem:optimal} guarantees linear convergence and optimal algebraic convergence rates for the estimator sequence generated by Algorithm~\ref{algorithm} for the hypersingular equation~\eqref{eq:modelproblem:hypsing}.

Similarly to Remark~\ref{remark:direct}, one may consider a direct formulation for the Neumann boundary-value problem. In this case, the model problem reads as follows: Given Neumann data $\phi \in H^{-1/2}(\po)$, find $u \in H^{1/2}(\po)$ such that
	\begin{align}\label{eq:hypsing:direct}
		\hsok u = (\frac{1}{2} {\rm Id} - \adlok)\, \phi \quad \text{on } \po.
	\end{align}
In practice, the implementation of the right-hand side requires to approximate $\phi \approx \Phi_\coarse \in \PP^{p-1}(\TT_\coarse)$. Provided that $\phi \in L^2(\po)$, a suitable approximation $\Phi_\coarse := \Pi_\coarse \phi$ is given by the $L^2$-orthogonal projection onto $\PP^{p-1}(\TT_\coarse)$. The local data oscillations which control the additional approximation error $\norm{\phi - \Phi_\coarse}{H^{-1/2}(\Gamma)}$ are discussed and analyzed in~\cite{ffkmp:part2} for the Laplace problem. There, it is shown that the adaptive algorithm then still leads to optimal convergence behavior. Together with the present analysis, the results of~\cite{ffkmp:part2} transfer immediately to the direct boundary element aproach~\eqref{eq:hypsing:direct}.\qed


\section{Numerical Experiments}
\label{section:numerics}

In this section, we present some numerical experiments for the 3D Helmholtz equation that underpin the theoretical findings
of this work. 
We use lowest-order BEM and consider $\XX_\coarse = \PP^0(\TT_\coarse)$ for the weakly-singular integral equation 
and $\XX_\coarse = \SS^1(\TT_\coarse)$ for the hypersingular equation. The numerical computations were done with help
of BEM++, which is an open-source Galerkin boundary element library. 

We consider sound-soft (exterior Dirichlet) and sound-hard (exterior Neumann) acoustic scattering problems 
in $\R^3 \setminus \Omega$, where $\Omega \subset \R^3$ denotes the scatterer. 
Let $a \in \R^3$ with $|a| = 1$ denote the directional vector of the incident wave. 
Then, the incident (plane-) wave is given by $u^{\operatorname{inc}} = \exp(ika \cdot x)$. 
Let $u^{\operatorname{scat}}$ be the scattered field and the resulting total field is defined by $u^{\operatorname{tot}} =  u^{\operatorname{inc}} + u^{\operatorname{scat}}$.

For the sake of simplicity, we restrict the numerical examples to an indirect approach, in which the solution is in the form 
of a layer potential with some unknown density. 
For the sound-soft scattering problem, we obtain: 
Find $u^{\operatorname{scat}} = \slpk (\phi)$ such that
\begin{align}\label{eq:soundsoft}
	\slok \, \phi &= g \quad \text{subject to } \quad 
	g = - u^{\operatorname{inc}} \quad \text{on } \po.
\end{align}
The indirect approach for the sound-hard reads:
Find $u^{\operatorname{scat}} = \dlpk (\phi)$ such that
\begin{align}\label{eq:soundhard}
	\hsok \, \phi &= g \quad \text{subject to } \quad 
	g =  - \partial_n u^{\operatorname{inc}} \quad \text{on } \po.
\end{align}

\subsection{Implementational issues}
\label{subsectoin:implementation}
In this subsection, we briefly comment on some of the 
challenges arising in the implementation of Algorithm~\ref{algorithm}. 
We use BEM++ to compute all involved potential and integral operators;
see~\cite{bempp,bempp2,bempp3} as well as {\tt https://bempp.com} for details on BEM++ and, in particular, on the implemented quadrature rules, 
which are based on~\cite{sauterschwab}. The arising discrete linear systems
\begin{align*}
	\dual{\slok \Phi_\coarse}{\Psi_\coarse} = \dual{f}{\Psi_\coarse} \quad \text{resp.} \quad \dual{\hsok \Phi_\coarse}{\Psi_\coarse} = \dual{f}{\Psi_\coarse}
\end{align*}
are solved by GMRES.
Preconditioning can be done by diagonal or multilevel additive Schwarz preconditioners;
see, e.g.,~\cite{MR2257114,abem+solve} for BEM for 
the Laplace problem. 
To reduce the cost for storage of the system matrices, BEM++ supports $\HH$-matrices.

We emphasize that in the numerical experiments the refinement indicators
$\eta_\coarse(T) = \norm{h_\coarse^{1/2} \, \nabla_{\Gamma} \, (\slok \Phi_\coarse - f)}{L^2(T)}$ cannot be computed exactly. 
Instead, on each mesh $\TT_\coarse$ we compute the indicators by constructing a discrete integral operator 
of higher order $\slok^{\PP^0 \to \PP^1}: \PP^0(\TT_\coarse) \to \PP^1(\TT_\coarse)$ and by approximating the 
residual by $\slok^{\PP^0 \to \PP^1} \Phi_\coarse - f_\coarse \in  \PP^1(\TT_\coarse)$, where $f_\coarse \in \PP^1(\TT_\coarse)$
denotes the $\TT_\coarse$-piecewise $L^2(\Gamma)$ best approximation. Then, we employ 
\begin{align*}
	\norm{h_\coarse^{1/2} \, \nabla_{\Gamma} \, (\slok^{\PP^0 \to \PP^1}\Phi_\coarse - f_\coarse)}{L^2(T)} \approx \eta_\coarse(T) \quad \text{for all } T \in \TT_\bullet,
\end{align*}
where the left-hand side is computed exactly by numerical quadrature.
Our implementation is restricted to lowest-order elements $\PP^0(\TT_\coarse)$ (and discontinuous piecewise affine elements $\PP^1(\TT_\coarse)$ for the computation of the residual). Due to generic edge singularities, however, the use of higher-order polynomials would not lead to an improved order of convergence, without using anistropic mesh-refinement.

The hypersingular integral equation is treated similarly.

\subsection{Sound-soft scattering on a L-shaped domain (non-convex case)}
%
As first numerical example, we consider a so called L-shaped domain in $(x,y)$-direction and expand it on the $z$-axis up to $[-1,1]$ (Figure~\ref{fig:ex1:inital_mesh}). 
We compare two directions of the incident wave. One with $a = (-1/\sqrt{2},1/\sqrt{2},0)^T$ (Figure~\ref{fig:ex1:setting}, left) hitting the scatterer on the non-convex part vs.\ 
$a = (1/\sqrt{2},-1/\sqrt{2},0)^T$ hitting the convex part of $\Omega$ (Figure~\ref{fig:ex1:setting}, right).
\vspace{-2mm}
\subsubsection{Non-convex case}
\label{ex:one:non}
First, we comment on the non-convex case. 
Figure~\ref{fig:ex1:compare_marking} (left) shows the convergence rate of $\eta_\ell^2$ for $k=1$ and different marking strategies. 
We compare uniform refinement to standard D\"orfler marking (i.e., with Algorithm~\ref{algorithm}
without Step~(v)) as well as to the expanded D\"orfler marking (Algorithm~\ref{algorithm}), both using $\theta = 0.4$.
The experiments show that uniform mesh-refinement leads to a suboptimal rate  
$\eta_\ell^2 = \OO(N^{-2/3})$, while 
adaptive refinement with Algorithm~\ref{algorithm} leads to the improved rate 
$\OO(N^{-\delta})$ with $\delta = 1.075$, independently of the actual marking. 
Empirically, the results generated by employing the standard D\"orfler
marking are of no difference compared to the results generated by employing the expanded D\"orfler marking. 
The same observation is made for all computations (not displayed). 
While the optimal rate for a smooth solution (without edge singularities) is $\OO(N^{-3/2})$, we
	refer to a heuristic argument from~\cite[Section 7.3]{cmps} that, in the presence of edge singularities, isotropic mesh-refinement can only lead to a reduced order of $\OO(N^{-1})$, which is, in fact, observed here. 

Figure~\ref{fig:ex1:compare_marking} (right) compares uniform vs.\ adaptive refinement for fixed $\theta = 0.4$ but various $k \in \{1,2,4,8,16\}$. 
As expected, the preasymptotic phase increases with $k$, but adaptive mesh-refinement asymptotically regains improved convergence rates for every $k$.
For $k=8$ and $k=16$, the last mesh of the preasymptotic phase is marked with a black symbol. 
Table~\ref{table:epwl} displays the number of elements per wavelength, when asymptotic convergence behavior kicks in.  

\begin{table}[h!]
\begin{tabular}{|c||c|c|c|c||c|c|c|c|}

\hline
$k$ & $\# \TT_\coarse$ & $\max$ & $\min $ & el.\ per $\lambda$ & $\# \TT_\coarse$ & $\max$ & $\min$ & el.\  per $\lambda$ \\
\hline\hline
& \multicolumn{4}{|c||}{non-convex case} & \multicolumn{4}{|c|}{convex case}\\
\hline
$16$-uniform & 896 & 0.125 & 0.125&	3.14 & 896 & 0.125 & 0.125 & 3.14\\
\hline
$16$-adaptive & 198 & 0.5  & 0.125 & 1.57  & 442 & 0.5 & 0.0625 & 1.57 \\
\hline
$8$-uniform	& 224 &  0.25  & 0.25 & 3.14 &224 &  0.25 &0.25  & 3.14  \\
\hline
$8$-adaptive & 184 & 0.5 & 0.125 & 1.57& 126 & 0.5 & 0.25& 1.57 \\
\hline
$4,2,1$ & 56 & 0.5  & 0.5  & 3.14 & 56 & 0.5 & 0.5 & 3.14  \\ 
\hline

\end{tabular}

\caption{Ex.~\ref{ex:one:non} and Ex.~\ref{ex:one:conv}. Number of elements per wavelength (el.\ per $\lambda$), on the surface-part hit by the incoming wave, for the last mesh of the preasymptotic phase. The corresponding meshes are marked by black symbols in Figure~\ref{fig:ex1:compare_marking} and Figure~\ref{fig:ex1:compare_pp_marking}. 
 Here $\max$ and $\min$ denote the maximal and minimal diameter in $(x,y)$-direction. 
 For adaptive meshes, the number of elements per $\lambda$ is computed with the maximal diameter.
  For $k \leq 4$ the asymptotic phase starts with $\TT_0$. 
  For $k=16$ the coarsest uniform refinement with $6$ elements per wavelength has $1792$ elements.   
}
\label{table:epwl}
\end{table}

Figure~\ref{fig:ex1:compare_theta} compares the convergence of the estimator for different values of the marking parameter $\theta \in \{0.2,0.4,0.6,0.8\}$ as well as for uniform mesh-refinement.
Again, uniform mesh-refinement leads to a suboptimal rate of convergence for the error estimator, while 
adaptive refinement with Algorithm~\ref{algorithm} regains the improved rate of convergence, independently of the actual choice of the marking parameter. 
Although Theorem~\ref{theorem:optimal} predicts optimal convergence rates only for small marking parameters  $0 < \theta < \thetaopt :=(1+\Cstab^2\Crel^2/ \widehat{\beta})^{-1}$,
 we observe that Algorithm~\ref{algorithm} is stable 
 with respect to $\theta$, where we tested $\theta \in \{0.2,\ldots,0.8 \}$.
In Figure~\ref{fig:ex1:meshes}, one can see some of the obtained adaptive meshes $\TT_\ell$ with $\ell=4,8,12$.  
 Since Algorithm~\ref{algorithm} refines mostly elements with big error indicator, 
	the refinement is essentially focused on the edges hit by the incoming wave, while facets in the shadow essentially remain coarse. 
	If the elements along the edges are sufficiently fine,
	for some adaptive steps the elements in the plain (smooth) surface parts contain the largest estimator
	contributions; see $\TT_{12}$ in Figure~\ref{fig:ex1:meshes} (right). 

Figure~\ref{fig:ex1:compare_cond} illustrates the condition number of the arising linear system in~\eqref{eq:discreteform:kompakt}.
As expected, the condition number grows with progressing mesh-adaptation, but stays bounded 
in the first couple of steps. This indicates that the linear system in the discrete formulation~\eqref{eq:discreteform:kompakt} allows for a unique solution for
every $\ell \in \N$. Hence, Algorithm~\ref{algorithm} never enforced uniform mesh-refinement in Step~{\rm{(i)}}. 
\vspace{-2mm}
\subsubsection{Convex case}
\label{ex:one:conv}
In the second case, the scatterer is hit on the convex part of the domain (Figure~\ref{fig:ex1:setting}, right), 
we compute very similar results as in the non-convex case. 
As shown in Figure~\ref{fig:ex1:compare_pp_marking}, the expanded as well as the standard D\"orfler marking both 
lead to improved rates of $\eta_\ell^2 = \OO(N^{-1.06})$, 
while uniform refinement leads only to $\OO(N^{-2/3})$.
 The rate of convergence is independent of the wavelength $k>0$, but increasing $k$ leads to a longer preasymptotic phase. 
Figure~\ref{fig:ex1:meshes_pp} shows the triangulation $\TT_{16}$. 
Again, the mesh-refinement is focused around the facets and edges hit by the incoming wave and the facets in the shadow essentially remain coarse.

%

%
%

\begin{figure}
  \centering
  \includegraphics[width=0.40\textwidth]{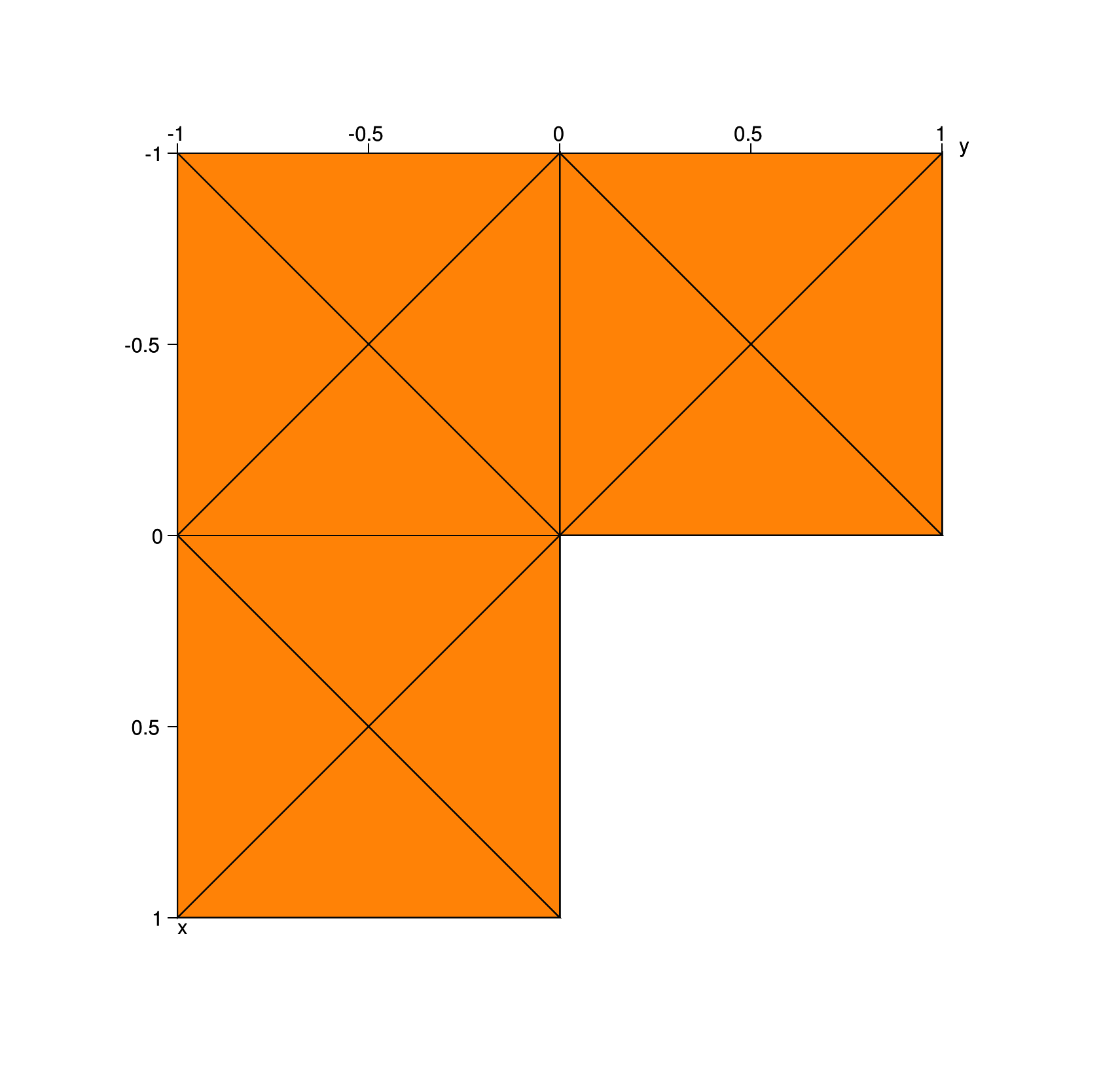}
  \includegraphics[width=0.40\textwidth]{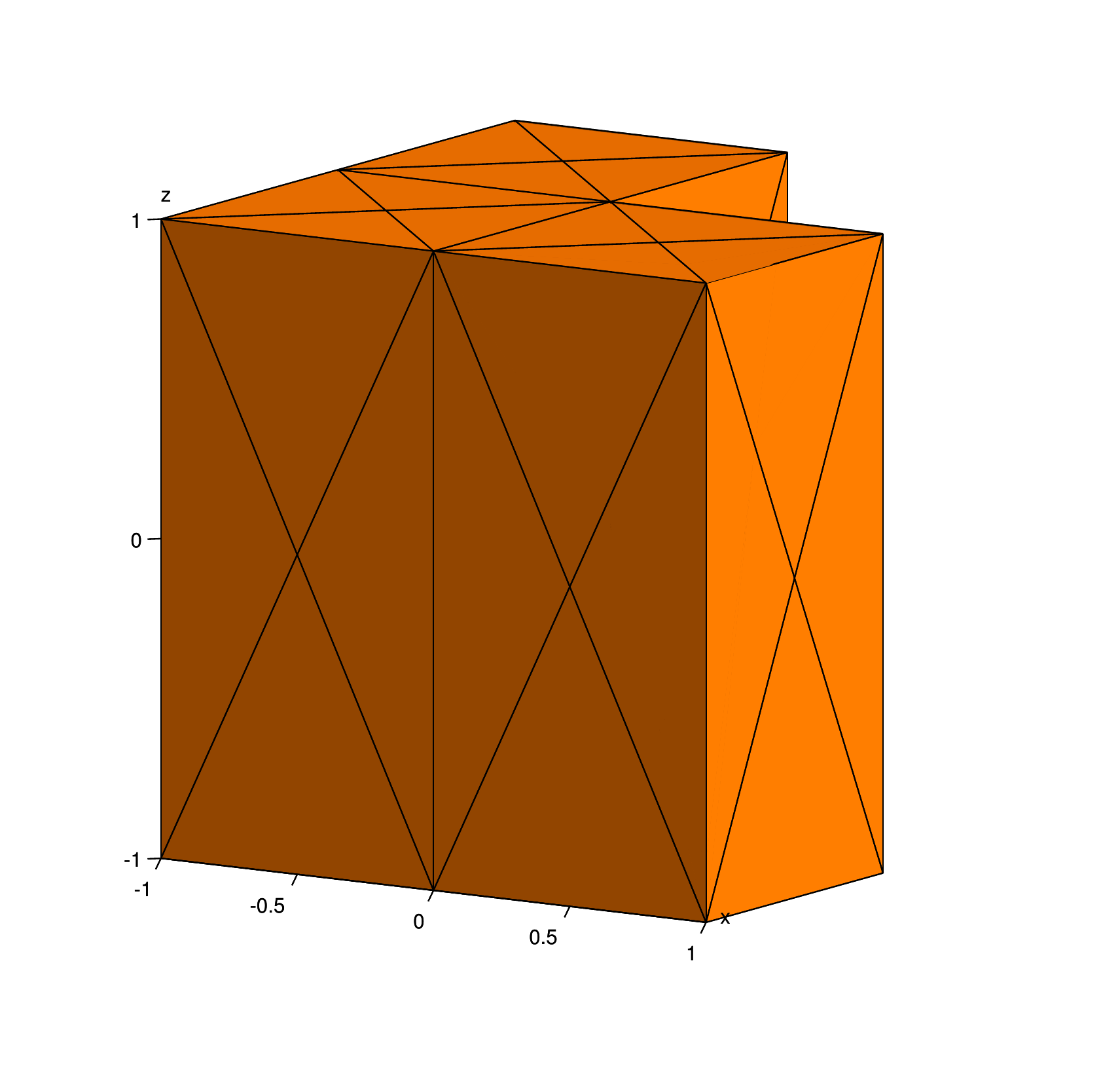}

  \caption{Geometry and initial mesh $\TT_0$ with $56$ elements (left: top view, right: 3D view). The reentrant edge 
   is given by $\{(x,y,z): x=y=0, \, z \in [-1,1]  \} $. 
 }
  \label{fig:ex1:inital_mesh}
\end{figure}

\begin{figure}
  \centering
  \adjincludegraphics[width=0.48\textwidth,Clip={.0\width} {.05\height} {0.0\width} {.125\height}]{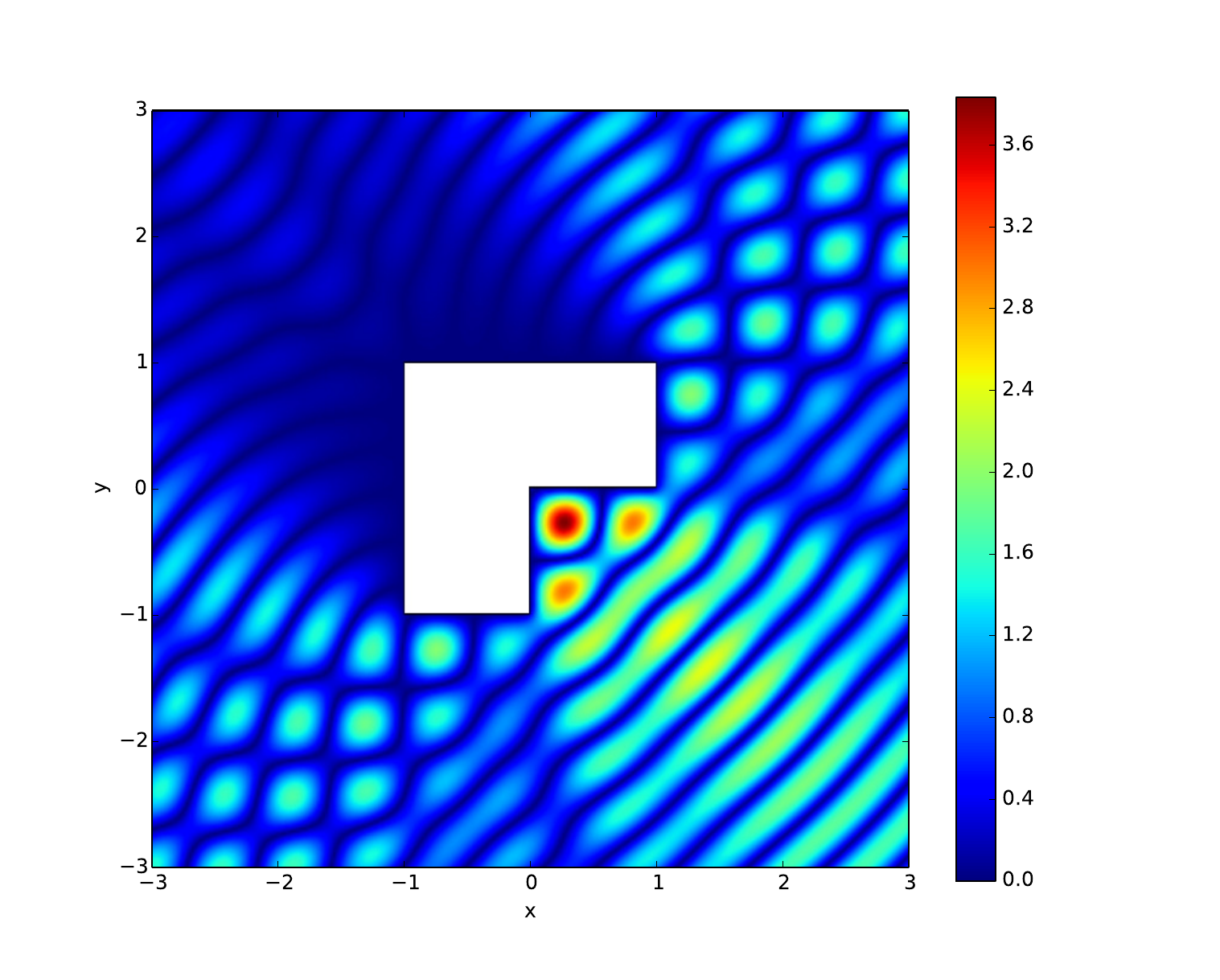}
	\hfill
	\adjincludegraphics[width=0.48\textwidth,Clip={.0\width} {.05\height} {0.0\width} {.125\height}]{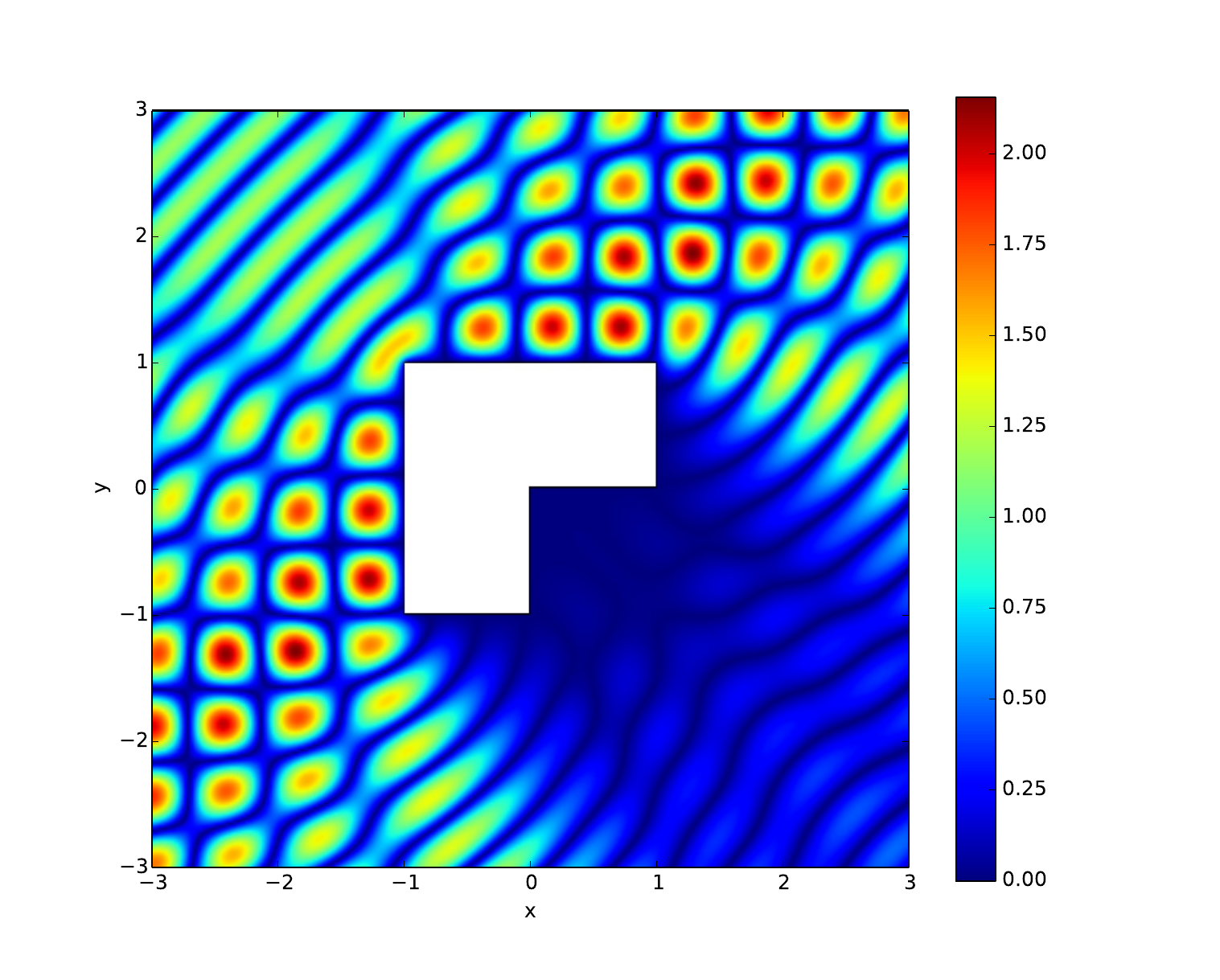}
  \caption{Ex.~\ref{ex:one:non} and Ex.~\ref{ex:one:conv}: Total field $u^{\operatorname{tot}}$ at the plane $z=0$ for different directions of $u^{\operatorname{inc}}$ with $k=8$.
  	The incident wave hits the scatterer on the non-convex part (left) and on the convex part (right).}
  \label{fig:ex1:setting}
\end{figure}

%

\begin{figure}


\begin{tikzpicture}
\begin{loglogaxis}[
 width=0.48\textwidth, height=9.5cm,
legend style={at={(0.02,0.02)},anchor=south west, draw = none },
legend cell align={left},
xlabel={number of elements},
ylabel={estimator},
]
	\addplot+[solid,mark=square,mark size=2pt,mark options={line width=1.0pt},color=red] table
		[x=number_of_elements,y=squared]{figures/Lshape_mm_k1_ada.csv};
	\addlegendentry{D\"orfler}

	\addplot+[solid,mark=triangle*,mark size=2pt,mark options={line width=1.0pt},color=green] table
		[x=number_of_elements,y=squared]{figures/Lshape_em_mm_k1_ada.csv};
	\addlegendentry{expanded D\"orfler}

	\addplot+[solid,mark=*,mark size=2pt,mark options={line width=1.0pt},color=blue] table
		[x=number_of_elements,y=squared]{figures/Lshape_mm_k1_unif.csv};
	\addlegendentry{uniform}

	\addplot [black,dashed ] expression [domain=56:142000, samples = 10] {30*x^(-2/3)} node [midway,above,yshift=+0.40cm] {$\mathcal{O}(N^{-2/3})$};
	
	\addplot [black,dashed ] expression [domain=56:142000, samples = 10] {40*x^(-1.075)} node [midway,below,yshift=-0.25cm, xshift=-0.4cm] {$\mathcal{O}(N^{-\delta})$};

\end{loglogaxis}
\end{tikzpicture}
\hfill
\begin{tikzpicture}
\begin{loglogaxis}[
 width=0.48\textwidth, height=9.5cm,
legend style={at={(0.02,0.02)},anchor=south west, align=left, draw = none },
legend cell align={left},
xlabel={number of elements},
ylabel={estimator},
]

	\addplot+[solid,mark=square,mark size=2pt,mark options={line width=1.0pt},color=red] table
		[x=number_of_elements,y=squared]{figures/Lshape_em_mm_k1_ada.csv};
	\addlegendentry{$k=1$}

	\addplot+[solid,mark=square,mark size=2pt,mark options={line width=1.0pt},color=blue] table
		[x=number_of_elements,y=squared]{figures/Lshape_em_mm_k2_ada.csv};
	\addlegendentry{$k=2$}

	\addplot+[solid,mark=square,mark size=2pt,mark options={line width=1.0pt},color=green] table
		[x=number_of_elements,y=squared]{figures/Lshape_em_mm_k4_ada.csv};
	\addlegendentry{$k=4$}

	\addplot+[solid,mark=square,mark size=2pt,mark options={line width=1.0pt},color=yellow] table
		[x=number_of_elements,y=squared]{figures/Lshape_em_mm_k8_ada.csv};
	\addlegendentry{$k=8$}

	\addplot+[solid,mark=square,mark size=2pt,mark options={line width=1.0pt},color=pink] table
		[x=number_of_elements,y=squared]{figures/Lshape_em_mm_k16_ada.csv};
	\addlegendentry{$k=16$}

	\addplot+[solid,mark=*,mark size=2pt,mark options={line width=1.0pt},color=blue] table
		[x=number_of_elements,y=squared]{figures/Lshape_mm_k2_unif.csv};
	
	\addplot+[solid,mark=*,mark size=2pt,mark options={line width=1.0pt},color=red] table
		[x=number_of_elements,y=squared]{figures/Lshape_mm_k1_unif.csv};

	\addplot+[solid,mark=*,mark size=2pt,mark options={line width=1.0pt},color=green] table
		[x=number_of_elements,y=squared]{figures/Lshape_mm_k4_unif.csv};
	
	\addplot+[solid,mark=*,mark size=2pt,mark options={line width=1.0pt},color=yellow] table
		[x=number_of_elements,y=squared]{figures/Lshape_mm_k8_unif.csv};

	\addplot+[solid,mark=*,mark size=2pt,mark options={line width=1.0pt},color=pink] table
		[x=number_of_elements,y=squared]{figures/Lshape_mm_k16_unif.csv};

	\addplot [black,dashed ] expression [domain=56:142000, samples = 10] {40*x^(-1.075)} node [midway,below,yshift=-1.0cm, xshift=0.8cm] {$\mathcal{O}(N^{-\delta})$};
	
		\addplot [mark=square*,mark size=3pt] coordinates { (208,277) };
			
		\addplot [mark=square*,mark size=3pt] coordinates { (178,46.13) };

		\addplot [mark=*,mark size=3pt] coordinates { (896,156.8916) };
					
		\addplot [mark=*,mark size=3pt] coordinates { (224,52.8074) };

\end{loglogaxis}

\end{tikzpicture}




			



\caption{Ex.~\ref{ex:one:non}: Convergence of $\eta_\ell^2$ for standard D\"orfler marking vs.\ expanded D\"orfler and uniform refinement with $k=1$ (left). 
Expanded D\"orfler marking (squares) vs.\ uniform refinement (circles) for different values of $k>0$ (right). Both plots are computed with $\theta = 0.4$. 
The black symbols mark last meshes of the preasymptotic phase for $k=8$ and $k=16$. 
}
  \label{fig:ex1:compare_marking}

\end{figure}
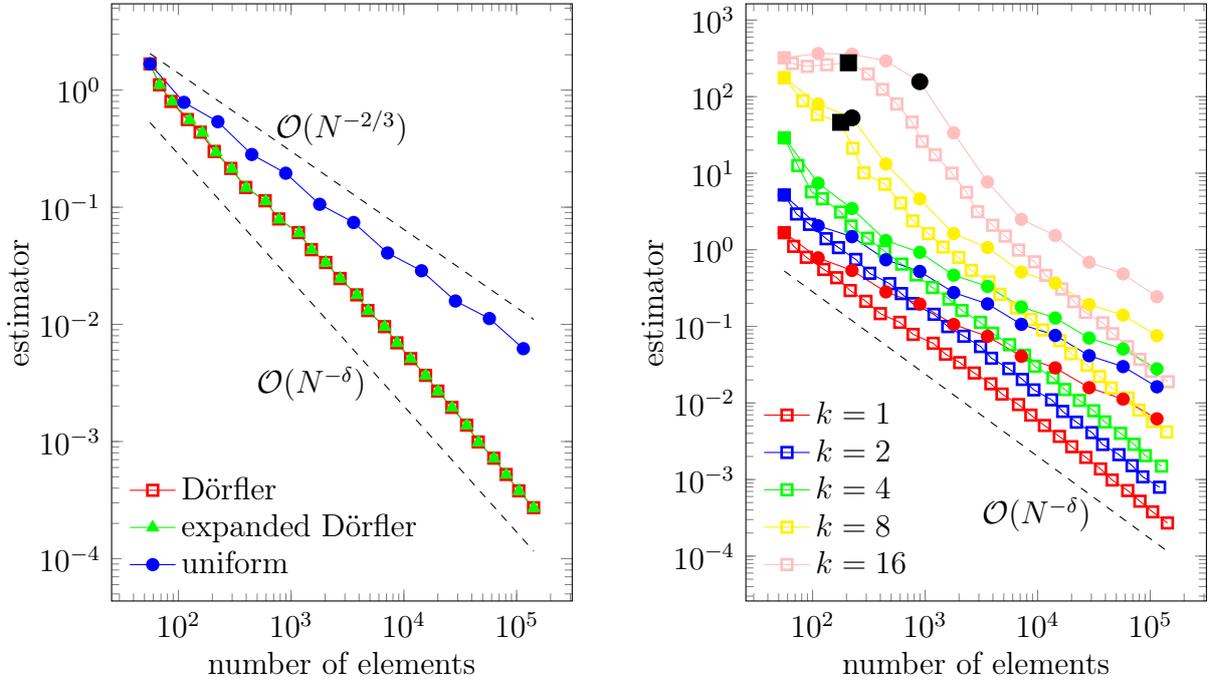

\begin{figure}
	\centering
	\adjincludegraphics[width=0.32\textwidth,Clip={.10\width} {.1\height} {0.0\width} {.05\height}]{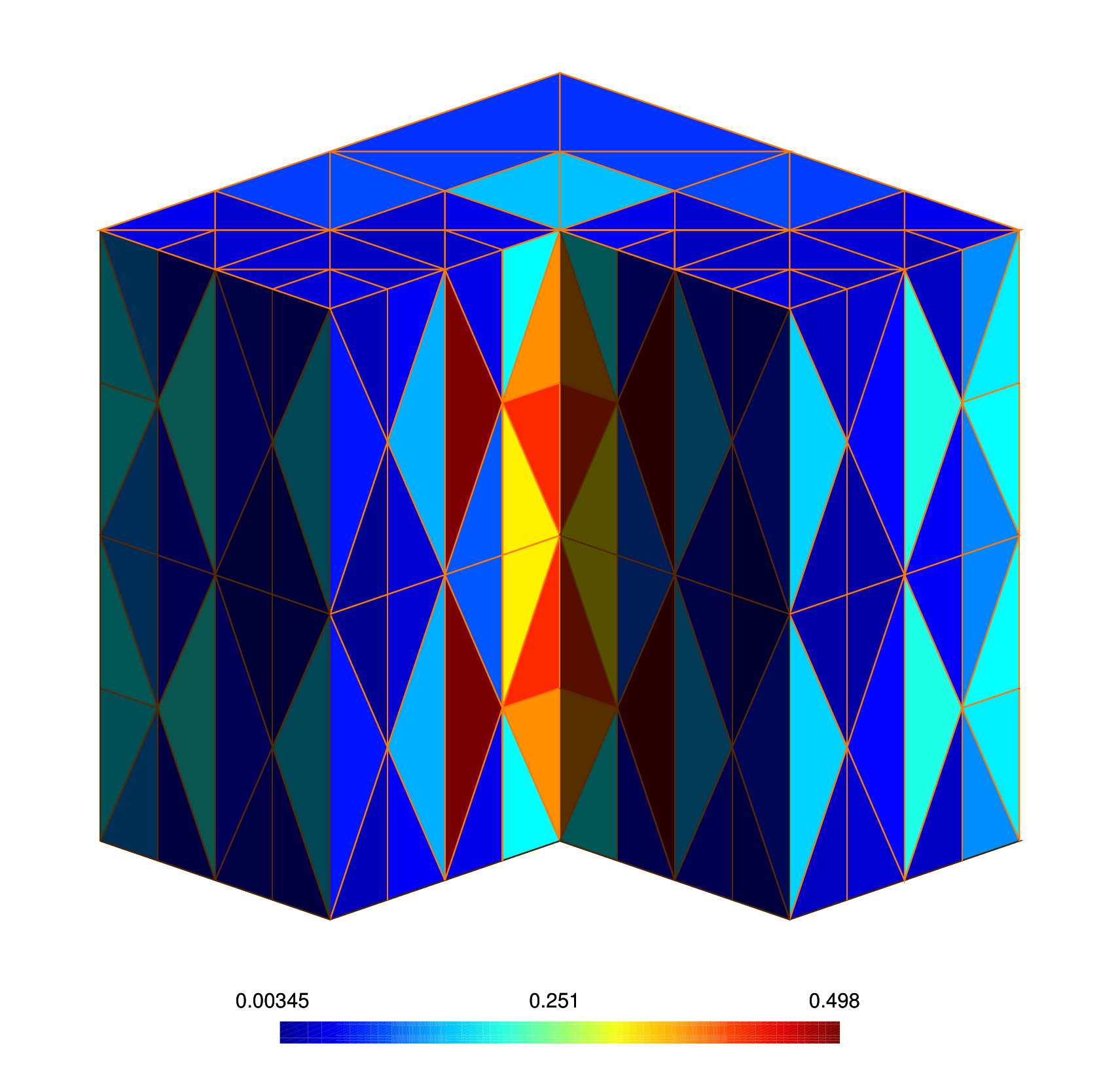}
	\adjincludegraphics[width=0.32\textwidth,Clip={.10\width} {.1\height} {0.0\width} {.05\height}]{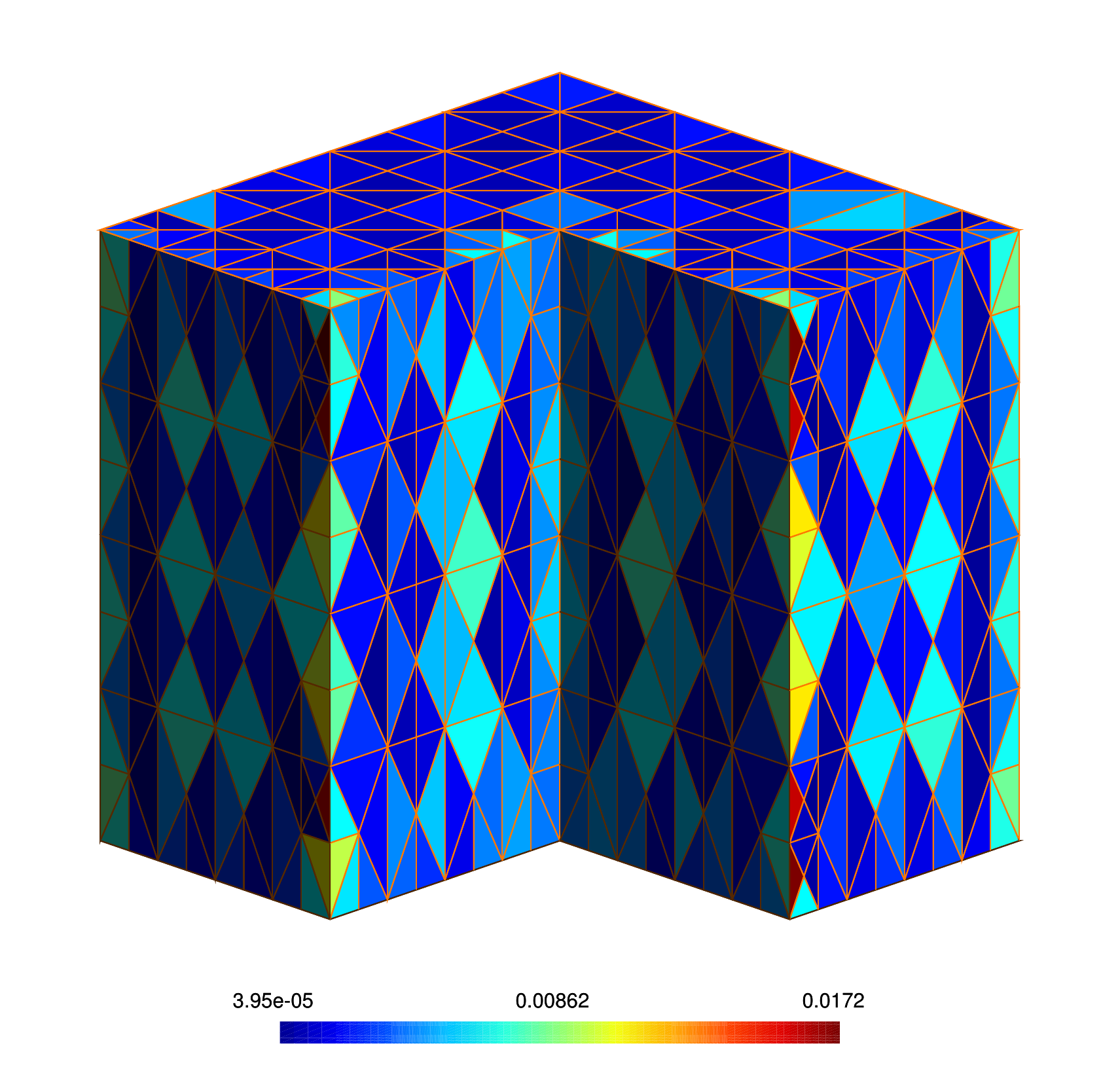}
	\adjincludegraphics[width=0.32\textwidth,Clip={.10\width} {.1\height} {0.0\width} {.05\height}]{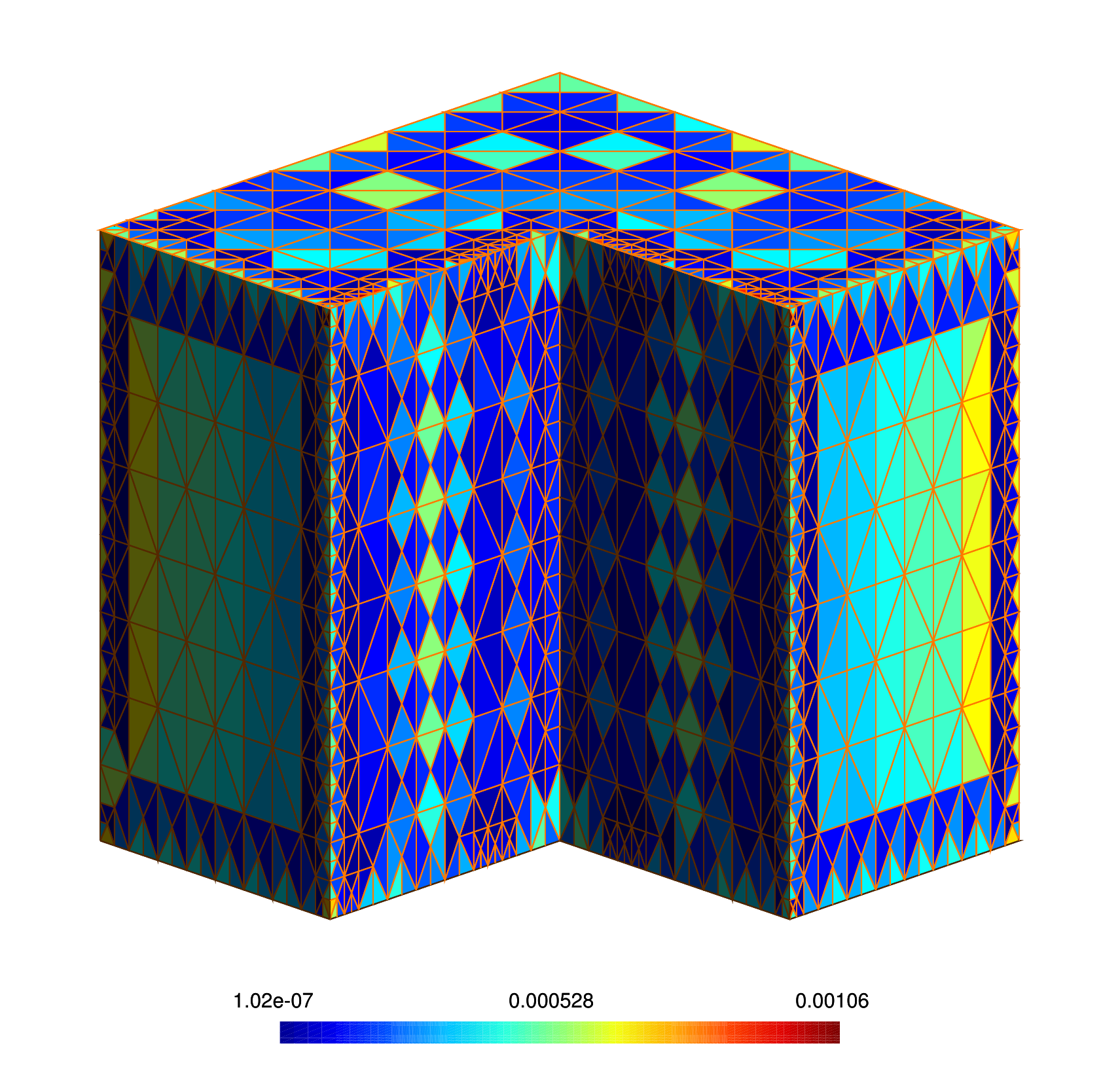}
	\caption{Ex.~\ref{ex:one:non}: Triangulations $\TT_4$, $\TT_8$ and $\TT_{12}$ with $208,766$ and $2332$ elements. 
	The color indicates the element contribution of the error estimator $\eta_\ell(T)^2$ for all $T \in \TT_\ell$.}
	\label{fig:ex1:meshes}	  
\end{figure}

%
%
%
%

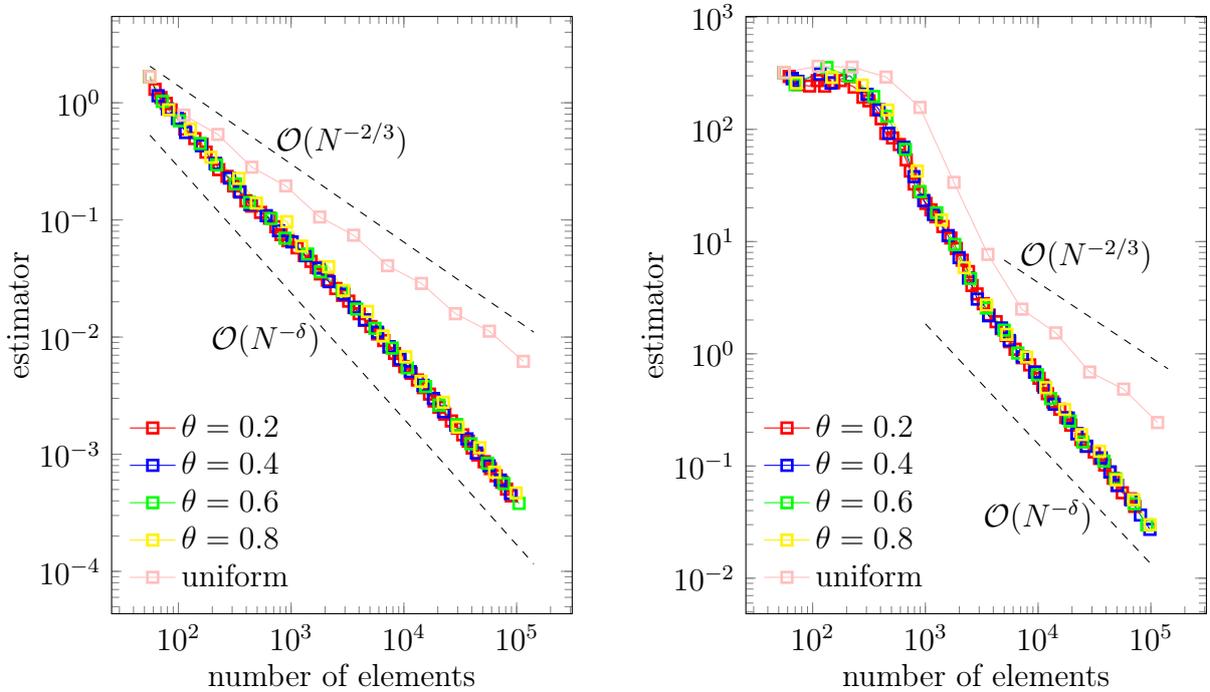
\begin{figure}


\begin{tikzpicture}
\begin{loglogaxis}[
 width=0.48\textwidth, height=9.5cm,
legend style={at={(0.02,0.02)},anchor=south west, align=left, draw = none },
legend cell align={left},
xlabel={number of elements},
ylabel={estimator},
]
	\addplot+[solid,mark=square,mark size=2pt,mark options={line width=1.0pt},color=red] table
		[x=number_of_elements,y=squared]{figures/Lshape_mm_k1_theta_02.csv};
	\addlegendentry{$\theta=0.2$}

	\addplot+[solid,mark=square,mark size=2pt,mark options={line width=1.0pt},color=blue] table
		[x=number_of_elements,y=squared]{figures/Lshape_mm_k1_theta_04.csv};
	\addlegendentry{$\theta=0.4$}

	\addplot+[solid,mark=square,mark size=2pt,mark options={line width=1.0pt},color=green] table
		[x=number_of_elements,y=squared]{figures/Lshape_mm_k1_theta_06.csv};
	\addlegendentry{$\theta=0.6$}

	\addplot+[solid,mark=square,mark size=2pt,mark options={line width=1.0pt},color=yellow] table
		[x=number_of_elements,y=squared]{figures/Lshape_mm_k1_theta_08.csv};
	\addlegendentry{$\theta=0.8$}

	\addplot+[solid,mark=square,mark size=2pt,mark options={line width=1.0pt},color=pink] table
		[x=number_of_elements,y=squared]{figures/Lshape_mm_k1_theta_1.csv};
	\addlegendentry{uniform}

	\addplot [black,dashed ] expression [domain=56:142000, samples = 10] {40*x^(-1.075)} node [midway,below,yshift=0.50cm, xshift=-1cm] {$\mathcal{O}(N^{-\delta})$};
	
	\addplot [black,dashed ] expression [domain=56:142000, samples = 10] {30*x^(-2/3)} node [midway,above,yshift=+0.40cm] {$\mathcal{O}(N^{-2/3})$};

\end{loglogaxis}

\end{tikzpicture}
\hfill
\begin{tikzpicture}
\begin{loglogaxis}[
 width=0.48\textwidth, height=9.5cm,
legend style={at={(0.02,0.02)},anchor=south west, align=left, draw = none },
legend cell align={left},
xlabel={number of elements},
ylabel={estimator},
]
	\addplot+[solid,mark=square,mark size=2pt,mark options={line width=1.0pt},color=red] table
		[x=number_of_elements,y=squared]{figures/Lshape_mm_k16_theta_02.csv};
	\addlegendentry{$\theta=0.2$}

	\addplot+[solid,mark=square,mark size=2pt,mark options={line width=1.0pt},color=blue] table
		[x=number_of_elements,y=squared]{figures/Lshape_mm_k16_theta_04.csv};
	\addlegendentry{$\theta=0.4$}

	\addplot+[solid,mark=square,mark size=2pt,mark options={line width=1.0pt},color=green] table
		[x=number_of_elements,y=squared]{figures/Lshape_mm_k16_theta_06.csv};
	\addlegendentry{$\theta=0.6$}

	\addplot+[solid,mark=square,mark size=2pt,mark options={line width=1.0pt},color=yellow] table
		[x=number_of_elements,y=squared]{figures/Lshape_mm_k16_theta_08.csv};
	\addlegendentry{$\theta=0.8$}

	\addplot+[solid,mark=square,mark size=2pt,mark options={line width=1.0pt},color=pink] table
		[x=number_of_elements,y=squared]{figures/Lshape_mm_k16_theta_1.csv};
	\addlegendentry{uniform}

	\addplot [black,dashed ] expression [domain=1000:100000, samples = 10] {3000*x^(-1.07)} node [midway,below,yshift=-0.60cm] {$\mathcal{O}(N^{-\delta})$};
	
	\addplot [black,dashed ] expression [domain=5000:142000, samples = 10] {2000*x^(-2/3)} node [midway,above,yshift=+0.40cm] {$\mathcal{O}(N^{-2/3})$};

\end{loglogaxis}

\end{tikzpicture}




			



\caption{Ex.~\ref{ex:one:non}: Convergence of $\eta_\ell^2$ for different values of $\theta \in \{0.2,0.4,0.6,0.8\}$ as well as for uniform refinement. 
Both plots use expanded D\"orfler marking with $k=1$ (left) and $k=16$ (right). }
 \label{fig:ex1:compare_theta}

\end{figure}

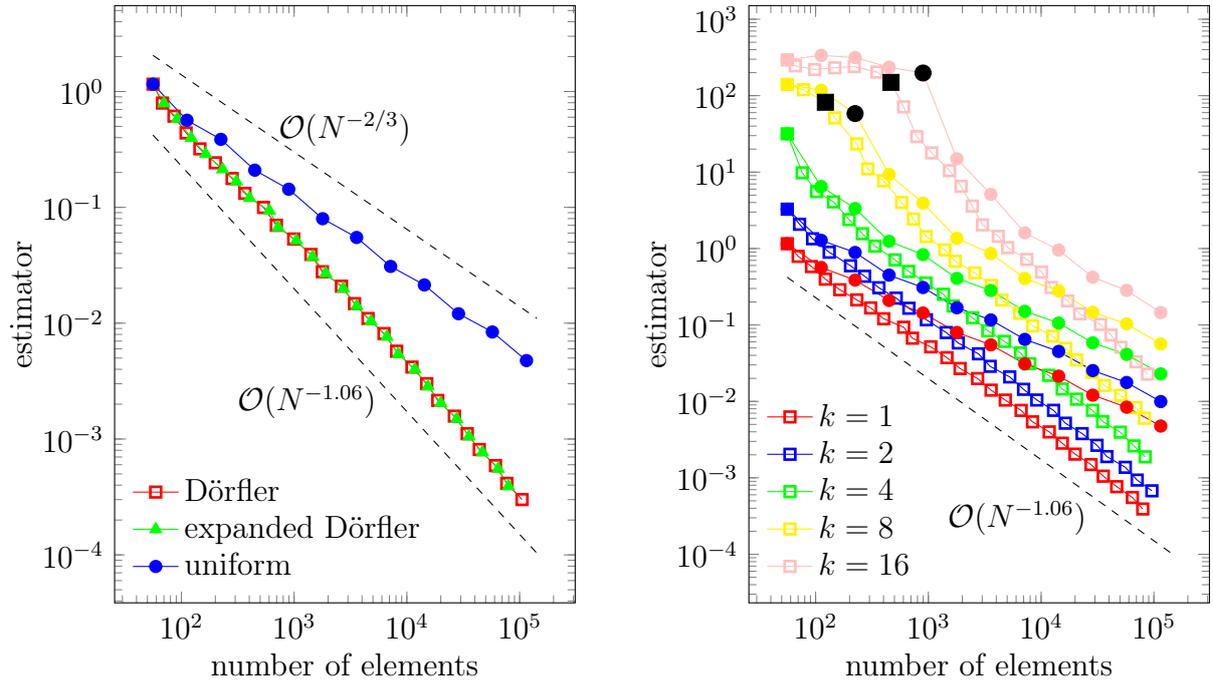
\begin{figure}


\begin{tikzpicture}
\begin{loglogaxis}[
 width=0.48\textwidth, height=9.5cm,
legend style={at={(0.02,0.02)},anchor=south west, align=left, draw = none },
legend cell align={left},
xlabel={number of elements},
ylabel={estimator},
]
	\addplot+[solid,mark=square,mark size=2pt,mark options={line width=1.0pt},color=red] table
		[x=number_of_elements,y=squared]{figures/Lshape_pp_k1_theta04.csv};
	\addlegendentry{D\"orfler}

	\addplot+[solid,mark=triangle*,mark size=2pt,mark options={line width=1.0pt},color=green] table
		[x=number_of_elements,y=squared]{figures/Lshape_pp_k1_expanded_theta04.csv};
	\addlegendentry{expanded D\"orfler}

	\addplot+[solid,mark=*,mark size=2pt,mark options={line width=1.0pt},color=blue] table
		[x=number_of_elements,y=squared]{figures/Lshape_pp_k1_theta1.csv};
	\addlegendentry{uniform}

	\addplot [black,dashed ] expression [domain=56:142000, samples = 10] {30*x^(-2/3)} node [midway,above,yshift=+0.40cm] {$\mathcal{O}(N^{-2/3})$};
	
	\addplot [black,dashed ] expression [domain=56:142000, samples = 10] {30*x^(-1.06)} node [midway,below,yshift=-0.40cm, xshift=-0.5cm] {$\mathcal{O}(N^{-1.06})$};

\end{loglogaxis}
\end{tikzpicture}
\hfill
\begin{tikzpicture}
\begin{loglogaxis}[
 width=0.48\textwidth, height=9.5cm,
legend style={at={(0.02,0.02)},anchor=south west, align=left, draw = none },
legend cell align={left},
xlabel={number of elements},
ylabel={estimator},
]
	\addplot+[solid,mark=square,mark size=2pt,mark options={line width=1.0pt},color=red] table
		[x=number_of_elements,y=squared]{figures/Lshape_pp_k1_expanded_theta04.csv};
	\addlegendentry{$k=1$}

	\addplot+[solid,mark=square,mark size=2pt,mark options={line width=1.0pt},color=blue] table
		[x=number_of_elements,y=squared]{figures/Lshape_pp_k2_expanded_theta04.csv};
	\addlegendentry{$k=2$}

	\addplot+[solid,mark=square,mark size=2pt,mark options={line width=1.0pt},color=green] table
		[x=number_of_elements,y=squared]{figures/Lshape_pp_k4_expanded_theta04.csv};
	\addlegendentry{$k=4$}

	\addplot+[solid,mark=square,mark size=2pt,mark options={line width=1.0pt},color=yellow] table
		[x=number_of_elements,y=squared]{figures/Lshape_pp_k8_expanded_theta04.csv};
	\addlegendentry{$k=8$}

	\addplot+[solid,mark=square,mark size=2pt,mark options={line width=1.0pt},color=pink] table
		[x=number_of_elements,y=squared]{figures/Lshape_pp_k16_expanded_theta04.csv};
	\addlegendentry{$k=16$}

	\addplot+[solid,mark=*,mark size=2pt,mark options={line width=1.0pt},color=red] table
		[x=number_of_elements,y=squared]{figures/Lshape_pp_k1_theta1.csv};
	
	\addplot+[solid,mark=*,mark size=2pt,mark options={line width=1.0pt},color=blue] table
		[x=number_of_elements,y=squared]{figures/Lshape_pp_k2_theta1.csv};

	\addplot+[solid,mark=*,mark size=2pt,mark options={line width=1.0pt},color=green] table
		[x=number_of_elements,y=squared]{figures/Lshape_pp_k4_theta1.csv};
	
	\addplot+[solid,mark=*,mark size=2pt,mark options={line width=1.0pt},color=yellow] table
		[x=number_of_elements,y=squared]{figures/Lshape_pp_k8_theta1.csv};

	\addplot+[solid,mark=*,mark size=2pt,mark options={line width=1.0pt},color=pink] table
		[x=number_of_elements,y=squared]{figures/Lshape_pp_k16_theta1.csv};

	\addplot [black,dashed ] expression [domain=56:142000, samples = 10] {30*x^(-1.06)} node [midway,below,yshift=-1.00cm, xshift=0.50cm] {$\mathcal{O}(N^{-1.06})$};
	
			\addplot [mark=square*,mark size=3pt] coordinates { (466,149.200) };
				
			\addplot [mark=square*,mark size=3pt] coordinates { (122,82.0019130404) };

			\addplot [mark=*,mark size=3pt] coordinates { (896,197.9393) };
						
			\addplot [mark=*,mark size=3pt] coordinates { (224,58.2478) };

\end{loglogaxis}

\end{tikzpicture}




			



\caption{Ex.~\ref{ex:one:conv}: Convergence of $\eta_\ell^2$ for standard D\"orfler marking vs.\ expanded D\"orfler with 
$\theta = 0.4$ and uniform refinement (left). 
Expanded D\"orfler marking (squares) vs.\ uniform refinement (circles) for different values of $k>0$ (right).}
  \label{fig:ex1:compare_pp_marking}

\end{figure}

\begin{figure}


\begin{tikzpicture}
\begin{loglogaxis}[
 width=0.48\textwidth, height=6.5cm,
legend style={at={(0.02,0.42)},anchor=south west, align=left, draw = none },
legend cell align={left},
xlabel={number of elements},
ylabel={condition number},
]
\addplot+[solid,mark=triangle*, mark size=2pt,mark options={line width=1.0pt},color=red] table
	[x=number_of_elements,y=cond_number]{figures/Lshape_cond_mm_k1_ada.csv};
\addlegendentry{$k=1$}	

\addplot+[solid,mark=*, mark size=2pt,mark options={line width=1.0pt},color=blue] table
	[x=number_of_elements,y=cond_number]{figures/Lshape_cond_mm_k2_ada.csv};
\addlegendentry{$k=2$}	

\addplot+[solid,mark=diamond*, mark size=2pt,mark options={line width=1.0pt},color=green] table
[x=number_of_elements,y=cond_number]{figures/Lshape_cond_mm_k4_ada.csv};
\addlegendentry{$k=4$}	

\addplot+[solid,mark=pentagon*, mark size=2pt,mark options={line width=1.0pt},color=yellow] table
[x=number_of_elements,y=cond_number]{figures/Lshape_cond_mm_k8_ada.csv};
\addlegendentry{$k=8$}	

\addplot+[solid,mark=square*, mark size=2pt,mark options={line width=1.0pt},color=pink] table
[x=number_of_elements,y=cond_number]{figures/Lshape_cond_mm_k16_ada.csv};
\addlegendentry{$k=16$}	
		
	

\end{loglogaxis}
\end{tikzpicture}
\hfill
\begin{tikzpicture}
\begin{loglogaxis}[
 width=0.48\textwidth, height=6.5cm,
legend style={at={(0.02,0.42)},anchor=south west, align=left, draw = none },
legend cell align={left},
xlabel={number of elements},
ylabel={condition number},
]
\addplot+[solid,mark=triangle*, mark size=2pt,mark options={line width=1.0pt},color=red] table
[x=number_of_elements,y=cond_number]{figures/Lshape_cond_pp_k1_ada.csv};
\addlegendentry{$k=1$}	

\addplot+[solid,mark=*,mark size=2pt,mark options={line width=1.0pt},color=blue] table
[x=number_of_elements,y=cond_number]{figures/Lshape_cond_pp_k2_ada.csv};
\addlegendentry{$k=2$}	

\addplot+[solid,mark=diamond*, mark size=2pt,mark options={line width=1.0pt},color=green] table
[x=number_of_elements,y=cond_number]{figures/Lshape_cond_pp_k4_ada.csv};
\addlegendentry{$k=4$}	

\addplot+[solid,mark=pentagon*, mark size=2pt,mark options={line width=1.0pt},color=yellow] table
[x=number_of_elements,y=cond_number]{figures/Lshape_cond_pp_k8_ada.csv};
\addlegendentry{$k=8$}	

\addplot+[solid,mark=square*, mark size=2pt,mark options={line width=1.0pt},color=pink] table
[x=number_of_elements,y=cond_number]{figures/Lshape_cond_pp_k16_ada.csv};
\addlegendentry{$k=16$}



\end{loglogaxis}
\end{tikzpicture}

	
	\caption{Ex.~\ref{ex:one:non}: Condition number of the arising linear system. The condition number for the first 20 adaptive steps for non-convex case (left) and the convex case (right).}
	\label{fig:ex1:compare_cond}
\end{figure}
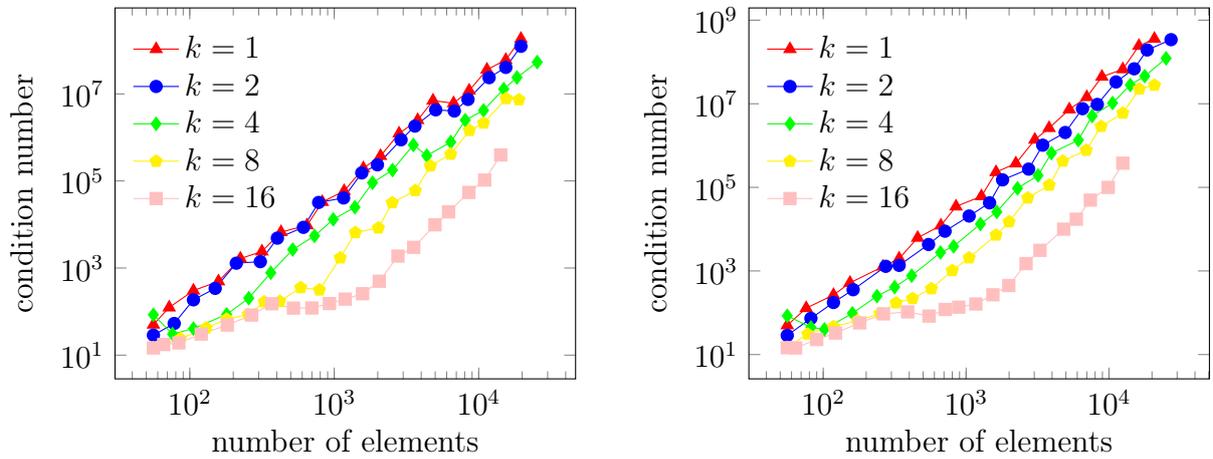

%
%
%
%
%

\begin{figure}
	\adjincludegraphics[width=0.40\textwidth,Clip={.10\width} {.06\height} {0.0\width} {.05\height}]{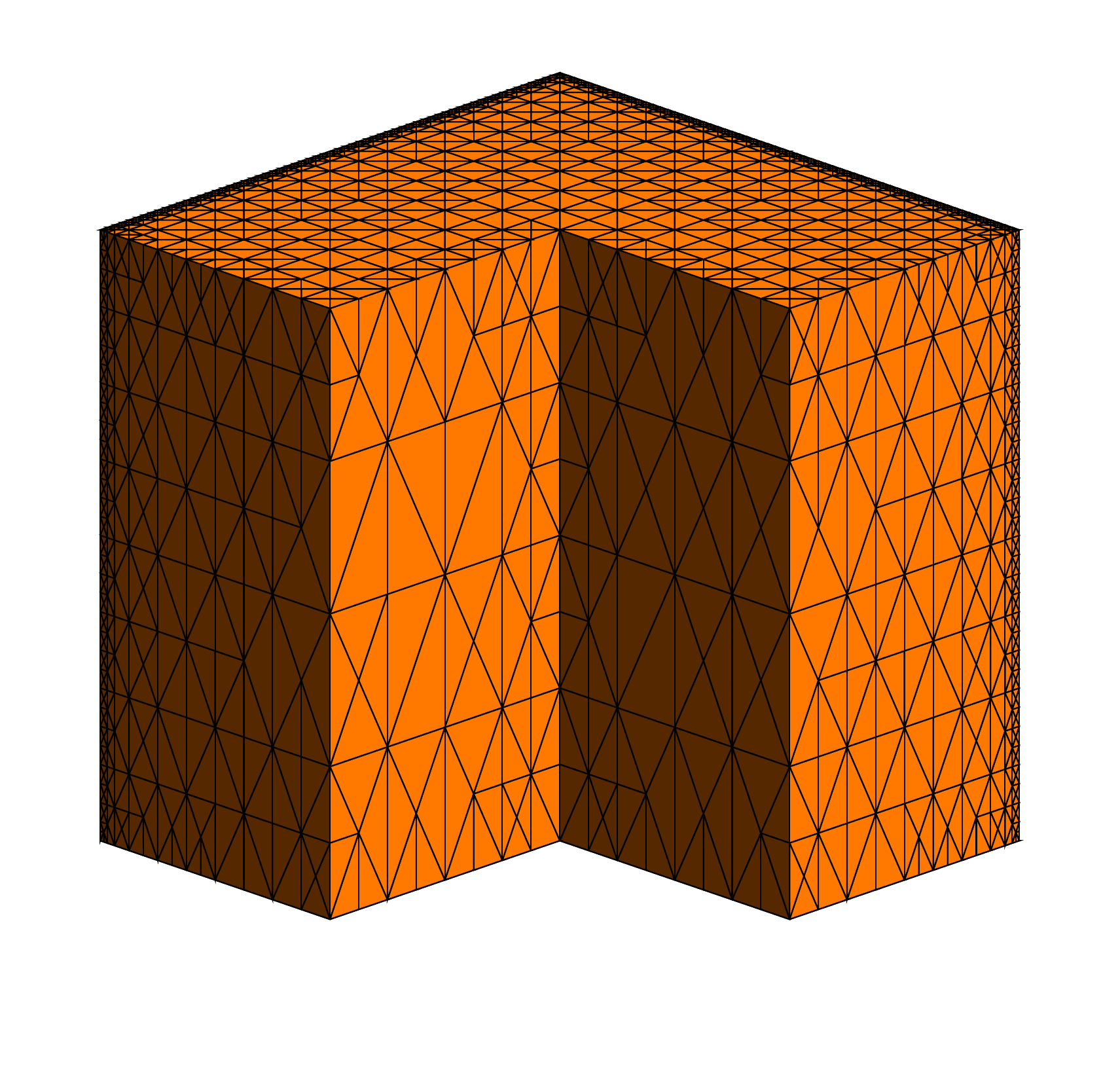}
	\adjincludegraphics[width=0.40\textwidth,Clip={.10\width} {.06\height} {0.0\width} {.05\height}]{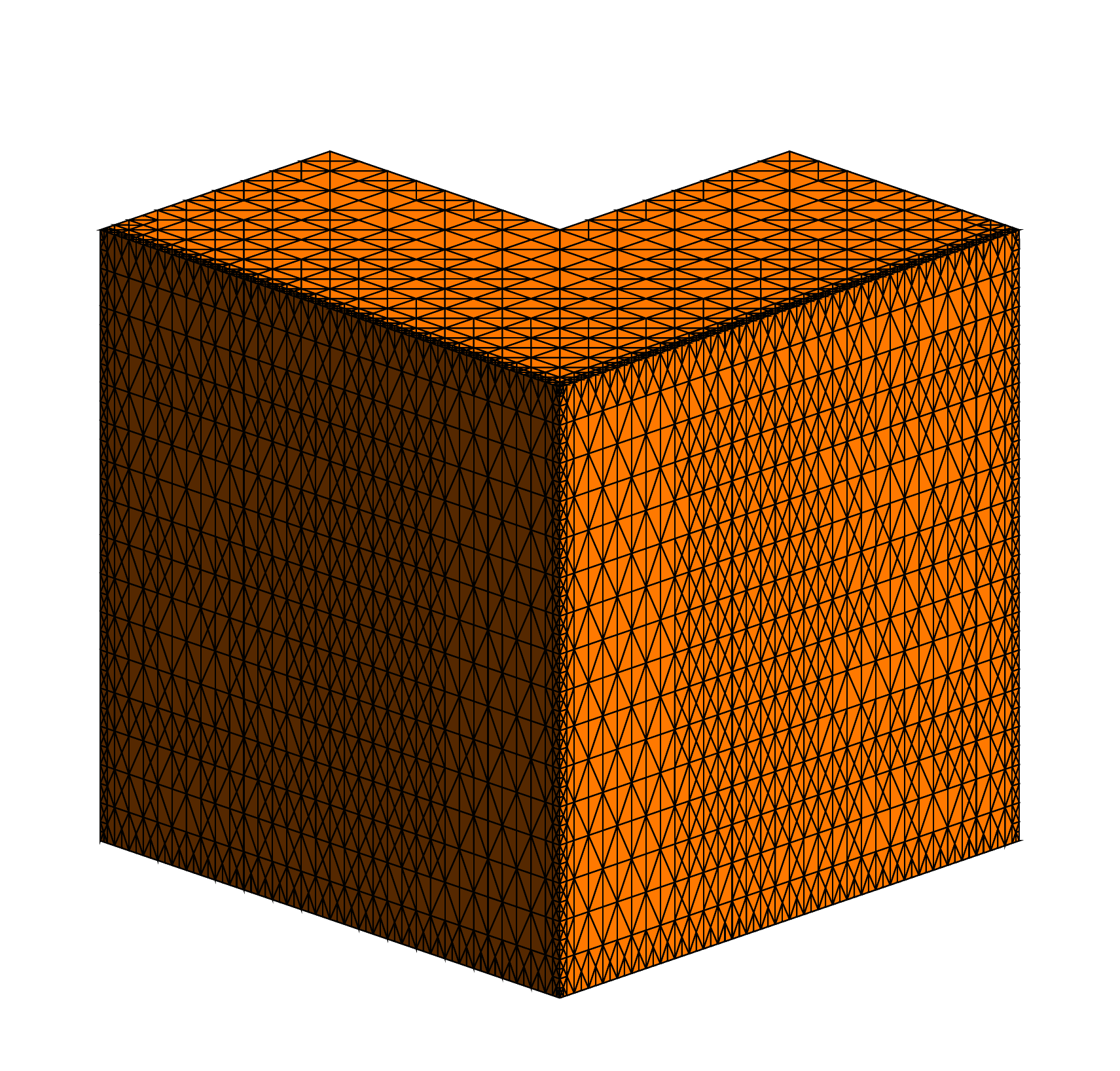}

	\caption{Ex.~\ref{ex:one:conv}: Triangulation $\TT_{16}$ with $7472$ elements in the convex case.
	The refinement focuses on the surface hit by the incoming wave (right), whereas
	all parts in the shadow remain relatively coarse (left).}
	\label{fig:ex1:meshes_pp}	  
\end{figure}

\subsection{Sound-hard scattering on a L-shaped domain}
\label{ex:two}
For the second example, we consider sound-hard scattering on an L-shaped domain from Figure~\ref{fig:ex1:inital_mesh}.
The direction of the incident wave is given by  $a = (-1/\sqrt{2},1/\sqrt{2},0)^T$, hitting the scatterer on 
the non-convex part; see Figure~\ref{fig:ex2:setup} (left).

 Figure~\ref{fig:ex2:compare_theta} (left) compares uniform vs.\ adaptive mesh-refinement for fixed $k=1$ and various $\theta = \{0.2,0.4,0.6,0.8\}$.
Algorithm~\ref{algorithm} leads to the improved rate $\eta_\ell^2 = \OO(N^{-1})$, while uniform mesh-refinement leads to 
a reduced rate $\OO(N^{-2/3})$. Figure~\ref{fig:ex2:compare_k} shows the adaptive rate for various $k \in \{1,2,4,8,16 \}$ and fixed $\theta=0.2$ (left)
as well as $\theta=0.4$ (right).
A higher wavenumber $k$ just influences the invoked constants and the length of the preasymptotic phase, but does not effect the rate of convergence. 
For $k=16$, we admit that the computed range is not sufficient to observe a better rate of convergence for the adaptive scheme.
Finally,  Figure~\ref{fig:ex2:compare_theta} (right) plots the condition number of the arising linear system of the discrete formulation~\eqref{eq:discreteformulation:kompakt:hypsing}. Similar to sound-soft scattering, the condition number indicates that the linear system admits a unique solution for every $\ell \in \N$ and hence Algorithm~\ref{algorithm} never enforced uniform mesh-refinement in Step~{\rm{(i)}}.


 \begin{figure}
 	\centering
 	\adjincludegraphics[width=0.50\textwidth,Clip={.10\width} {.05\height} {0.0\width} {.1\height}]{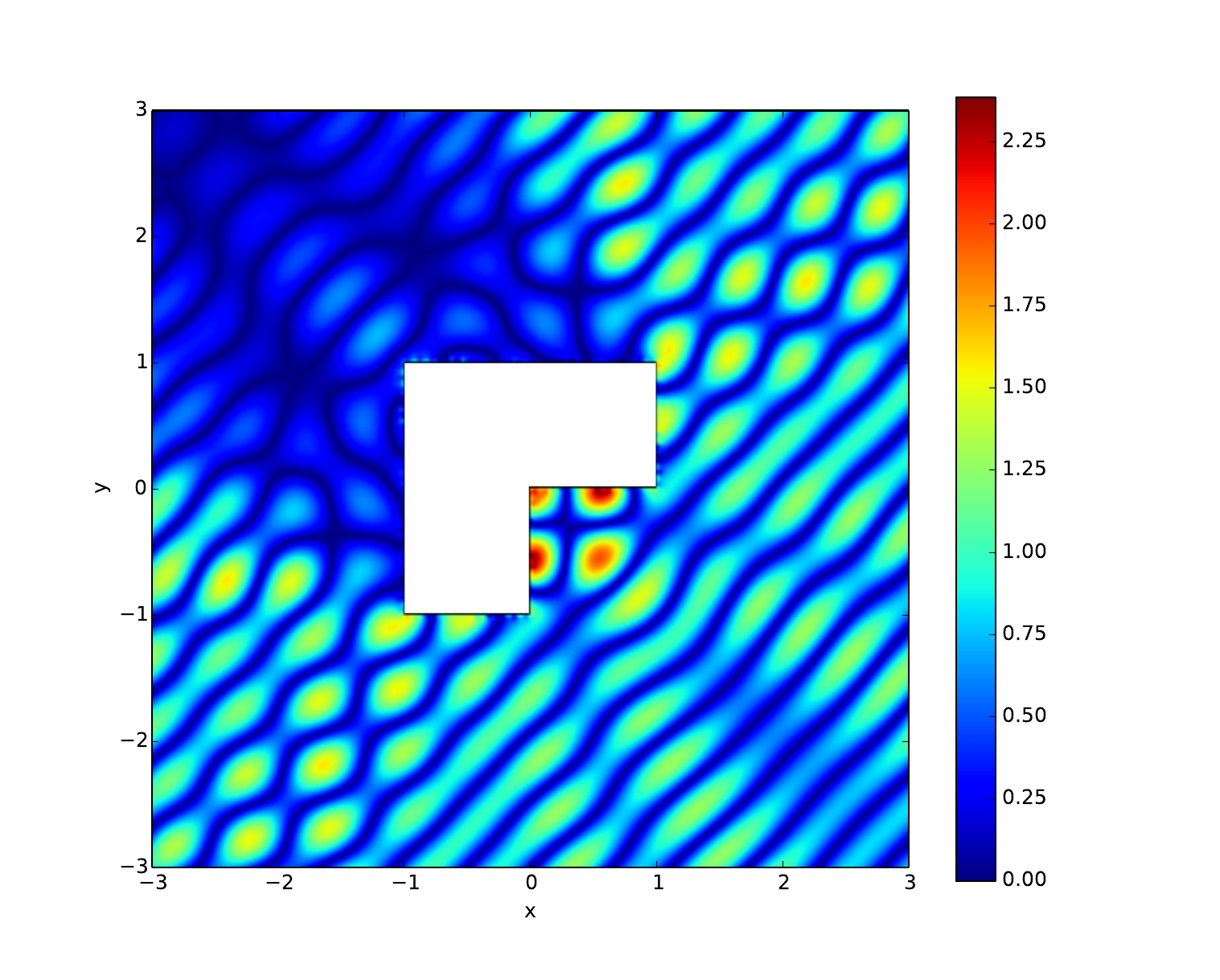}
 	\caption{Ex.~\ref{ex:two}: Total field $u^{\rm tot}$ for sound-hard scattering with wavenumber $k=8$.
 	The incident wave $u^{\operatorname{inc}}$ hits the scatterer on the non-convex part of the domain. }
 	 	\label{fig:ex2:setup}
 \end{figure} 	

 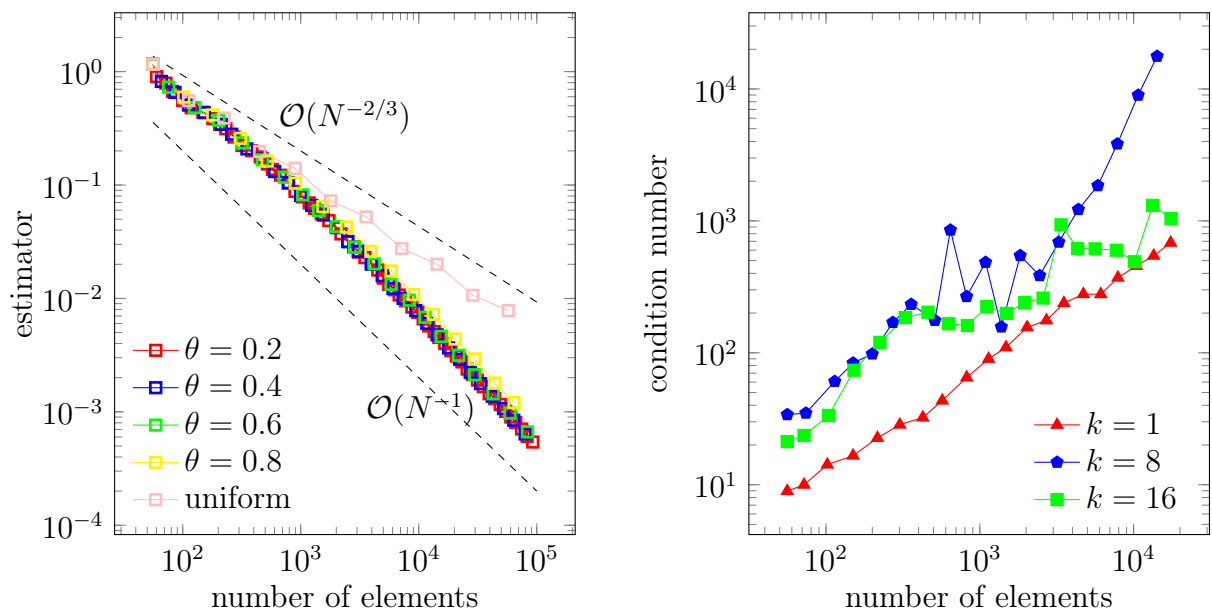
\begin{figure}
 	\centering


\begin{tikzpicture}
\begin{loglogaxis}[
 width=0.48\textwidth, height=8.5cm,
legend style={at={(0.02,0.02)},anchor=south west, align=left, draw = none },
legend cell align={left},
xlabel={number of elements},
ylabel={estimator},
]

	\addplot+[solid,mark=square,mark size=2pt,mark options={line width=1.0pt},color=red] table
		[x=number_of_elements,y=squared]{figures/sh_Lshape_mm_k1_theta02.csv};
	\addlegendentry{$\theta=0.2$}

	\addplot+[solid,mark=square,mark size=2pt,mark options={line width=1.0pt},color=blue] table
		[x=number_of_elements,y=squared]{figures/sh_Lshape_mm_k1_theta04.csv};
	\addlegendentry{$\theta=0.4$}

	\addplot+[solid,mark=square,mark size=2pt,mark options={line width=1.0pt},color=green] table
		[x=number_of_elements,y=squared]{figures/sh_Lshape_mm_k1_theta06.csv};
	\addlegendentry{$\theta=0.6$}

	\addplot+[solid,mark=square,mark size=2pt,mark options={line width=1.0pt},color=yellow] table
		[x=number_of_elements,y=squared]{figures/sh_Lshape_mm_k1_theta08.csv};
	\addlegendentry{$\theta=0.8$}

	\addplot+[solid,mark=square,mark size=2pt,mark options={line width=1.0pt},color=pink] table
		[x=number_of_elements,y=squared]{figures/sh_Lshape_mm_k1_theta1.csv};
	\addlegendentry{uniform}

	\addplot [black,dashed ] expression [domain=56:100000, samples = 10] {20*x^(-2/3)} node [midway,above,yshift=+0.45cm] {$\mathcal{O}(N^{-2/3})$};	
	\addplot [black,dashed ] expression [domain=56:100000, samples = 10] {20*x^(-1)} node [midway,below,yshift=-1.0cm, xshift=1.0cm] {$\mathcal{O}(N^{-1})$};

\end{loglogaxis}
\end{tikzpicture}
\hfill
\begin{tikzpicture}
\begin{loglogaxis}[
width=0.48\textwidth, height=8.5cm,
legend style={at={(0.60,0.02)},anchor=south west, align=left, draw = none },
legend cell align={left},
xlabel={number of elements},
ylabel={condition number},
]
\addplot+[solid,mark=triangle*, mark size=2pt,mark options={line width=1.0pt},color=red] table
[x=number_of_elements,y=cond_number]{figures/sh_Lshape_cond_k1_ada.csv};
\addlegendentry{$k=1$}	

%

\addplot+[solid,mark=pentagon*, mark size=2pt,mark options={line width=1.0pt},color=blue] table
[x=number_of_elements,y=cond_number]{figures/sh_Lshape_cond_k8_ada.csv};
\addlegendentry{$k=8$}	

\addplot+[solid,mark=square*, mark size=2pt,mark options={line width=1.0pt},color=green] table
[x=number_of_elements,y=cond_number]{figures/sh_Lshape_cond_k16_ada.csv};
\addlegendentry{$k=16$}	



\end{loglogaxis}
\end{tikzpicture}
 	\caption{Ex.~\ref{ex:two}: Convergence of $\eta_\ell^2$ for different values of $\theta \in \{0.2,0.4,0.6,0.8\}$ as well as uniform refinement (left).
 	The plot uses expanded D\"orfler marking with $k=1$. Condition number of the linear system in~\eqref{eq:discreteformulation:kompakt:hypsing} (right).}
 	\label{fig:ex2:compare_theta}		  
 \end{figure}

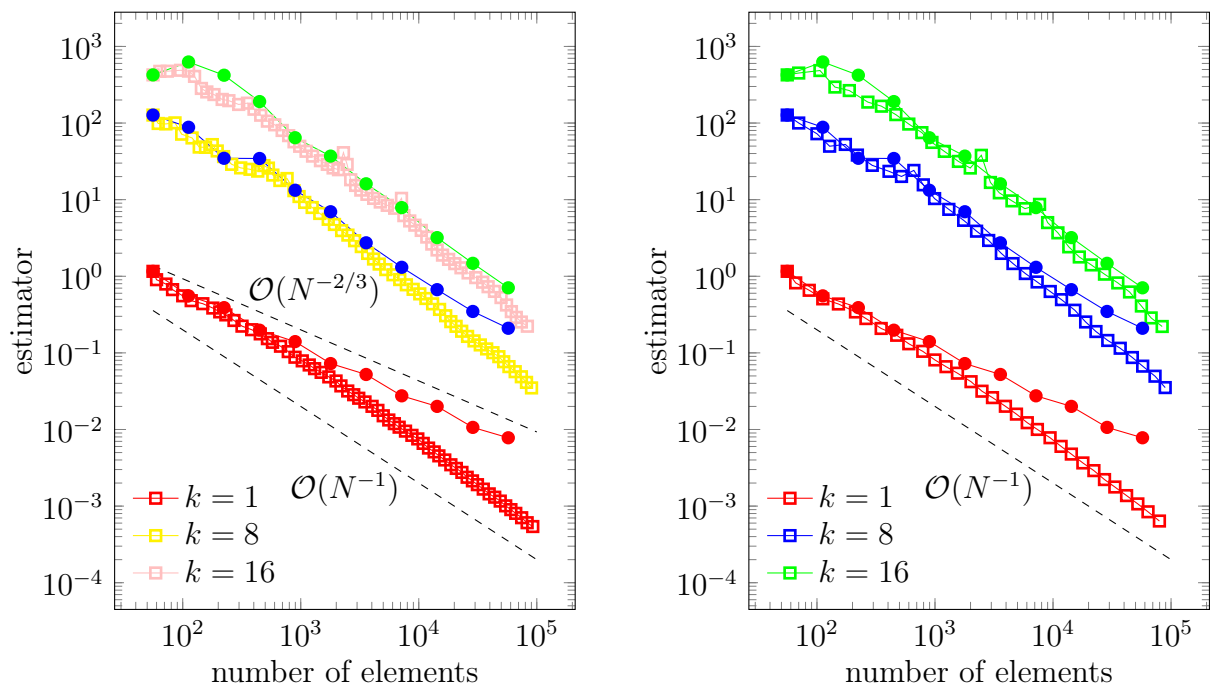
\begin{figure}


\begin{tikzpicture}
\begin{loglogaxis}[
 width=0.48\textwidth, height=9.5cm,
legend style={at={(0.02,0.02)},anchor=south west, align=left, draw = none },
legend cell align={left},
xlabel={number of elements},
ylabel={estimator},
]
	\addplot+[solid,mark=square,mark size=2pt,mark options={line width=1.0pt},color=red] table
		[x=number_of_elements,y=squared]{figures/sh_Lshape_mm_k1_theta02.csv};
	\addlegendentry{$k=1$}

	\addplot+[solid,mark=square,mark size=2pt,mark options={line width=1.0pt},color=yellow] table
		[x=number_of_elements,y=squared]{figures/sh_Lshape_mm_k8_theta02.csv};
	\addlegendentry{$k=8$}

	\addplot+[solid,mark=square,mark size=2pt,mark options={line width=1.0pt},color=pink] table
		[x=number_of_elements,y=squared]{figures/sh_Lshape_mm_k16_theta02.csv};
	\addlegendentry{$k=16$}

	\addplot+[solid,mark=*,mark size=2pt,mark options={line width=1.0pt},color=red] table
		[x=number_of_elements,y=squared]{figures/sh_Lshape_mm_k1_theta1.csv};

	\addplot+[solid,mark=*,mark size=2pt,mark options={line width=1.0pt},color=blue] table
		[x=number_of_elements,y=squared]{figures/sh_Lshape_mm_k8_theta1.csv};

	\addplot+[solid,mark=*,mark size=2pt,mark options={line width=1.0pt},color=green] table
		[x=number_of_elements,y=squared]{figures/sh_Lshape_mm_k16_theta1.csv};

	\addplot [black,dashed ] expression [domain=56:100000, samples = 10] {20*x^(-2/3)} node [midway,above,yshift=+0.40cm, xshift=-0.40cm] {$\mathcal{O}(N^{-2/3})$};		
	\addplot [black,dashed ] expression [domain=56:100000, samples = 10] {20*x^(-1)} node [midway,below,yshift=-0.35cm] {$\mathcal{O}(N^{-1})$};

\end{loglogaxis}

\end{tikzpicture}
\hfill
\begin{tikzpicture}
\begin{loglogaxis}[
 width=0.48\textwidth, height=9.5cm,
legend style={at={(0.02,0.02)},anchor=south west, align=left, draw = none },
legend cell align={left},
xlabel={number of elements},
ylabel={estimator},
]
	\addplot+[solid,mark=square,mark size=2pt,mark options={line width=1.0pt},color=red] table
		[x=number_of_elements,y=squared]{figures/sh_Lshape_mm_k1_theta04.csv};
	\addlegendentry{$k=1$}



	\addplot+[solid,mark=square,mark size=2pt,mark options={line width=1.0pt},color=blue] table
		[x=number_of_elements,y=squared]{figures/sh_Lshape_mm_k8_theta04.csv};
	\addlegendentry{$k=8$}

	\addplot+[solid,mark=square,mark size=2pt,mark options={line width=1.0pt},color=green] table
		[x=number_of_elements,y=squared]{figures/sh_Lshape_mm_k16_theta04.csv};
	\addlegendentry{$k=16$}

	\addplot+[solid,mark=*,mark size=2pt,mark options={line width=1.0pt},color=red] table
		[x=number_of_elements,y=squared]{figures/sh_Lshape_mm_k1_theta1.csv};
	
%
	
	\addplot+[solid,mark=*,mark size=2pt,mark options={line width=1.0pt},color=blue] table
		[x=number_of_elements,y=squared]{figures/sh_Lshape_mm_k8_theta1.csv};

	\addplot+[solid,mark=*,mark size=2pt,mark options={line width=1.0pt},color=green] table
		[x=number_of_elements,y=squared]{figures/sh_Lshape_mm_k16_theta1.csv};

	\addplot [black,dashed ] expression [domain=56:100000, samples = 10] {20*x^(-1)} node [midway,below,yshift=-0.35cm] {$\mathcal{O}(N^{-1})$};

\end{loglogaxis}

\end{tikzpicture}




			



\caption{Ex.~\ref{ex:two}: Convergence of $\eta_\ell^2$ for expanded D\"orfler (squares) vs.\ uniform refinement (circles) for different values of $k \in \{1,2,4,8,16\}$. 
The computations use $\theta = 0.2$ (left) as well as $\theta = 0.4$ (right).
}
\label{fig:ex2:compare_k}
\end{figure}

\subsection{Sound-soft scattering on a smooth sphere}
\label{ex:three}
As third example, we consider sound-soft scattering on a smooth sphere with radius $1$.
The direction of the incident wave is given by  $a = (1,0,0)^T$; see Figure~\ref{fig:ex3:initial_mesh} (right). 
We start with a very coarse approximation of the sphere and hence with a coarse initial mesh $\TT_{0}$ with 
$32$ elements (Figure~\ref{fig:ex3:initial_mesh} (left)), which hardly resolves the geometry. Additionally to the usual mesh-refinement by bisection of triangles, the newly created nodes are projected onto the sphere. Hence, the mesh-refinement in Algorithm~\ref{algorithm} provides 
also a better geometry approximation. If not stated otherwise, all computation use $\theta = 1/2$.

Figure~\ref{fig:ex3:convergence} compares uniform vs.\ adaptive mesh-refinement with and without the expanded marking strategy for fixed
values of $k=2$ (left) and $k=16$ (right). 
Since the computational domain is asymptotically smooth, there are no generic singularities, in contrast to the previous examples.
Hence, we observe that both, uniform mesh-refinement as well as Algorithm~\ref{algorithm} lead to the optimal rate of convergence $\eta_\ell^2 = \OO(N^{-3/2})$.
Different values of $k \in \{1,2,4,8,16 \}$ just affect the length of the preasymptotic phase; see Figure~\ref{fig:ex3:theta} (right). 
Since only one half of the (approximated) sphere is hit by the incoming wave, the adaptive algorithm 
does not lead to uniform refinement. Instead the mesh-refinement is focused on the surface hit by the incoming wave; see Figure~\ref{fig:ex3:mesh_evo}. 	
Figure~\ref{fig:ex3:theta} (left) compares different values of $\theta \in \{0.2,0.4,0.6,0.8,1\}$, where $\theta =1$ corresponds
to uniform mesh-refinement. In all cases, we observe optimal rates.

\begin{figure}
	\centering
	\adjincludegraphics[width=0.40\textwidth,Clip={.0\width} {.15\height} {0.0\width} {.2\height}]{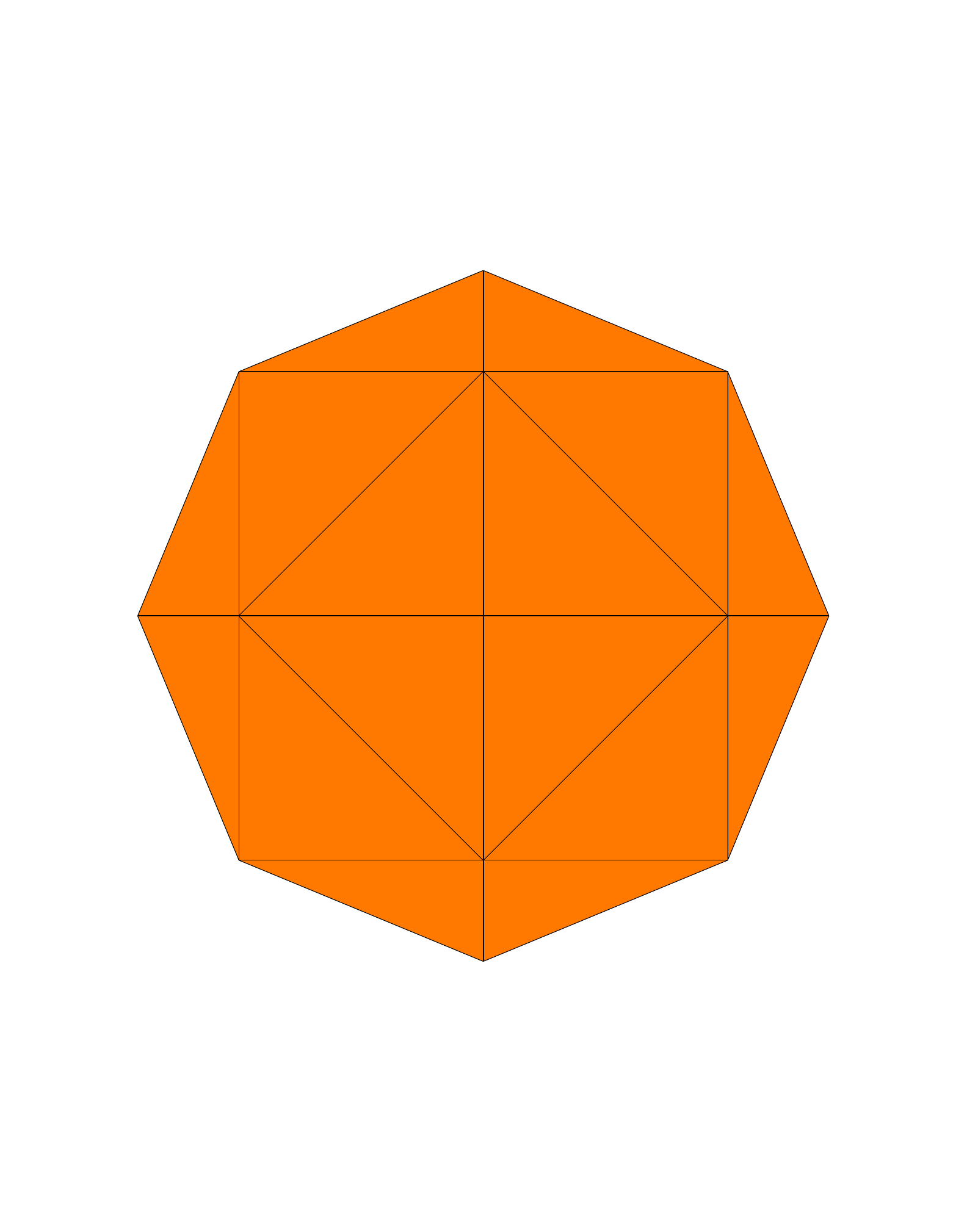}
	\includegraphics[width=0.46\textwidth]{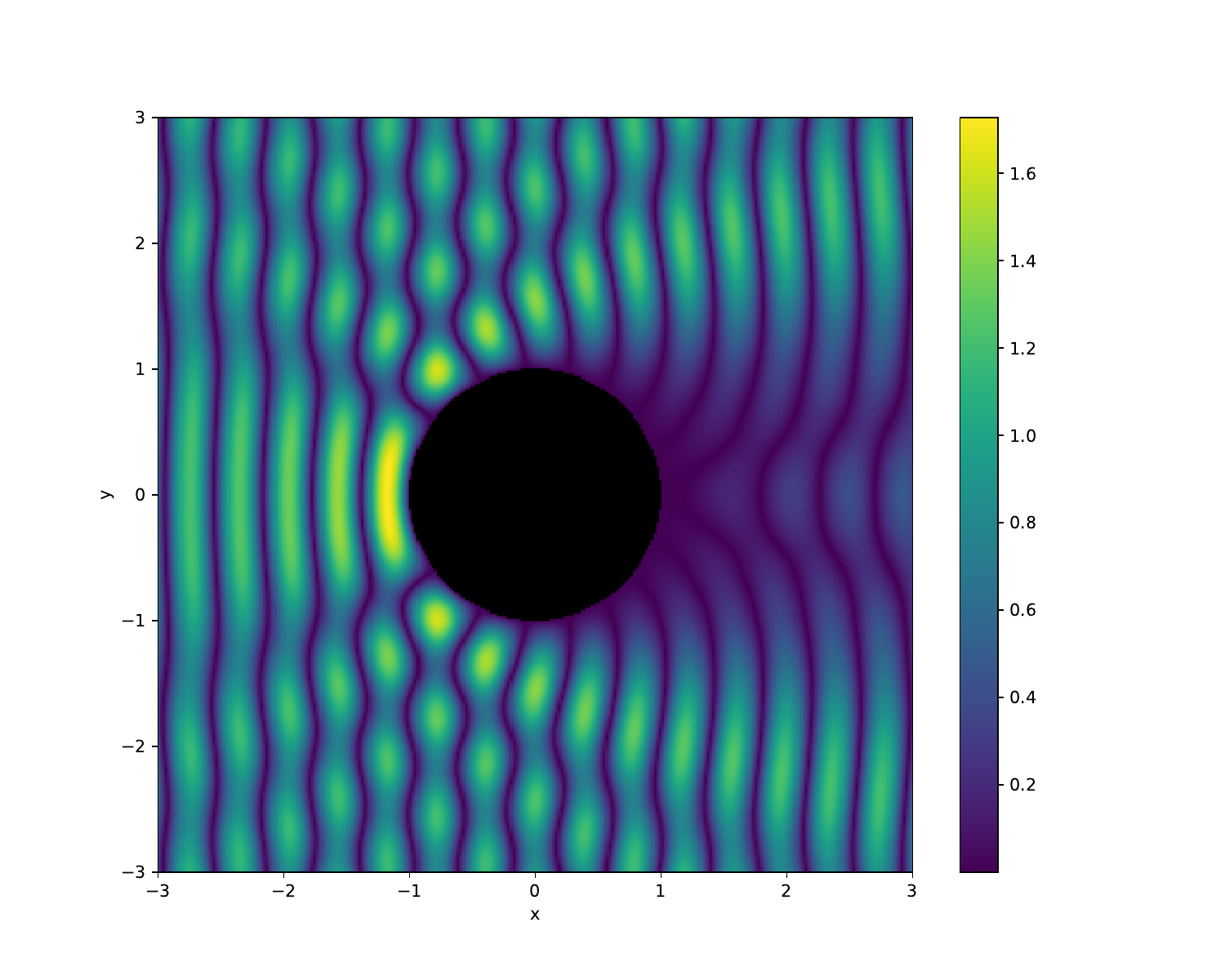}
	
	\caption{Ex.~\ref{ex:three}: Initial geometry and initial mesh $\TT_{0}$ (left).  Total field $u^{\rm tot}$ for sound-soft scattering with $k=8$ for the sphere (right). } 	 
	\label{fig:ex3:initial_mesh}
\end{figure}

 \begin{figure}
	\centering
	\begin{tikzpicture}
\begin{loglogaxis}[
width=0.48\textwidth, height=9.5cm,
legend style={at={(0.02,0.02)},anchor=south west, align=left, draw = none },
legend cell align={left},
xlabel={number of elements},
ylabel={estimator $\eta_\ell^2$},
]

\addplot+[solid,mark=*,mark size=2pt,mark options={line width=1.0pt},color=blue] table
[x=number_of_elements,y=squared]{figures/Ball_pp_k2_normal.csv};
\addlegendentry{D\"orfler}

\addplot+[solid,mark=triangle*,mark size=2pt,mark options={line width=1.0pt},color=green] table
[x=number_of_elements,y=squared]{figures/Ball_pp_k2_expanded.csv};
\addlegendentry{expanded D\"orfler}

\addplot+[solid,mark=square*,mark size=2pt,mark options={line width=1.0pt},color=red] table
[x=number_of_elements,y=squared]{figures/Ball_pp_k2_uniform.csv};
\addlegendentry{uniform}


\addplot [black,dashed ] expression [domain=20:120000, samples = 10] {600*x^(-3/2)} node [midway,above,yshift=+1.00cm] {$\mathcal{O}(N^{-3/2})$};

\end{loglogaxis}

\end{tikzpicture}
\hfill
\begin{tikzpicture}
\begin{loglogaxis}[
width=0.48\textwidth, height=9.5cm,
legend style={at={(0.02,0.02)},anchor=south west, align=left, draw = none },
legend cell align={left},
xlabel={number of elements},
ylabel={estimator $\eta_\ell^2$},
]

\addplot+[solid,mark=*,mark size=2pt,mark options={line width=1.0pt},color=blue] table
[x=number_of_elements,y=squared]{figures/Ball_pp_k16_normal.csv};
\addlegendentry{D\"orfler}

\addplot+[solid,mark=triangle*,mark size=2pt,mark options={line width=1.0pt},color=green] table
[x=number_of_elements,y=squared]{figures/Ball_pp_k16_expanded.csv};
\addlegendentry{expanded D\"orfler}

\addplot+[solid,mark=square*,mark size=2pt,mark options={line width=1.0pt},color=red] table
[x=number_of_elements,y=squared]{figures/Ball_pp_k16_uniform.csv};
\addlegendentry{uniform}


\addplot [black,dashed ] expression [domain=20:120000, samples = 10] {15000*x^(-3/2)} node [midway,below,xshift=-1.00cm] {$\mathcal{O}(N^{-3/2})$};

\end{loglogaxis}

\end{tikzpicture}




			

	\caption{Ex.~\ref{ex:three}: Convergence of $\eta_\ell^2$ for uniform refinement vs. Algorithm~\ref{algorithm} with 
		expanded D\"orfler and standard D\"orfler marking for $k=2$ (left) and $k=16$ (right). }
	\label{fig:ex3:convergence}		  
\end{figure}
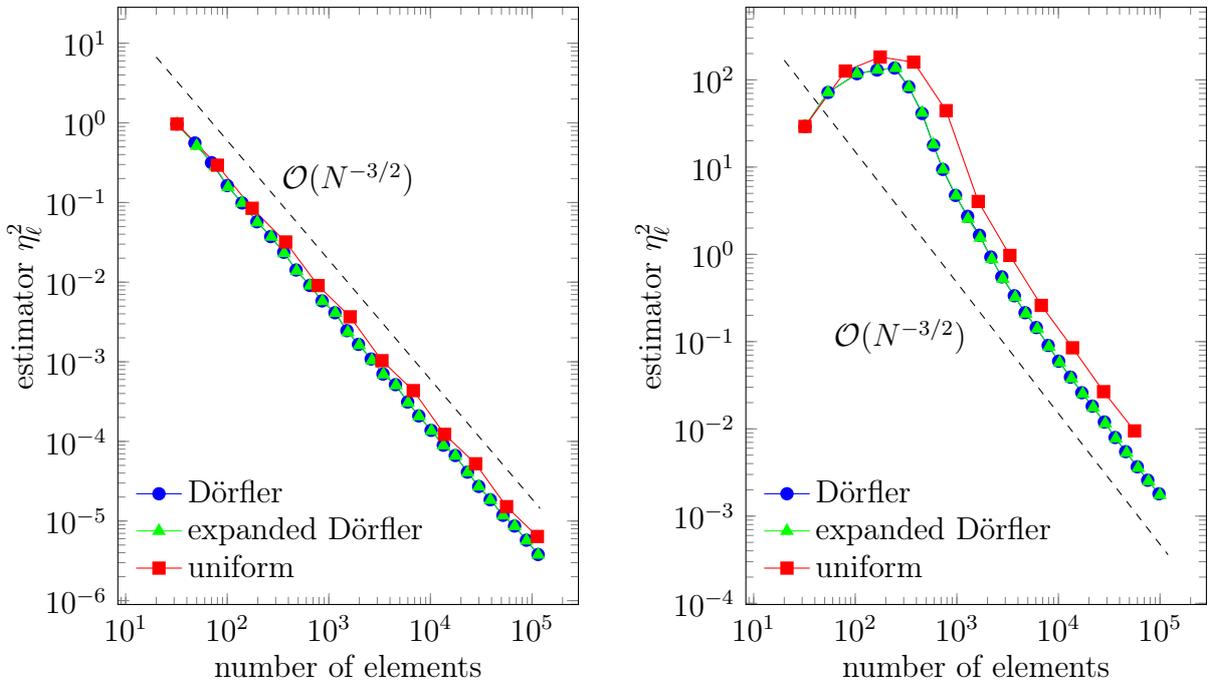

\begin{figure}
	\centering
	\adjincludegraphics[width=0.40\textwidth,Clip={.0\width} {.2\height} {0.0\width} {.2\height}]{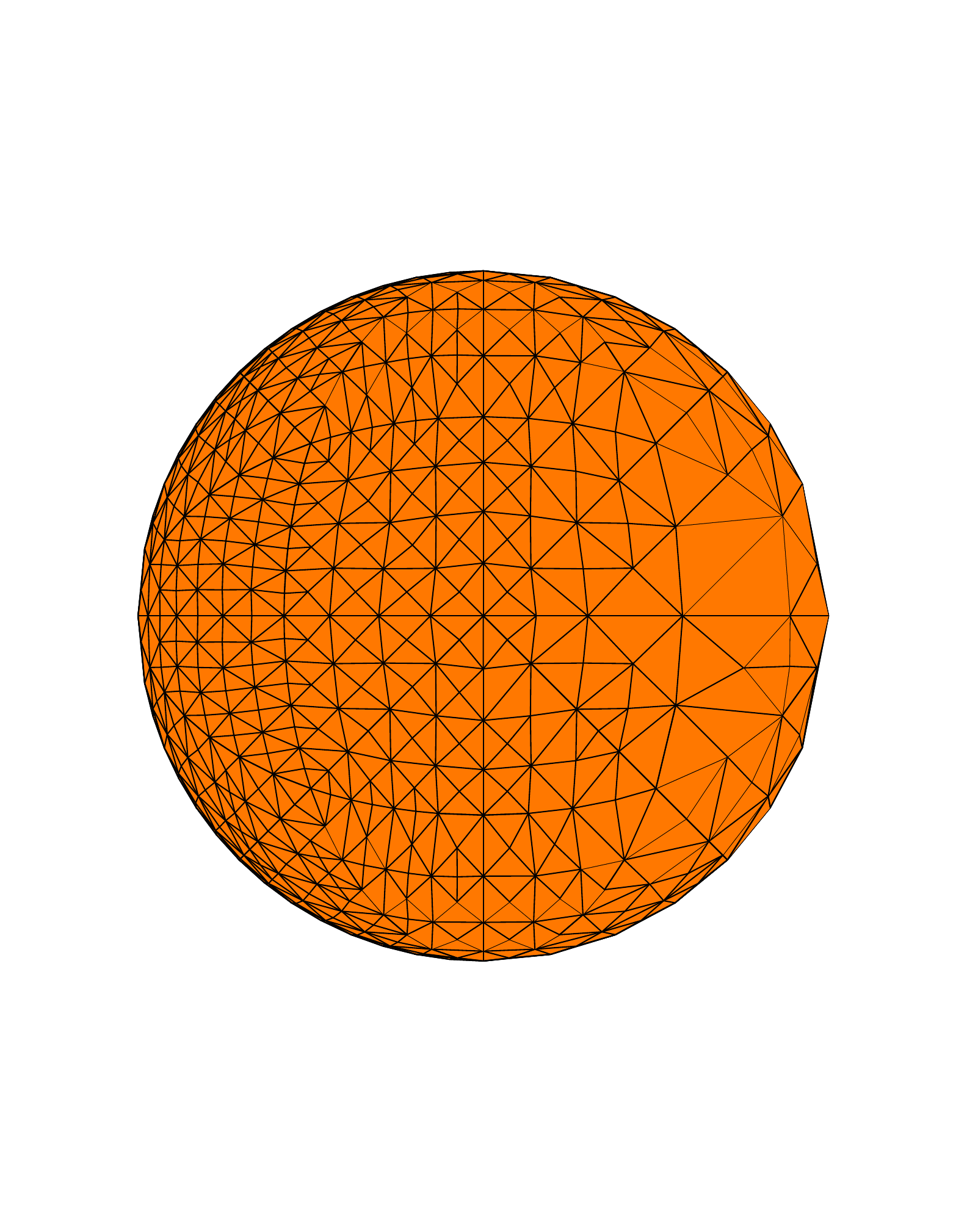}
	\adjincludegraphics[width=0.40\textwidth,Clip={.0\width} {.2\height} {0.0\width} {.2\height}]{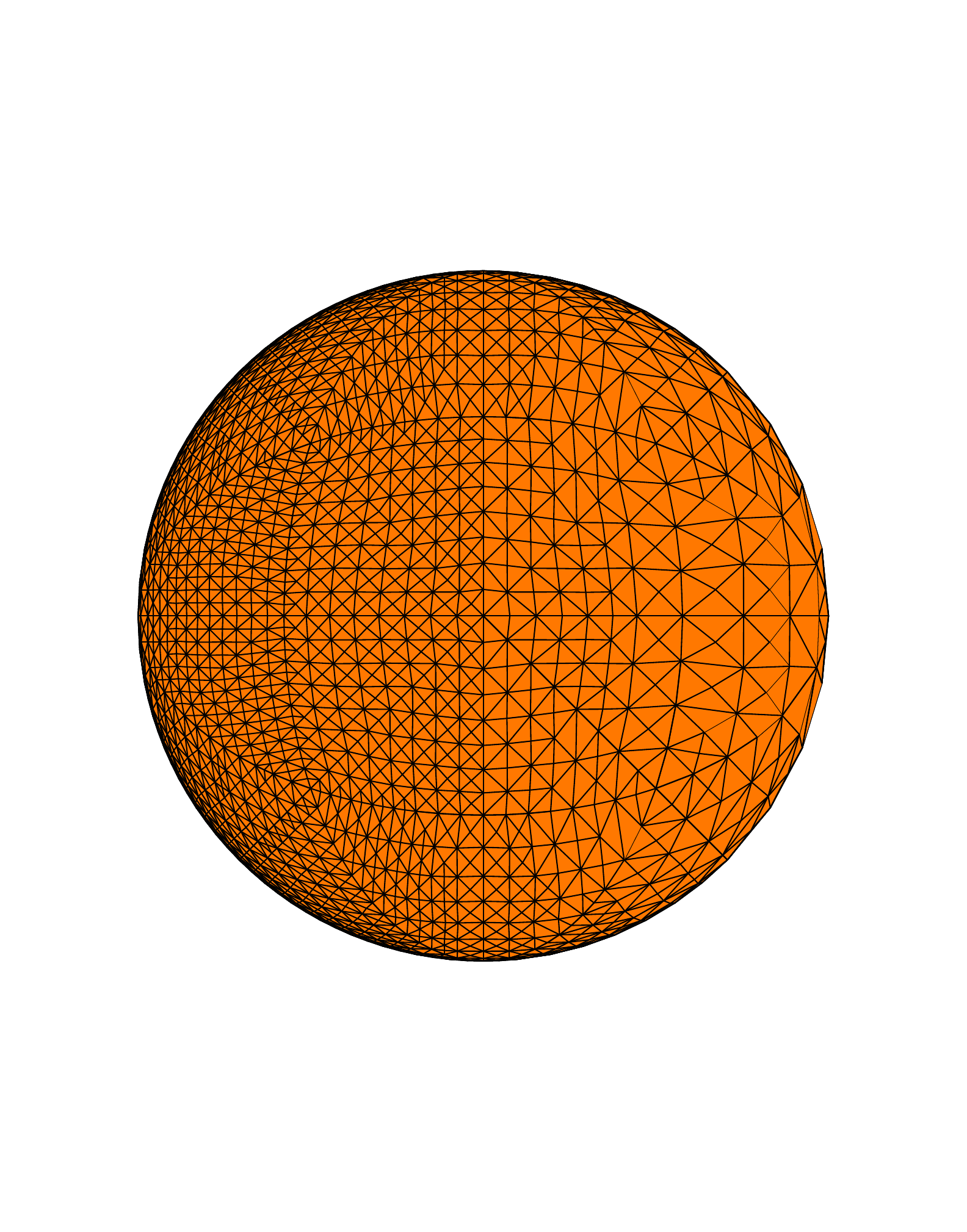}
	
	\caption{Ex.~\ref{ex:three}: Geometry and meshes $\TT_{10}$ (left) and $\TT_{15}$ (right) with $1698$ and $6206$ elements
		for Algorithm~\ref{algorithm} with the expanded D\"orfler marking. The incoming wave hits the sphere on the left.	 }
	\label{fig:ex3:mesh_evo}
\end{figure}

 \begin{figure}
	\centering
%
%
%
%
%
%
%
%
%
%
%
%
%
%
%
%
%
\begin{tikzpicture}
\begin{loglogaxis}[
 width=0.48\textwidth, height=9.5cm,
legend style={at={(0.02,0.02)},anchor=south west, align=left, draw = none },
legend cell align={left},
xlabel={number of elements},
ylabel={estimator $\eta_\ell^2$},
]
	\addplot+[solid,mark=square,mark size=2pt,mark options={line width=1.0pt},color=red] table
		[x=number_of_elements,y=squared]{figures/Ball_pp_k2_theta02.csv};
	\addlegendentry{$\theta = 0.2$}

	\addplot+[solid,mark=square,mark size=2pt,mark options={line width=1.0pt},color=blue] table
		[x=number_of_elements,y=squared]{figures/Ball_pp_k2_theta04.csv};
	\addlegendentry{$\theta = 0.4$}

	\addplot+[solid,mark=square,mark size=2pt,mark options={line width=1.0pt},color=green] table
		[x=number_of_elements,y=squared]{figures/Ball_pp_k2_theta06.csv};
	\addlegendentry{$\theta = 0.6$}

	\addplot+[solid,mark=square,mark size=2pt,mark options={line width=1.0pt},color=yellow] table
		[x=number_of_elements,y=squared]{figures/Ball_pp_k2_theta08.csv};
	\addlegendentry{$\theta = 0.8$}

	\addplot+[solid,mark=square,mark size=2pt,mark options={line width=1.0pt},color=orange] table
		[x=number_of_elements,y=squared]{figures/Ball_pp_k2_uniform.csv};
	\addlegendentry{uniform}

	\addplot [black,dashed ] expression [domain=20:120000, samples = 10] {500*x^(-3/2)} node [midway,above, xshift=+1.00cm] {$\mathcal{O}(N^{-3/2})$};

\end{loglogaxis}
\end{tikzpicture}
\hfill
\begin{tikzpicture}
\begin{loglogaxis}[
width=0.48\textwidth, height=9.5cm,
legend style={at={(0.02,0.02)},anchor=south west, align=left, draw = none },
legend cell align={left},
xlabel={number of elements},
ylabel={estimator $\eta_\ell^2$},
]
\addplot+[solid,mark=square,mark size=2pt,mark options={line width=1.0pt},color=red] table
[x=number_of_elements,y=squared]{figures/Ball_pp_k1_expanded.csv};
\addlegendentry{$k=1$}

\addplot+[solid,mark=square,mark size=2pt,mark options={line width=1.0pt},color=blue] table
[x=number_of_elements,y=squared]{figures/Ball_pp_k2_expanded.csv};
\addlegendentry{$k=2$}

\addplot+[solid,mark=square,mark size=2pt,mark options={line width=1.0pt},color=green] table
[x=number_of_elements,y=squared]{figures/Ball_pp_k4_expanded.csv};
\addlegendentry{$k=4$}

\addplot+[solid,mark=square,mark size=2pt,mark options={line width=1.0pt},color=yellow] table
[x=number_of_elements,y=squared]{figures/Ball_pp_k8_expanded.csv};
\addlegendentry{$k=8$}

\addplot+[solid,mark=square,mark size=2pt,mark options={line width=1.0pt},color=pink] table
[x=number_of_elements,y=squared]{figures/Ball_pp_k16_expanded.csv};
\addlegendentry{$k=16$}

\addplot [black,dashed ] expression [domain=56:142000, samples = 10] {15*x^(-3/2)} node [midway,below,yshift=-0.60cm] {$\mathcal{O}(N^{-3/2})$};

\end{loglogaxis}
\end{tikzpicture}




			

	\caption{Ex.~\ref{ex:three}: Convergence of $\eta_\ell^2$ for $k=2$ and different values of $\theta \in \{0.2,0.4,0.6,0.8\}$
		and uniform refinement (left). Estimator $\eta_\ell^2$ for fixed $\theta = 1/2$ and multiple values of $k \in \{1,2,4,8,16\}$ (right).}
	\label{fig:ex3:theta}		  
\end{figure}
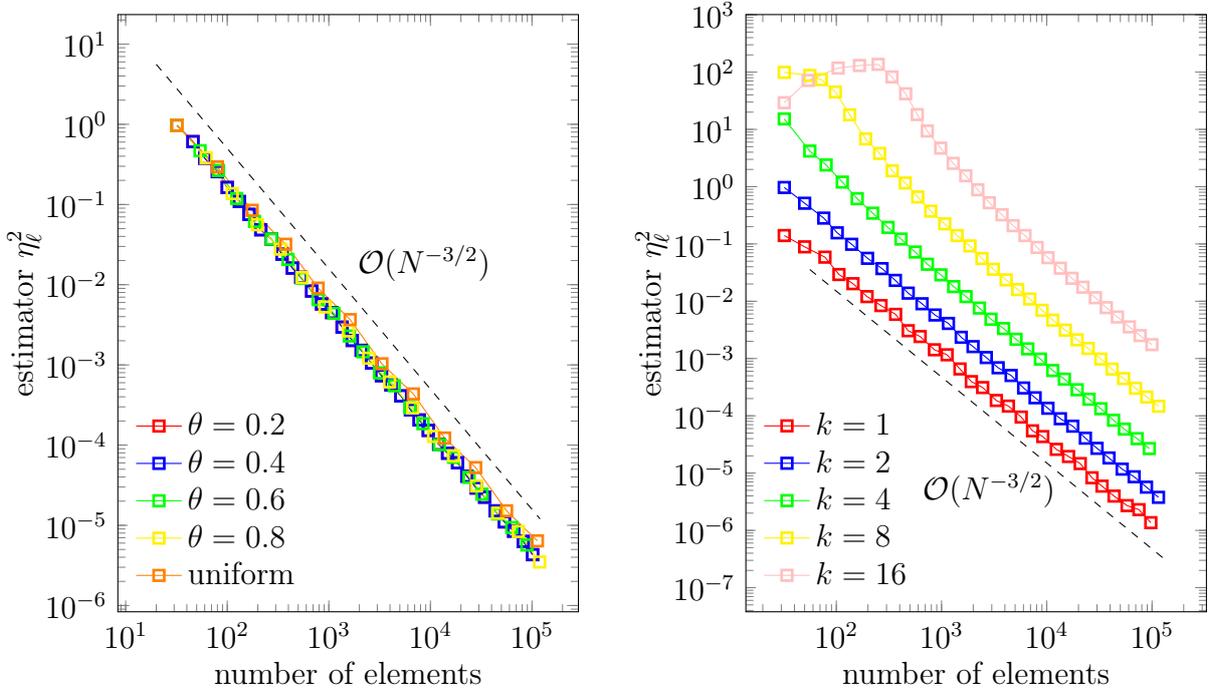


\section{Conclusions}
\label{section:conclusion}
In this article, we have analyzed an ABEM for the Helmholtz equation. 
Based on our previous work~\cite{helmholtz} on adaptive FEM, we proposed an adaptive algorithm (Algorithm~\ref{algorithm}), 
which adds an additional step to the usual adaptive scheme, if the underlying discrete system does not admit a unique solution.
Even though the necessity of the latter step could not be observed in the experiments, from the theoretical point of view, 
	this step has to be added to the algorithm in order to show even plain convergence.

We prove that Algorithm~\ref{algorithm} generates a sequence of discrete solutions, which is not only linearly convergent, 
but also converges with optimal algebraic rate. This generalizes existing results of optimal convergence for the Laplace 
equation to arbitrary wavenumber $k>0$. Although the presentation focuses on the weakly-singular integral equation,
all results directly transfer to the hypersingular integral equation as well. 
The proof of the main theorem relies on specific properties of the weighted residual error estimator. 
To verify these so-called \emph{axioms of adaptivity} from~\cite{axioms}, we prove novel and $k$-explicit inverse-type estimates 
for all underlying boundary integral operators associated with the Helmholtz equation. 
With techniques from~\cite{gantnerphd,MR3723732,fghp17}, the analysis can be generalized to adaptive isogeometric
BEM and allows the treatment of curved surfaces. 

We underpinned our theoretical findings with numerical experiments for the 3D Helm\-holtz equation. 
Thereby, the focus lies on the sound-soft and sound-hard scattering on different domains in $\R^3$. 
For the implementation, we restricted ourselves to lowest-order BEM. Due to generic edge singularities, 
higher-order polynomials would not lead to an improved order of convergence. 

The experiments confirm the theoretical results and show that Algorithm~\ref{algorithm} leads to optimal rates
independently of the actual wave number $k>0$. For both, sound-soft and sound-hard scattering,
the increase of $k$ just effects the length of the preasymptotic phase. 
The computations also show that the asymptotic regime starts quite early and 
that using restrictions for the initial mesh (e.g., $12$ elements per wavelength) is too pessimistic 
and not necessary. 

Overall, this paper appears to give the first mathematical proof of optimal convergence rates for ABEM for the 
Helmholtz equation. 


 \appendix

\section{Proof of Lemma~\ref{lemma:axioms}.}
\label{appendix:estimator}


\noindent
{\bf Proof of Lemma~\ref{lemma:axioms}~ \textrm{(i)}.} \quad 
Let $\TT_\coarse, \TT_\fine \in \T$ such that $\TT_\fine \in \refine(\TT_\coarse)$ and the corresponding discrete solutions $\Phi_\coarse\in \PP^p(\TT_\coarse)$
and $\Phi_\fine\in \PP^p(\TT_\fine)$ exist. For all non-refined elements
$T \in \TT_\coarse \cap \TT_\fine$, it holds that $h_\coarse(T) = h_\fine(T)$. 
Let $\UU:=\TT_\coarse \cap \TT_\fine$. Together with the inverse triangle inequality and the inverse estimate~\eqref{eq:discrete:invest:slo}, we obtain that
\begin{align*}
	| \eta_\coarse(\UU) - \eta_\fine(\UU)| 
	&= \big| \norm{h_\coarse^{1/2} \, \nabla_\Gamma \, (\slok \, \Phi_\coarse  - f)}{L^2(\bigcup (\UU))} 
		- \norm{h_\fine^{1/2} \, \nabla_\Gamma \, \slok \, (\Phi_\fine  - f)}{L^2(\bigcup (\UU))} \big| \\
	&\leq \norm{h_\coarse^{1/2} \, \nabla_\Gamma \, \slok \, (\Phi_\coarse  - \Phi_\fine)}{L^2(\Gamma)} 
	\leq \Cinvtilde \, (1+k^3) \, \norm{\Phi_\coarse - \Phi_\fine}{\H^{-1/2}(\Gamma)}.
\end{align*}
This concludes~\eqref{eq:stability} with $\Cstab := (1 + k^3) \, \Cinvtilde$. \qed



\medskip
\noindent
{\bf Proof of Lemma~\ref{lemma:axioms}~ \textrm{(ii)}.} \quad 
Let $\TT_\coarse, \TT_\fine \in \T$ such that $\TT_\fine \in \refine(\TT_\coarse)$ and the corresponding discrete solutions $\Phi_\coarse\in \PP^p(\TT_\coarse)$
and $\Phi_\fine\in \PP^p(\TT_\fine)$ exist.
For all  $T \in \TT_\fine \setminus \TT_\coarse$, reduction of the local mesh size implies that $h_\fine|_T \leq \qmesh \, h_\coarse|_T$. 
Using the Young inequality with arbitrary $\delta>0$, we estimate
\begin{align*}
	&\eta_\fine(\TT_\fine \setminus \TT_\coarse)^2 = \sum_{T \in \TT_\fine \setminus \TT_\coarse} \norm{h_{\fine}^{1/2} \, \nabla_{\Gamma} \, (\slok \Phi_\fine - f)}{L^2(T)}^2 \\
	&\leq \sum_{T \in \TT_\fine \setminus \TT_\coarse} \Big( \norm{h_\fine^{1/2} \, \nabla_{\Gamma} \, (\slok \Phi_\coarse - f)}{L^2(T)} + \norm{h_\fine^{1/2} \, \nabla_{\Gamma} \, \slok (\Phi_\fine - \Phi_\coarse)}{L^2(T)} \Big)^2 \\
	&\leq \sum_{T \in \TT_\fine \setminus \TT_\coarse} \Big( (1+\delta)  \, \qmesh \, \norm{h_\coarse^{1/2} \, \nabla_{\Gamma} \, (\slok \Phi_\coarse - f)}{L^2(T)}^2 
		 + (1+\delta^{-1}) \, \norm{h_\fine^{1/2} \, \nabla_{\Gamma} \, \slok (\Phi_\fine - \Phi_\coarse)}{L^2(T)}^2  \Big).
\end{align*}
Next, the inverse inequality~\eqref{eq:discrete:invest:slo} yields that
\begin{align*}		
	\eta_\coarse(\TT_\fine \setminus \TT_\coarse)^2 
	\leq  (1+\delta) \, \qmesh  \, \eta_\coarse(\TT_\coarse \setminus \TT_\fine)^2 +  (1+\delta^{-1}) \, (1+k^3)^2 \, \Cinvtilde^2  \, \norm{\Phi_\coarse - \Phi_\fine}{\H^{-1/2}(\Gamma)}^2.
\end{align*}
Choosing $\delta>0$ sufficiently small such that $\qred := (1+\delta) \, \qmesh < 1 $, we conclude~\eqref{eq:reduction} with $\Cred = (1+\delta^{-1}) \, (1+k^3)^2 \, \Cinvtilde^2$. \qed


\medskip
\noindent
{\bf Proof of Lemma~\ref{lemma:axioms}~ \textrm{(iii)}.} \quad 
We follow the arguments from~\cite[Theorem 5.3]{fkmp} for the case of $k=0$. 
Recall that notation of Proposition~\ref{eq:ss:model}. Then, existence and uniqueness of $\Phi_\fine\in \PP^p(\TT_\fine)$ is equivalent to $\beta_\fine>0$. 
The discrete inf--sup condition~\eqref{eq:ssprop:infsup} 
for $\XX_\fine := \PP^p(\TT_\fine)$ and $W_\fine := \Phi_\fine - \Phi_\coarse$ reads as
\begin{align}
\begin{split} \label{eq:proof:dreliability:1}
	\beta_\fine \, \norm{\Phi_\fine - \Phi_\coarse}{\H^{-1/2}(\Gamma)} 
	\leq \sup_{\Psi_\fine \in \XX_\fine \setminus \{0\} } \frac{\dual{\slok (\Phi_\fine -  \Phi_\coarse)}{\Psi_\fine}}{\norm{\Psi_\fine}{\H^{-1/2}(\Gamma)}}.
\end{split}
\end{align}
Let $\NN_\coarse$ denote the set of nodes corresponding to a triangulation $\TT_\coarse$. 
Let $\rho_z \in \SS^1(\TT_\coarse)$ denote the hat function associated with a node $z \in \NN_\coarse$. 
Further, let $\NN^{\RR}_\coarse:= \NN_\coarse \cap \big( \bigcup (\TT_\coarse \setminus \TT_\fine) \big)$ be the set of all nodes which belong to 
the refined elements.
Define $\RR_{\coarse,\fine} := \patch_{\coarse}(\TT_\coarse \setminus \TT_\fine)$ and $\QQ_\coarse :=  \RR_{\coarse,\fine} \setminus  (\TT_\coarse \setminus \TT_\fine)$.
These definitions give rise to disjoint decompositions
\begin{align*}
	\RR_{\coarse, \fine} = (\TT_\coarse \setminus \TT_\fine) \stackrel{\bullet}{\cup} \QQ_\coarse \quad \text{and} 
	\quad \TT_\coarse =  (\TT_\coarse \setminus \RR_{\coarse,\fine}) \stackrel{\bullet}{\cup} (\TT_\coarse \setminus \TT_\fine) \stackrel{\bullet}{\cup} \QQ_\coarse.
\end{align*}
Define $\chi := \sum_{z \in \NN^{\RR}_\coarse} \rho_z$. Then, $\chi \in \SS^1(\TT_\coarse)$ satisfies
$\rm{supp} (\chi) = \bigcup \RR_{\coarse,\fine}$ and $\chi |_{\bigcup (\TT_\coarse \setminus \TT_\fine)} \equiv 1$.
We define the operator $\pi_\coarse: \PP^p(\TT_\fine) \to \PP^p(\TT_\coarse)$ by
\begin{align*}
	\pi_\coarse(\Psi_\fine) := \begin{cases}
									0 \quad &\text{on }  \bigcup \big( \TT_\coarse \setminus \TT_\fine \big), \\
									\Psi_\fine \quad &\text{elsewhere}.
								\end{cases}
\end{align*}
For any $\Psi_\fine \in \PP^p(\TT_\fine)$ and $\Psi_\coarse \in \PP^p(\TT_\coarse)$, the Galerkin orthogonality yields that
\begin{align}\label{eq:proof:dreliability:2}
\dual{\slok \, (\Phi_\fine -  \Phi_\coarse)}{\Psi_\fine} = \dual{f- \slok \, \Phi_\coarse}{\Psi_\fine} = \dual{f- \slok \, \Phi_\coarse}{\Psi_\fine - \Psi_\coarse}.
\end{align}
Choose $\Psi_\coarse := \pi_\coarse(\Psi_\fine)\in \PP^p(\TT_\coarse)$ and note that $\rm{supp}  \big( (1-\pi_\coarse) \, \Psi_\fine \big)  \subseteq \bigcup (\TT_\coarse \setminus \TT_\fine)$. Using~\eqref{eq:proof:dreliability:2}, we derive that
\begin{align*}
	\dual{\slok \, (\Phi_\fine -&  \Phi_\coarse)}{\Psi_\fine} = \dual{f- \slok \Phi_\coarse}{(1 -\pi_\coarse) \, \Psi_\fine } 
	= \bigdual{\sum_{z \in \NN^{\RR}_\coarse} \rho_z \, (f- \slok \, \Phi_\coarse) }{(1 -\pi_\coarse) \, \Psi_\fine } \\
	&= \bigdual{\sum_{z \in \NN^{\RR}_\coarse} \rho_z \, (f- \slok \, \Phi_\coarse) }{\Psi_\fine} 
			- \bigdual{\sum_{z \in \NN^{\RR}_\coarse} \rho_z \,  (f- \slok \, \Phi_\coarse) }{\Psi_\fine|_{\bigcup \QQ_\coarse}}.
\end{align*}
Since $\QQ_\coarse \subset \TT_\coarse \cap \TT_\fine$, we obtain that $h_\coarse(T) = h_\fine(T)$ for all $T \in \QQ_\coarse$. We estimate 
\begin{align*}
	| \dual{\slok (\Phi_\fine -  \Phi_\coarse)}{\Psi_\fine} | 
	&\leq \bignorm{\sum_{z \in \NN^{\RR}_\coarse} \rho_z \, (f- \slok \, \Phi_\coarse)}{H^{1/2}(\Gamma)} \norm{\Psi_\fine}{\H^{-1/2}(\Gamma)} \\
		&\qquad + 	\bignorm{h_\coarse^{-1/2} \sum_{z \in \NN^{\RR}_\coarse} \rho_z \, (f- \slok \, \Phi_\coarse)}{L^2(\Gamma)} \norm{h_\fine^{1/2} \, \Psi_\fine}{L^2(\Gamma)}.
\end{align*}
Applying the inverse estimate~\eqref{eq:smallinvest:one} to the right-hand side, we see that
\begin{align*}
	| \dual{\slok (\Phi_\fine -  \Phi_\coarse)}{\Psi_\fine} | 
	&\lesssim \Big( \bignorm{\sum_{z \in \NN^{\RR}_\coarse} \rho_z (f- \slok \Phi_\coarse)}{H^{1/2}(\Gamma)} \\
	& \qquad + \bignorm{h_\coarse^{-1/2} \sum_{z \in \NN^{\RR}_\coarse} \rho_z (f- \slok \Phi_\coarse)}{L^2(\Gamma)} \Big) \norm{\Psi_\fine}{\H^{-1/2}(\Gamma)}.
\end{align*}
The terms in the parentheses are estimated as in~\cite{cms}. The sole difference is that compared to~\cite[Theorem 3.2]{cms} only hat functions associated with 
nodes $z \in \NN^{\RR}_\coarse$ are involved. Hence, the upper bound affects only $ \bigcup \RR_{\coarse,\fine} \subset \Gamma$ and reads
\begin{align}\label{eq:proof:dreliability:3}
	| \dual{\slok (\Phi_\fine -  \Phi_\coarse)}{\Psi_\fine} | \lesssim \norm{h_\coarse^{1/2} \, \nabla_\Gamma \, (f - \slok \Phi_\coarse)}{L^2(\bigcup \RR_{\coarse,\fine})} \, \norm{ \Psi_\fine}{\H^{-1/2}(\Gamma)}.
\end{align} 
Altogether, the combination of~\eqref{eq:proof:dreliability:1}--\eqref{eq:proof:dreliability:3} proves that
\begin{align*}
	\norm{\Phi_\fine - \Phi_\coarse}{\H^{-1/2}(\Gamma)} 
	\leq \frac{1}{\beta_\fine} \sup_{\Psi_\fine \in \XX_\fine} \frac{\dual{\slok (\Phi_\fine -  \Phi_\coarse)}{\Psi_\fine}}{\norm{\Psi_\fine}{\H^{-1/2}(\Gamma)}} 
	\lesssim \beta_\fine^{-1} \, \norm{h_\coarse^{1/2} \, \nabla_\Gamma \, (f - \slok \Phi_\coarse)}{L^2(\bigcup\RR_{\coarse,\fine})}.
\end{align*}
This concludes the proof.\qed

\medskip
\noindent
{\bf Proof of Lemma~\ref{lemma:axioms}~ \textrm{(iv)}.} \quad 
Let $\varepsilon>0$. Since uniform mesh-refinement yields convergence, we may choose $\TT_\fine \in \refine(\TT_\coarse)$ such that
$\norm{\phi - \Phi_\fine}{\H^{-1/2}(\Gamma)} \leq \varepsilon$. Lemma~\ref{lemma:axioms} (iii) hence proves that
\begin{align*}
\norm{\phi - \Phi_\coarse}{\H^{-1/2}(\Gamma)} \leq \norm{\phi - \Phi_\fine}{\H^{-1/2}(\Gamma)} + \norm{\Phi_\coarse - \Phi_\fine}{\H^{-1/2}(\Gamma)}
\leq \varepsilon + \Crel \eta_\coarse. 
\end{align*}
For $\varepsilon \to 0$, this concludes the proof. \qed

%

\section{Proof of Theorem~\ref{theorem:optimal}}
\label{appendix:optimal_proof}

The following Section gives a rigorous proof of Theorem~\ref{theorem:optimal}.
In doing so, we fill a gap in the proof of~\eqref{eq:optimal} in the abstract setting of~\cite{helmholtz}. 

\subsection{Proof of~(\ref{eq:linearconvergence})--(\ref{eq:cea})}
\label{appendix:proof:lin:convergence}
To see convergence~\eqref{eq:convergence} and, in particular, linear convergence~\eqref{eq:linearconvergence} of Algorithm~\ref{algorithm},
note that Step(v) of Algorithm~\ref{algorithm} implies
$\overline{\bigcup_{\ell \geq 0} \XX_\ell} = \H^{-1/2}(\Gamma)$ and hence the well-posedness~\cite[(A5)]{helmholtz} of the Galerkin formulation on the ``discrete limit space''. 
Moreover, Section~\ref{subsetcion:estimatorproperties} proves~\cite[(A1)--(A4)]{helmholtz}. 
Additionally, recall that $\bilina{\cdot}{\cdot} := \dual{\slonull(\cdot)}{(\cdot)}_{\Gamma}$ induces an equivalent energy norm $\enorm{\cdot}$ on $\H^{-1/2}(\Gamma)$.
Then, using $\HH := \H^{-1/2}(\Gamma)$ in~\cite[Proposition 11]{helmholtz} and~\cite[Theorem 19, Theorem 20]{helmholtz}, 
we immediately derive~\eqref{eq:linearconvergence}--\eqref{eq:cea}.

\subsection{Proof of~(\ref{eq:optimal})}
\label{appendix:proof:lin:convergence}
First, recall~\cite[Lemma 21]{helmholtz}, which recaps some important properties of the mesh refinement. 
\begin{lemma}\label{lemma:m:uniform}
	There exist $m \in \N$ and $\gamma >0$ such that the $m$-times uniform refinement $\widehat{\TT}_0$ of $\TT_0$ 
	satisfies the following properties {\rm (a)}--{\rm (d)}:
	\begin{enumerate}[label= {\rm (\alph*)}]
		\item  For all $\TT_\coarse \in \refine(\widehat{\TT}_0)$, the discrete inf--sup constant is bounded from below by $\gamma_\coarse \geq \gamma >0$. 
			In particular, there exists a unique Galerkin solution $\Phi_\coarse \in \PP^p(\TT_\coarse)$ to~\eqref{eq:modelproblem}.
	
		\item There holds quasi-monotonicity of the estimator, i.e., there exists a constant $\Cmon>0$ such that
		\begin{align*}
			\eta_\fine \leq \Cmon \eta_\coarse \quad \text{for all } \TT_\coarse \in \T \text{ and all } \TT_\fine \in \refine(\widehat{\TT}_0) \cap \refine(\TT_\coarse),
		\end{align*}
		provided that the Galerkin solution $\Phi_\coarse \in \PP^p(\TT_\coarse)$ exists and is unique. 
		
		\item There exists $\ell_3 \in \N_0$ such that $\TT_{\ell} \in \refine(\widehat{\TT}_0)$ for all $\ell \geq \ell_3$,
		where $\TT_\ell$ denotes the sequence of meshes generated by Algorithm~\ref{algorithm}.
		
		\item For all $\TT_\coarse \in \T$, the $m$-times uniform refinement $\widehat{\TT}_\coarse$ of $\TT_\coarse$ guarantees 
		$\# \widehat{\TT}_\coarse \leq \Cson^m \# \TT_\coarse$. \qed
	\end{enumerate}
\end{lemma}
Second, we need two auxiliary lemmas, which are found in~\cite[Proposition 4.12]{axioms}
resp.~\cite[Lemma 28]{helmholtz}.
\begin{lemma}[optimality of D\"orfler marking]
Suppose the assumptions of Theorem~\ref{theorem:optimal}. Then, for all $0 < \theta < \thetaopt$, there exists some $0<\gamma_{{\rm opt}}<1$ such that for all 
$\TT_\coarse \in \refine(\TT_{\ell_3})$ and all $\TT_\fine \in \refine(\TT_\coarse)$, it holds that
\begin{align}
	\eta_\fine \leq \gamma_{{\rm opt}} \eta_\coarse \quad \Longrightarrow \quad \theta \eta_\coarse^2 \leq \eta_\coarse (\RR_{\fine,\coarse})^2,
\end{align}
where $\RR_{\fine,\coarse}$ is the enlarged set of refined elements from ${\rm (3)}$ in Lemma~\ref{lemma:axioms}. \qed
\end{lemma}

\begin{lemma}\label{lemma:optimal}
Suppose the assumptions of Theorem~\ref{theorem:optimal}. Then, there exist constants $C_1,C_2>0$ such that for all $\ell \geq \ell_3$ and all $s >0$, there exists 
$\RR_\ell \subseteq \TT_\ell$ such that the following holds: If $\norm{\phi}{\Approx_{s}(\TT_{\ell_3})} < \infty$, then it holds that
\begin{align}\label{eq:lemma:optimal:one}
	\# \RR_\ell \leq C_1 \big( C_2 \norm{\phi}{\Approx_{s} (\TT_{\ell_3})} \big)^{1/s} \eta_\ell^{-1/s}
\end{align}
as well as the D\"orfler marking criterion
\begin{align}\label{eq:lemma:optimal:two}
	\theta \eta_\ell^2 \leq \eta_\ell(\RR_\ell)^2.
\end{align}
The constant $C_2$ depends only on $\theta,\widehat{\gamma}_0$, and on the constants in Lemma~\ref{lemma:axioms}. 
The constant $C_1$ depends additionally on $\# \TT_{\ell_3}$ and $\TT_0$. \qed
\end{lemma}

\bigskip
\noindent
{\bf Proof of~(\ref{eq:optimal}).} \quad 
The implication ``$\Longleftarrow$'' in~\eqref{eq:optimal} follows by definition of the approximation class, cf.~\cite[Proposition 4.15]{axioms}. 
Hence, we focus on the converse implication ``$\Longrightarrow $''. 
We suppose that $\norm{\phi}{\Approx_s} < \infty$. With $\ell_3$ being the constant from Lemma~\ref{lemma:m:uniform}, define $\lopt := \max\{\ell_1,\ell_3 \}$.
Then, \cite[Lemma 23]{helmholtz} implies that $\norm{\phi}{\Approx_s(\TT_{\lopt})} < \infty$. Further, let $\MM_\ell$ 
denote the set of marked elements in the $\ell$-th step of Algorithm~\ref{algorithm}. For $\ell \geq \lopt$, Lemma~\ref{lemma:optimal} 
provides a set $\RR_\ell \subseteq \TT_\ell$
with~\eqref{eq:lemma:optimal:one}--\eqref{eq:lemma:optimal:two}. According to the minimality of $\MM_\ell$ 
(cf. Step(iv) and Step(v) of Algorithm~\ref{algorithm}), we obtain that
\begin{align}\label{eq:optimality:proof:one}
	\# \MM_\ell \lesssim \# \RR_\ell \lesssim \norm{\phi}{\Approx_{s}(\TT_{\lopt})}^{1/s} \, \eta_\ell^{-1/s} \quad \text{for all } \ell \geq \lopt.
\end{align}
Linear convergence~\eqref{eq:linearconvergence} yields that $\eta_\ell \lesssim  \qlin^{\ell-j} \eta_j$ 
for all $\lopt \leq j \leq \ell$ and hence
\begin{align}\label{eq:optimality:proof:two}
	\eta_j^{-1/s} \lesssim \qlin^{(\ell - j)/s} \eta_\ell^{-1/s} \quad \text{for all } \lopt \leq j \leq \ell.
\end{align}
The mesh-closure estimate~\eqref{mesh:mesh-closure} yields that
\begin{align}\label{eq:optimality:proof:three}
\# \TT_\ell  - \# \TT_0 +1 \leq \Cmesh \sum_{j=0}^{\ell-1} \# \MM_j.
\end{align}
Together with $C:= \max_{j=0,\ldots,\lopt} \frac{\# \MM_j}{\# \MM_{\lopt}}$, we obtain that
\begin{align}\label{eq:optimality:proof:four}
	\sum_{j=0}^{\ell-1} \# \MM_j = \sum_{j=0}^{\lopt} \# \MM_j + \sum_{j=\lopt}^{\ell-1} \# \MM_j \leq (\lopt C+ 1) \sum_{j=\lopt}^{\ell-1} \# \MM_j.
\end{align}
Combining~\eqref{eq:optimality:proof:one}--\eqref{eq:optimality:proof:four} as well as $ 0 < q:= \qlin^{1/s} < 1$, we exploit the 
geometric series and reveal that
\begin{align*}
\# \TT_\ell  - \# \TT_0 +1 &\stackrel{\eqref{eq:optimality:proof:three}}{\lesssim} \sum_{j=0}^{\ell-1} \# \MM_j 
\stackrel{\eqref{eq:optimality:proof:four}}{\lesssim} \sum_{j=\lopt}^{\ell-1} \# \MM_j
\stackrel{\eqref{eq:optimality:proof:one}}{\lesssim} \norm{\phi}{\Approx_{s}(\TT_{\lopt})}^{1/s} \sum_{j=\lopt}^{\ell-1} \eta_j^{-1/s} \\
&\stackrel{\eqref{eq:optimality:proof:two}}{\lesssim} \norm{\phi}{\Approx_{s}(\TT_{\lopt})}^{1/s} \eta_\ell^{-1/s} \sum_{j=\lopt}^{\ell-1} \qlin^{(\ell - j)/s} 
~\lesssim~ \norm{\phi}{\Approx_{s}(\TT_{\lopt})}^{1/s} \eta_\ell^{-1/s}.
\end{align*} 
Rearranging the terms, we conclude that $ \eta_\ell \leq \Copt  \big( \# \TT_\ell  - \# \TT_0 +1 \big)^{-s}$.
The constant $\Copt>0$ is given by
\begin{align*}
	\Copt = \Clin \, C_2 \, \norm{\phi}{\Approx_s(\TT_{\ell_4})} \, \Big( \frac{2 \Cmark \Cmesh C_1 (\ell_4 C + 1 )}{1-\qlin^{1/s}} \Big)^{s},
\end{align*}
where $C_1,C_2$ are the constants from Lemma~\ref{lemma:optimal}. This concludes the proof. \qed

\bigskip
\thanks{{\bf Acknowledgements.} 
	The author AB  acknowledges support by the EPSRC under grant EP/P013791/1.
	In addition, the authors AH and DP acknowledge support of the the Austria Science Fund (FWF) through the research project \emph{Optimal adaptivity for BEM and FEM-BEM coupling} (grant P27005), and the research program \emph{Taming complexity in partial differential systems} (grant SFB F65).
}

\bibliographystyle{alpha}
\bibliography{literature}

\newcommand{\etalchar}[1]{$^{#1}$}
\begin{thebibliography}{vWGBA15}

\bibitem[AFF{\etalchar{+}}13]{affkp}
Markus Aurada, Michael Feischl, Thomas F{\"u}hrer, Michael Karkulik, and Dirk
  Praetorius.
\newblock Efficiency and optimality of some weighted-residual error estimator
  for adaptive {2D} boundary element methods.
\newblock {\em J. Comput. Appl. Math.}, 13:305--332, 2013.

\bibitem[AFF{\etalchar{+}}17]{invest}
Markus Aurada, Michael Feischl, Thomas F{\"u}hrer, Michael Karkulik, Markus
  Melenk, and Dirk Praetorius.
\newblock Local inverse estimates for non-local boundary integral operators.
\newblock {\em Math. Comp.}, 86:2651--2686, 2017.

\bibitem[BDD04]{bdd}
Peter Binev, Wolfgang Dahmen, and Ronald DeVore.
\newblock Adaptive finite element methods with convergence rates.
\newblock {\em Numer. Math.}, 97(2):219--268, 2004.

\bibitem[BHP17]{helmholtz}
Alex Bespalov, Alexander Haberl, and Dirk Praetorius.
\newblock Adaptive {FEM} with coarse initial mesh guarantees optimal
  convergence rates for compactly perturbed elliptic problems.
\newblock {\em Comput. Methods Appl. Mech. Engrg.}, 317:318--340, 2017.

\bibitem[BS08]{brennerscott}
Susanne~C. Brenner and L.~Ridgway Scott.
\newblock {\em The mathematical theory of finite element methods}.
\newblock Springer, New York, third edition, 2008.

\bibitem[Car96]{cc96}
Carsten Carstensen.
\newblock Efficiency of a posteriori {BEM}-error estimates for first-kind
  integral equations on quasi-uniform meshes.
\newblock {\em Math. Comp.}, 65(213):69--84, 1996.

\bibitem[CFPP14]{axioms}
Carsten Carstensen, Michael Feischl, Marcus Page, and Dirk Praetorius.
\newblock Axioms of adaptivity.
\newblock {\em Comput. Math. Appl.}, 67(6):1195--1253, 2014.

\bibitem[CK83]{MR700400}
David~L. Colton and Rainer Kress.
\newblock {\em Integral equation methods in scattering theory}.
\newblock John Wiley \& Sons, New York, 1983.

\bibitem[CKNS08]{ckns}
J.~Manuel Cascon, Christian Kreuzer, Ricardo~H. Nochetto, and Kunibert~G.
  Siebert.
\newblock Quasi-optimal convergence rate for an adaptive finite element method.
\newblock {\em SIAM J. Numer. Anal.}, 46(5):2524--2550, 2008.

\bibitem[CMPS04]{cmps}
Carsten Carstensen, Matthias Maischak, Dirk Praetorius, and Ernst~P. Stephan.
\newblock Residual-based a posteriori error estimate for hypersingular equation
  on surfaces.
\newblock {\em Numer. Math.}, 97(3):397--425, 2004.

\bibitem[CMS01]{cms}
Carsten Carstensen, Matthias Maischak, and Ernst~P. Stephan.
\newblock A posteriori error estimate and {$h$}-adaptive algorithm on surfaces
  for {S}ymm's integral equation.
\newblock {\em Numer. Math.}, 90(2):197--213, 2001.

\bibitem[CS87]{cs87}
Martin Costabel and Ernst~P. Stephan.
\newblock On the convergence of collocation methods for boundary integral
  equations on polygons.
\newblock {\em Math. Comp.}, 49(180):461--478, 1987.

\bibitem[CS95]{cs95}
Carsten Carstensen and Ernst~P. Stephan.
\newblock A posteriori error estimates for boundary element methods.
\newblock {\em Math. Comp.}, 64(210):483--500, 1995.

\bibitem[CWGLS12]{MR2916382}
Simon~N. Chandler-Wilde, Ivan~G. Graham, Stephen Langdon, and Euan~A. Spence.
\newblock Numerical-asymptotic boundary integral methods in high-frequency
  acoustic scattering.
\newblock {\em Acta Numer.}, 21:89--305, 2012.

\bibitem[D{\"o}r96]{doerfler}
Willy D{\"o}rfler.
\newblock A convergent adaptive algorithm for {P}oisson's equation.
\newblock {\em SIAM J. Numer. Anal.}, 33(3):1106--1124, 1996.

\bibitem[FFK{\etalchar{+}}14]{ffkmp:part1}
Michael Feischl, Thomas F{\"u}hrer, Michael Karkulik, Jens~Markus Melenk, and
  Dirk Praetorius.
\newblock Quasi-optimal convergence rates for adaptive boundary element methods
  with data approximation, part {I}: weakly-singular integral equation.
\newblock {\em Calcolo}, 51(4):531--562, 2014.

\bibitem[FFK{\etalchar{+}}15]{ffkmp:part2}
Michael Feischl, Thomas F\"uhrer, Michael Karkulik, {Jens Markus} Melenk, and
  Dirk Praetorius.
\newblock Quasi-optimal convergence rates for adaptive boundary element methods
  with data approximation, part {II}: {H}ypersingular integral equation.
\newblock {\em Electron. Trans. Numer. Anal.}, 44:153--176, 2015.

\bibitem[FFP14]{ffp}
Michael Feischl, Thomas F{\"u}hrer, and Dirk Praetorius.
\newblock Adaptive {FEM} with optimal convergence rates for a certain class of
  nonsymmetric and possibly nonlinear problems.
\newblock {\em SIAM J. Numer. Anal.}, 52(2):601--625, 2014.

\bibitem[FGHP16]{fghp16}
Michael Feischl, Gregor Gantner, Alexander Haberl, and Dirk Praetorius.
\newblock Adaptive 2{D} {IGA} boundary element methods.
\newblock {\em Eng. Anal. Bound. Elem.}, 62:141--153, 2016.

\bibitem[FGHP17]{fghp17}
Michael Feischl, Gregor Gantner, Alexander Haberl, and Dirk Praetorius.
\newblock Optimal convergence for adaptive {IGA} boundary element methods for
  weakly-singular integral equations.
\newblock {\em Numer. Math.}, 136(1):147--182, 2017.

\bibitem[FGP15]{igabem}
Michael Feischl, Gregor Gantner, and Dirk Praetorius.
\newblock Reliable and efficient a posteriori error estimation for adaptive
  {IGA} boundary element methods for weakly-singular integral equations.
\newblock {\em Comput. Methods Appl. Mech. Engrg.}, 290:362--386, 2015.

\bibitem[FHPS18]{abem+solve}
Thomas F\"uhrer, Alexander Haberl, Dirk Praetorius, and Stefan Schimanko.
\newblock Adaptive {B}{E}{M} with inexact {P}{C}{G} solver yields almost
  optimal computational costs.
\newblock {\em Numer. Math.}, in print, 2018.

\bibitem[FKMP13]{fkmp}
Michael Feischl, Michael Karkulik, {J.\ Markus} Melenk, and Dirk Praetorius.
\newblock Quasi-optimal convergence rate for an adaptive boundary element
  method.
\newblock {\em SIAM J. Numer. Anal.}, 51:1327--1348, 2013.

\bibitem[Gan17]{gantnerphd}
Gregor Gantner.
\newblock {\em Optimal adaptivity for splines in finite and boundary element
  methods}.
\newblock PhD thesis, TU Wien, Institute for Analysis and Scientific Computing,
  Wien, 2017.

\bibitem[GBB{\etalchar{+}}15]{bempp2}
{Samuel P.} Groth, {Anthony J.} Baran, Timo Betcke, Stephan Havemann, and
  Wojciech \'Smigaj.
\newblock The boundary element method for light scattering by ice crystals and
  its implementation in {BEM}++.
\newblock {\em J. Quant. Spectrosc. Radiat. Transfer}, 167:40 -- 52, 2015.

\bibitem[Geo08]{Geo08}
Emmanuil~H. Georgoulis.
\newblock Inverse-type estimates on {$hp$}-finite element spaces and
  applications.
\newblock {\em Math. Comp.}, 77(261):201--219, 2008.

\bibitem[GHP17]{MR3723732}
Gregor Gantner, Daniel Haberlik, and Dirk Praetorius.
\newblock Adaptive {IGAFEM} with optimal convergence rates: hierarchical
  {B}-splines.
\newblock {\em Math. Models Methods Appl. Sci.}, 27(14):2631--2674, 2017.

\bibitem[GHS05]{ghs}
Ivan~G. Graham, Wolfgang Hackbusch, and Stefan~A. Sauter.
\newblock Finite elements on degenerate meshes: inverse-type inequalities and
  applications.
\newblock {\em IMA J. Numer. Anal.}, 25(2):379--407, 2005.

\bibitem[GM06]{MR2257114}
Ivan~G. Graham and William McLean.
\newblock Anisotropic mesh refinement: the conditioning of {G}alerkin boundary
  element matrices and simple preconditioners.
\newblock {\em SIAM J. Numer. Anal.}, 44(4):1487--1513, 2006.

\bibitem[GS18]{epsbuch}
Joachim Gwinner and Ernst~Peter Stephan.
\newblock {\em Advanced boundary element methods}.
\newblock Springer, Cham, 2018.

\bibitem[Hab18]{haberlphd}
Alexander Haberl.
\newblock {\em On adaptive FEM and for indefinite and nonlinear problems}.
\newblock PhD thesis, TU Wien, Institute for Analysis and Scientific Computing,
  Wien, 2018.

\bibitem[KPP13]{kpp}
Michael Karkulik, David Pavlicek, and Dirk Praetorius.
\newblock On 2{D} newest vertex bisection: optimality of mesh-closure and
  {$H^1$}-stability of {$L_2$}-projection.
\newblock {\em Constr. Approx.}, 38(2):213--234, 2013.

\bibitem[McL00]{mclean}
William McLean.
\newblock {\em Strongly elliptic systems and boundary integral equations}.
\newblock Cambridge University Press, Cambridge, 2000.

\bibitem[Mel12]{Mel12}
Jens~Markus Melenk.
\newblock Mapping properties of combined field {H}elmholtz boundary integral
  operators.
\newblock {\em SIAM J. Math. Anal.}, 44(4):2599--2636, 2012.

\bibitem[SBA{\etalchar{+}}15]{bempp}
Wojciech \'Smigaj, Timo Betcke, Simon Arridge, Joel Phillips, and Martin
  Schweiger.
\newblock Solving boundary integral problems with {BEM}++.
\newblock {\em ACM Trans. Math. Software}, 41(2):Art. 6, 40, 2015.

\bibitem[SS11]{sauterschwab}
Stefan~A. Sauter and Christoph Schwab.
\newblock {\em Boundary element methods}.
\newblock Springer, Berlin, 2011.

\bibitem[Ste07]{stevenson07}
Rob Stevenson.
\newblock Optimality of a standard adaptive finite element method.
\newblock {\em Found. Comput. Math.}, 7(2):245--269, 2007.

\bibitem[Ste08a]{steinbach}
Olaf Steinbach.
\newblock {\em Numerical approximation methods for elliptic boundary value
  problems}.
\newblock Springer, New York, 2008.

\bibitem[Ste08b]{stevenson08}
Rob Stevenson.
\newblock The completion of locally refined simplicial partitions created by
  bisection.
\newblock {\em Math. Comp.}, 77(261):227--241, 2008.

\bibitem[Ste13]{st_hypsing}
Olaf Steinbach.
\newblock Boundary integral equations for {H}elmholtz boundary value and
  transmission problems.
\newblock In {\em Direct and inverse problems in wave propagation and
  applications}, volume~14 of {\em Radon Ser. Comput. Appl. Math.}, pages
  253--292. De Gruyter, Berlin, 2013.

\bibitem[Tar07]{tartar}
Luc Tartar.
\newblock {\em An introduction to {S}obolev spaces and interpolation spaces}.
\newblock Springer, Berlin, 2007.

\bibitem[Tri83]{triebel83}
Hans Triebel.
\newblock {\em Theory of function spaces}.
\newblock Birkh\"auser, Basel, 1983.

\bibitem[Tri92]{triebel92}
Hans Triebel.
\newblock {\em Theory of function spaces. {II}}.
\newblock Birkh\"auser, Basel, 1992.

\bibitem[Tso13]{gantumur}
Gantumur Tsogtgerel.
\newblock Adaptive boundary element methods with convergence rates.
\newblock {\em Numer. Math.}, 124(3):471--516, 2013.

\bibitem[Tso17]{gant}
Gantumur Tsogtgerel.
\newblock Convergence rates of adaptive methods, {B}esov spaces, and multilevel
  approximation.
\newblock {\em Found. Comput. Math.}, 17:917--956, 2017.

\bibitem[vWGBA15]{bempp3}
Elwin van't Wout, Pierre G\'elat, Timo Betcke, and Simon Arridge.
\newblock A fast boundary element method for the scattering analysis of
  high-intensity focused ultrasound.
\newblock {\em J. Acoust. Soc. Am.}, 138(5):2726--2737, 2015.

\end{thebibliography}

\end{document}